\documentclass[a4paper,12pt,oneside]{article}

\usepackage[utf8]{inputenc}
\usepackage[english]{babel}
\usepackage{amsthm}
\usepackage{amssymb}
\usepackage{amsmath}
\usepackage{hyperref}
\usepackage{mathrsfs}

\usepackage{pstricks}

 \usepackage{tikz}

\usepackage{graphicx}
\graphicspath{{images/}}

\usepackage{longtable,geometry}
\usepackage{color}

\usepackage{hyperref}

\newtheoremstyle{note}{12pt}{12pt}{}{}{\bfseries}{.}{.5em}{}
\title{\LARGE\textbf{Newhouse Laminations}}
\author{Michael Benedicks, Marco Martens and Liviana Palmisano}

\makeatletter
\newtheorem{theo}{Theorem}
\newtheorem{prop}{Proposition}
\numberwithin{equation}{section}
\newtheorem{defin}{Definition}

\newtheorem{conj}[equation]{Conjecture}
\newtheorem{rem}{Remark}
\newtheorem{cor}{Corollary}

\numberwithin{equation}{section}

\newtheorem{lem}{Lemma}

\usepackage{babel}

\newcommand{\N}{{\mathbb N}}
\newcommand{\Z}{{\mathbb Z}}
\newcommand{\R}{{\mathbb R}}
\newcommand{\Q}{{\mathbb Q}}

\newcommand{\Cinf}{{{\mathcal C}^\infty}}
\newcommand{\Cd}{{{\mathcal C}^2}}
\newcommand{\Ct}{{{\mathcal C}^3}}
\newcommand{\Cq}{{{\mathcal C}^4}}

\newcommand{\Cuno}{{{\mathcal C}^1}}

\begin{document}
\maketitle
\author
\textcolor{blue}{}\global\long\def\sbr#1{\left[#1\right] }
\textcolor{blue}{}\global\long\def\cbr#1{\left\{  #1\right\}  }
\textcolor{blue}{}\global\long\def\rbr#1{\left(#1\right)}
\textcolor{blue}{}\global\long\def\ev#1{\mathbb{E}{#1}}
\textcolor{blue}{}\global\long\def\R{\mathbb{R}}
\textcolor{blue}{}\global\long\def\E{\mathbb{E}}
\textcolor{blue}{}\global\long\def\norm#1#2#3{\Vert#1\Vert_{#2}^{#3}}
\textcolor{blue}{}\global\long\def\pr#1{\mathbb{P}\rbr{#1}}
\textcolor{blue}{}\global\long\def\qq{\mathbb{Q}}
\textcolor{blue}{}\global\long\def\aa{\mathbb{A}}
\textcolor{blue}{}\global\long\def\ind#1{1_{#1}}
\textcolor{blue}{}\global\long\def\pp{\mathbb{P}}
\textcolor{blue}{}\global\long\def\cleq{\lesssim}
\textcolor{blue}{}\global\long\def\ceq{\eqsim}
\textcolor{blue}{}\global\long\def\Var#1{\text{Var}(#1)}
\textcolor{blue}{}\global\long\def\TDD#1{{\color{red}To\, Do(#1)}}
\textcolor{blue}{}\global\long\def\dd#1{\textnormal{d}#1}
\textcolor{blue}{}\global\long\def\eqdef{:=}
\textcolor{blue}{}\global\long\def\ddp#1#2{\left\langle #1,#2\right\rangle }
\textcolor{blue}{}\global\long\def\En{\mathcal{E}_{n}}
\textcolor{blue}{}\global\long\def\Z{\mathbb{Z}}
\textcolor{blue}{{} }

\textcolor{blue}{}\global\long\def\nC#1{\newconstant{#1}}
\textcolor{blue}{}\global\long\def\C#1{\useconstant{#1}}
\textcolor{blue}{}\global\long\def\nC#1{\newconstant{#1}\text{nC}_{#1}}
\textcolor{blue}{}\global\long\def\C#1{C_{#1}}
\textcolor{blue}{}\global\long\def\meas{\mathcal{M}}
\textcolor{blue}{}\global\long\def\cSpace{\mathcal{C}}
\textcolor{blue}{}\global\long\def\pspace{\mathcal{P}}

\begin{abstract}
Newhouse laminations occur in unfoldings of rank-one homoclinic tangencies. Namely, in these unfoldings, there exist codimension $2$ laminations of maps with infinitely many sinks which move simultaneously along the leaves. As consequence, in the space of real polynomial maps, there are examples of: 
\begin{itemize}
\item[-] H\'enon maps, in any dimension, with infinitely many sinks,
\item[-] quadratic H\'enon-like maps with infinitely many sinks and a period doubling attractor,
\item[-] quadratic H\'enon-like maps with infinitely many sinks and a strange attractor,
\item[-] non trivial analytic families of polynomial maps with infinitely many sinks.
\end{itemize}
\end{abstract}

\section{Introduction}
Systems describing nature have a certain form of stability, which makes it possible to observe them. The strongest form of stability is structural stability. Hyperbolic systems, which are known to be structurally stable, have been intensively studied and completely understood. The situation becomes much more complicated if a hyperbolic system is deformed until it ceases to be hyperbolic. In particular, homoclinic tangencies can appear. Unfolding of homoclinic tangencies are only very partially understood. 

In the case of dissipative rank one\footnote{ A rank one system is a system with only one unstable Lyapunov exponent.} systems, part of the dynamics of unfolding of homoclinic tangencies can be described by H\'enon-like maps. 
Three phenomena have been detected in unfoldings of dissipative rank one homoclinic tangencies. 
\begin{enumerate}
\item Newhouse phenomena (see \cite{Newhouse}): there are maps near homoclinic tangencies which have infinitely many sinks.
\item Strange attractors (see \cite{BC, MV}): there are maps with an attractor having an SRB measure. In particular, the invariant measure has positive Lyapunov exponent, see \cite{BY, BV}.
\item Period doubling Cantor attractors (see \cite{LM}): there are maps with a Cantor attractor having zero Lyapunov exponent. 
\end{enumerate}
These phenomena appear in unfoldings of homoclinic tangencies, i.e near the boundary of hyperbolicity where the system ceases to be structurally stable. 
Although these phenomena are intrinsically related to instability, they all have their own form of stability. In other words, the set of parameters where these phenomena occur is not open, but it is large and has structure. In more details, the strange attractors appear for a positive Lebesgue measure set of parameters and the maps with a period doubling Cantor attractor form a codimension one manifold.

We study here the stability of the Newhouse phenomenon. We prove that there are codimension  $2$ laminations\footnote{We recall that a lamination is a Hausdorff space $X$ which is equipped with a covering $\left\{U_i\right\}$ by open subsets and coordinate charts $\phi_i:U_i\to T_i\times D_i$, where $D_i$ is homeomorphic to a domain in Euclidean space and where  $T_i$ is some topological space. The sheets are the subsets of $X$ which are sent locally by the mappings $\phi_i$ to the Euclidean factors and the transition mappings $\phi_{i,j}:\phi_j\left(U_i\cap U_j\right)\to \phi_i\left(U_i\cap U_j\right)$ are homeomorphisms which preserve the sheets. The lamination of Theorem B has only one chart. It is homeomorphic to $\mathbb{R}\setminus \mathbb{Q}\times \mathbb{R}^d$.} of maps with infinitely many sinks. The sinks moves smoothly along the leaves of the lamination.

More specifically, we consider $\Cinf$ local diffeomorphisms on a $\Cinf$ manifold of any dimension with a strong homoclinic tangency, see Definition \ref{stronghomtang}. The term {\it strong homoclinic tangency} refers to the fact that the initial map has an homoclinic tangency and a transversal homoclinic intersection between the stable and unstable manifold of a rank one saddle point. Moreover if $\mu$ is the unstable eigenvalue and $\lambda_1$ is the dominant stable eigenvalue, then a map with a strong homoclinic tangency satisfies also:
\begin{equation}\label{ourcond}
|\lambda_1||\mu|^3<1.
\end{equation}
 Given a map with a strong homoclinic tangency $f$, we consider finite dimensional unfoldings of $f$, see Definition \ref{unfolding}. The collection of these unfoldings is the complement of finitely many manifolds in the space of all families through $f$. As a preparation for our main result, namely Theorem B, we prove also the following previously known theorem.

\vskip .2 cm

\paragraph{\bf Theorem A.} 
Let $F:\mathcal P\times M \to M$ be a $2$ dimensional unfolding of a map $f$ with a strong homoclinic tangency, where $\mathcal P$ is a two dimensional parameter space. Then there exists a set $NH\subset\mathcal P$ such that
\begin{itemize}
\item[-] every map in $NH$ has infinitely many sinks,
\item[-] $NH$ is homeomorphic to $\R\setminus\Q$.
\end{itemize}
\vskip .2 cm
\noindent
The conclusions of this theorem hold in particular for the following families.
\begin{itemize}
\item[-] The two-dimensional real H\'enon family $F_{a,b}:\mathbb{R}^2\times\mathbb{R}^2\to\mathbb{R}^2$, 
$$
F_{a,b}\left(\begin{matrix}
x\\y
\end{matrix}
\right)=\left(\begin{matrix}
a-x^2-by\\x
\end{matrix}
\right)
$$
It was already shown in \cite{BP, Da, GST, Ro,VaS} that there are real H\'enon maps with infinitely many sinks. 
\item[-] The real H\'enon family of maps of $\R^n$.
It was already shown in \cite{GST, PV} that there are real H\'enon maps of $\R^n$ with infinitely many sinks. 
\end{itemize}

 Our main result describes the stability aspects of the Newhouse phenomenon. There are codimension $2$ laminations of maps with infinitely many sinks. All the sinks move smoothly and simultaneously along the leaves of the laminations. These laminations appear in higher dimensional families which extend two dimensional unfoldings, i.e. obtained by adding any number of parameters. This answers a question in \cite{CCH}. The existence of the laminations is a consequence of only infinitesimal properties of the saddle, namely $|\lambda_1||\mu|^3<1$ and $\partial\mu/\partial t\neq 0$, see Remark \ref{anglebound}. In particular the laminations occur among systems which are not necessarily globally area contracting. The transversal structure of the lamination is remarkable regular both in topological and geometrical sense, see Figure \ref{Fig6} and Figure \ref{Fig7}. 
\vskip .2 cm
\paragraph{\bf Theorem B.} 
Let $M$, $\mathcal P$ and $\mathcal T$ be $\Cinf$ manifolds and $F:\left(\mathcal P\times\mathcal T\right)\times M\to M$ a $\Cinf$ family with $\text{dim}(\mathcal P)=2$ and $\text{dim}(\mathcal T)\geq 1$. If $F_0:\left(\mathcal P\times\left\{\tau_0\right\}\right)\times M\to M$ is an unfolding of a map $f_{\tau_0}$ with a strong homoclinic tangency, then the set of maps with infinitely many sinks, $NH_F\subset \mathcal P\times \mathcal T$, satisfies the following:
\begin{itemize}
\item[-] $NH_F$ contains a codimension $2$ lamination $L_F$,
\item[-] $L_F$ is homeomorphic to $\left(\mathbb R\setminus\mathbb Q\right)\times\mathbb R^{\dim (\mathcal T)}$,
\item[-] the leaves of $L_F$ are $\Cuno$ codimension $2$ manifolds,
\item[-] infinitely many sinks persist along each leave of the lamination.
\end{itemize}
\vskip .2 cm

An application of the main theorems to two and higher dimensional H\' enon dynamics is the following. In the space of polynomial maps, there are codimension $2$ laminations of maps with infinitely many sinks. The lamination intersects the H\' enon family transversally.
\vskip .2 cm
\paragraph{\bf Theorem C.}
The real H\' enon family contains a set $NH$, homeomorphic to $\mathbb{R}\setminus \mathbb{Q}$, of maps with infinitely many sinks.  Moreover the space $\text{Poly}_d({\mathbb{R}^n})$ of real polynomials of $\mathbb{R}^n$ of degree at most $d$, with $d\ge 2$, contains a codimension $2$ lamination of maps with infinitely many sinks. The lamination is homeomorphic to $\left(\mathbb R\setminus\mathbb Q\right)\times\mathbb R^{\text{D}-2}$ where $\text{D}$ is the dimension of $\text{Poly}_d({\mathbb{R}^n})$ and the leaves of the lamination are $\Cuno$ smooth. The sinks persist along each leave of the lamination. Moreover the leaves of the laminations in $\text{Poly}_d({\mathbb{R}^2})$ are real-analytic.
\vskip .2 cm

Observe that the laminations mentioned in Theorem C are non trivial. Consider the two-dimensional H\' enon family $F_{a,b}$. One can perturb this family by adding polynomial terms to obtain a new family $\tilde F_{a,b}$. According to \cite{HMT}, one can adjust the polynomial perturbation such that strongly dissipative H\' enon maps at the boundary of chaos of $F$ are never topological conjugate to strongly dissipative H\' enon maps at the boundary of chaos of $\tilde F$. The two families are topologically different. Nevertheless, Theorem C says that the Newhouse points with their topological characteristics, persist.

The fact that there are infinitely many sinks which persist along codimension $2$ leaves, allows us to construct examples of maps with intricate attractor coexistence. In particular, in the three dimensional family of quadratic H\'enon-like maps, 
\begin{equation}
F_{a,b,\tau}\left(\begin{matrix}
x\\y
\end{matrix}
\right)=\left(\begin{matrix}
a-x^2-by+\tau y^2
\\
x
\end{matrix}
\right),
\end{equation}
 there are maps with infinitely many sinks and a period doubling Cantor attractor\footnote{See Definition \ref{cantorA}} and there are maps with infinitely many sinks and a strange attractor\footnote{See Definition \ref{strange}}. 
\vskip .2 cm
\paragraph{\bf Theorem D.} There are uncountable many quadratic H\'enon-like maps with infinitely many sinks and a period doubling Cantor attractor.
\vskip .2 cm
\paragraph{\bf Theorem E.} The set of quadratic H\'enon-like maps with infinitely many sinks and a strange attractor has Hausdorff dimension at least $1$.
\vskip .2 cm
The sinks and the Cantor attractor move smoothly along codimension $3$ sub-manifolds in the space of polynomial maps.
\paragraph{\bf Theorem F.}
The space $\text{Poly}_d({\mathbb{R}^2})$, $d\ge 2$,  of real polynomials of $\mathbb{R}^2$ of degree at most $d$ contains a codimension $3$ lamination of maps with infinitely many sinks and a period doubling Cantor attractor. The lamination is homeomorphic to $\left(\mathbb R\setminus\mathbb Q\right)\times\mathbb R^{\text{D}-3}$ where $\text{D}$ is the dimension of $\text{Poly}_d({\mathbb{R}^2})$ and the leaves of the lamination are real-analytic. The sinks and the period doubling Cantor attractor persist along the leaves. 
\vskip .2 cm
The next theorem should be considered in the context of the Palis Conjecture, \cite{Palis}.  The main result in \cite{Berger} states that there exists a Baire residual set of smooth $d$-dimensional families such that each map in such a family has infinitely many sinks. Theorem G states that there are analytic families of arbitrary dimension of polynomial maps such that each map in the family has infinitely many sinks.
\vskip .2 cm
\paragraph{\bf Theorem G.}
Every $d+2$-dimensional unfolding, $d\geq 1$, of a map with a strong homoclinic tangency contains smooth $d$-dimensional families of maps where each map has infinitely many sinks.
In particular, there are non trivial $d$-dimensional analytic families of polynomial H\'enon-like maps in which every map has infinitely many sinks.
\vskip .2 cm
In order to stress the fact that the families in Theorem G are non trivial we would like to emphasize the following. Take a one-parameter family as in Theorem G. The map at the beginning of the curve has infinitely many sinks and a period doubling Cantor attractor and the map at the end has still infinitely many sinks, but the Cantor attractor is replaced by a strange attractor. All maps in the middle have infinitely many sinks. 
\bigskip

%
Maps with the Newhouse phenomenon have been constructed in different contexts. In particular, there are many examples of Baire\footnote{A Baire set in a locally compact Hausdorff space $X$ is a countable intersection of open and dense subsets.} sets of maps with infinitely many sinks in the space of systems, see \cite{ABC, Berger, BDP, DR, GST1, KS, Newhouse,    PV,  Ures}.
In most studies the method behind these results is based on persistency of tangencies where the thickness (in the sense of Newhouse) of the stable and the unstable Cantor sets plays a crucial role.
 
In order to study the stability of maps with infinitely many sinks we needed to introduce a different method. This method does not rely on the persistency of tangency and the thickness condition is replaced by condition (\ref{ourcond}) which depends only on the eigenvalues at a saddle point. {Notice that in the H\'enon family the set of maps satisfying the thickness condition and the one satisfying (\ref{ourcond}) intersect but they are not contained in each other. }

The method is inductive, starting at a map with a sink and a homoclinic tangency. The sink will persist in a neighborhood of the map. Similar as in the classical Newhouse construction one uses the homoclinic tangency to create, by small changes of the parameter, a new homoclinic tangency and another sink. There are two essential differences of our method and the method due to Newhouse, see \cite{Newhouse}. This method uses perturbations in the space of systems. Our method is designed to work within a given family. The creation of the new sink and the new tangency occurs by parameter adjustment. 

The main difference is in the creation of the new secondary  tangency. The creation of the secondary tangency in the classical Newhouse method uses the persistency of tangencies, i.e. in a neighborhood of the starting map, every map has somewhere a tangency. This secondary tangency varies very discontinuously throughout the neighborhood. This discontinuity of the secondary tangency makes analysis very hard. Our method is inspired by \cite{BC, BP} and uses {\it critical} points and {\it binding} points and does not at all not rely on persistency of tangencies. 

\bigskip

{\it Outline of the method.} A critical point of a diffeomorphism is usually identified with a homoclinic tangency, see \cite{PRH}. A crucial aspect of homoclinic tangencies is that nearby, there are domains whose first return map are H\'enon-like, in the sense of \cite{PT}. There is precise analysis available to locate the sinks in such H\'enon-like maps. This analysis is summarized in Proposition \ref{sink}.

Iterations of the local unstable manifold at the tangency will accumulate at the unstable manifold of the original saddle point. In particular, these iterations will create a package of nearly parallel pieces of local unstable manifold near the tangency. One of such folded local unstable manifolds is  illustrated in Figure \ref{Fig4a}. The orbit of such a piece of folded local unstable manifold is controlled by the orbit of the original local unstable manifold at the tangency. They are created at binding points, similar to those in \cite{BC}, denoted here by $c'$. A next passage near the saddle point will create a piece of local unstable manifold with a large parameter speed. The folded local unstable manifold used to create the secondary tangency is illustrated in Figure \ref{Fig4b}. A precise analysis is available for the parameter dependance of these local unstable manifolds. This analysis is summarized in Proposition \ref{speed}.   

The creation of the secondary tangency is directly inspired by 
\cite{BC, BP}. The main idea is to use  the hyperbolicity of the saddle point to create sufficiently strong parameter dependance of the newly created folded local unstable manifolds. This parameter dependance allows to create a secondary tangency, see Proposition \ref{newtangency}.

The secondary tangencies occur along curves in parameter space, Proposition \ref{angle}. This is the most delicate aspect of the analysis. In particular, the angle between the curves of secondary tangencies and the curves of super attracting sinks is non-zero but tends exponentially to zero in the period of the sink. The key in the construction of the lamination is this delicate transversality which is responsible for its existence.

As all sinks, the sinks under consideration persist in open sets of parameters.  The consequence of the transversality is that these open sets are not small balls but are strongly elongated and contain the leaves of the lamination, see (\ref{alongated}).

\bigskip

As a final remark, indeed the Newhouse laminations constructed here are a small part of the set of all parameters for which the corresponding map has infinitely many sinks. However, in the construction we explore the effect of only one transversal homoclinic intersection. One can use all transversal homoclinic intersections and one expects the following.
\begin{conj}
Every $d$-dimensional unfolding of a map with a strong homoclinic tangency contains a codimension $2$ Newhouse lamination with Hausdorff dimension $d$. 
\end{conj}
Although one anticipates Newhouse laminations with Hausdorff dimension $d$ one should not expect codimension one Newhouse laminations.

\paragraph{Acknowledgements}
The first author was supported by the Swedish Research Council Grant 2016-05482 and partially by the NSF grant 1600554 and the IMS at Stony Brook University. The second author was partially supported by the NSF grant 1600554. The third author was supported by the Trygger foundation and partially by the NSF grant 1600503. The authors would like to thank the institutes where the paper was realized: KTH, the Mittag-Leffler Institute, Stony Brook University and Uppsala University.  

\section{Preliminaries}\label{section:preliminaries} The following well-known linearization result is due to Sternberg, see \cite{S}.
\begin{theo}\label{Ctlinearization}
Given $\left(\lambda_0,\lambda_1,\dots,\lambda_{m-1}\right)\in\R^{m}$, with $\lambda_i\neq\lambda_j$ for $i\neq j$, there exists $N\left(\lambda_0,\lambda_1,\dots,\lambda_{m-1}\right)\in\N$ such that the following holds. 
Let $M$ be a $m$ dimensional $\Cinf$ manifold and let $f:M\to M$ be a diffeomorphism with a rank one saddle point $p\in M$, with unstable eigenvalue $|\lambda_0|>1$ and stable eigenvalues $\lambda=(\lambda_1,\lambda_2,\dots,\lambda_{m-1})$. If for all $j=0,\dots,m-1$,
\begin{equation}\label{nonresonance}
\lambda_j\neq\prod_{i\neq j}\lambda_i^{k_i}
\end{equation}
for $k=\left(k_0,\dots,k_{m-1}\right)\in\N^{m}$ with $2\leq |k|=k_0+\dots+ k_{m-1}\leq N$, with $N$ large enough, then $f$ is $\Cq$ linearizable.
\end{theo}
\begin{defin}
Let $M$ be an $m$-dimensional $\Cinf$ manifold and $f:M\to M$ a diffeomorphism with a rank one saddle point $p\in M$. We say that $p$ satisfies the $\Cq$ non-resonance condition if (\ref{nonresonance}) holds.
\end{defin}

\begin{theo}\label{familydependence}
Let $M$ be a $m$ dimensional $\Cinf$ manifold and $f:M\to M$ a diffeomorphism with a rank one saddle point $p\in M$ which satisfies the $\Cq$ non-resonance condition. Let $0\in \mathcal P\subset\R^{n}$ and $F:\mathcal P\times M\to M$ a $\Cinf$ family with $F_0=f$. Then, there exists a neighborhood $U$ of $p$ and a neighborhood $V$ of $0$ such that, for every $t\in V$, $F_t$ has a saddle point $p_t\in U$ satisfying the $\Cq$ non-resonance condition. Moreover $p_t$ is $\Cq$ linearizable in the neighborhood $U$ and the linearization depends $\Cq$ on the parameters. 
\end{theo}
The proofs of Theorem \ref{Ctlinearization} and Theorem \ref{familydependence} can be found in \cite{BrKo, IlaYak}. The following lemma is a direct consequence of Theorem \ref{Ctlinearization}.
\begin{lem}
Let $M$ be an $m$-dimensional $\Cinf$ manifold and $f:M\to M$ a diffeomorphism with a rank one saddle point $p\in M$ satisfying the $\Cq$ non-resonance condition with $|\lambda_1|>\max_{2\leq i\leq m-1}|\lambda_i|$. If $q\in W^u_p$, then
$$
E_q=\left\{v\in T_qM | \lim_{n\to\infty}Df_q^{-n}(v)\lambda^n_1 \text{ exists }\right\}
$$
is a two-dimensional vector space with $T_q W_p^u\subset E_q$.
\end{lem}
In the following we define a map with a strong homoclinic tangency by listing several conditions.  To summarize, a map with a strong homoclinic tangency has a non degenerate homoclinic tangency $q_1$ and a transversal homoclinic intersection $q_2$. Moreover $q_1$ and $q_2$ are in general position. Namely, any tangent vector points in the direction of the strongest stable eigenvalue. This condition is redundant when the manifold has dimension $2$. Conditions $(f8)$, $(f9)$ and $(f10)$ ensure that $q_1$ and the unstable local manifolds of $q_1$ and $q_2$ accumulate on the leg of the unstable manifold of the saddle point containing the transversal homoclinic intersection $q_2$. These conditions are redundant if $\mu<-1$. To verify this, see Figure \ref{Fig1}. Moreover to implement the construction, a condition on the eigenvalues is also required, see $(f2)$. In the following the reader should keep Figure \ref{Fig1} in mind.
\begin{defin}\label{stronghomtang}
Let $M$ be an $m$-dimensional $\Cinf$ manifold and $f:M\to M$ a local diffeomorphism satisfying the following conditions:
\begin{itemize}
\item[$(f1)$] $f$ has a rank one saddle point $p\in M$, with unstable eigenvalue $|\mu|>1$ and distinct stable eigenvalues $\lambda=(\lambda_1,\lambda_2,\dots,\lambda_{m-1})$, where $\lambda_1$ is the largest one, namely $$|\lambda_1|>\max_{2\leq i\leq m-1}|\lambda_i|,$$
\item[$(f2)$] $|\lambda_1||\mu|^3<1$,
\item[$(f3)$] $p$ satisfies the $\Cq$ non-resonance condition,
\item[$(f4)$] $f$ has a non degenerate homoclinic tangency,  $q_1\in W^u(p)\cap W^s(p)$ in general position, namely $$\lim_{n\to\infty}\frac{1}{n}\log d(f^n(q_1),p)=\log|\lambda_1|,$$
\item[$(f5)$] the direction $0\neq B\in T_{q_1}W^u(p)$ is in general position, namely
$$
\lim_{n\to\infty}\frac{1}{n}\log |Df^n_{q_1}(B)|=\log|\lambda_1|,
$$
and 
$$
E_{q_1}\pitchfork W^s_{q_1}(p),
$$
\item[$(f6)$] $f$ has a transversal homoclinic intersection,  $q_2\in W^u(p)\pitchfork W^s(p)$ in general position, namely $$\lim_{n\to\infty}\frac{1}{n}\log d(f^n(q_2),p)=\log|\lambda_1|,$$
\item[$(f7)$] the direction of $0\neq v\in E_{q_2}\cap T_{q_2}W^s(q_2)$ is in general position, namely
$$
\lim_{n\to\infty}\frac{1}{n}\log |Df^n_{q_2}(v)|=\log|\lambda_1|,
$$
and \footnote{Observe that this is always verified because $q_2$ is a transversal intersection.}
$$
E_{q_2}\pitchfork W^s_{q_2}(p),
$$
\item[$(f8)$] let $[p,q_2]^u\subset W^u(p)$ be the arc connecting $p$ to $q_2$, then there exist arcs  $W^u_{\text{\rm loc},n}(q_2)=[q_2, u_n]^u\subset W^u(q_2)$ such that $[p,q_2]^u\cap [q_2,u_n]^u=\left\{q_2\right\}$ and 
$$
\lim_{n\to\infty}f^n\left(W^u_{\text{\rm loc},n}(q_2)\right)=[p,q_2]^u,
$$
\item[$(f9)$] there exist neighborhoods $W^u_{\text{\rm loc},n}(q_1)\subset W^u(q_1)$ such that 
$$
\lim_{n\to\infty}f^n\left(W^u_{\text{\rm loc},n}(q_1)\right)=[p,q_2]^u,
$$
\item[$(f10)$] there exists $N\in\N$ such that 
$$
f^{-N}(q_1)\in [p,q_2]^u.
$$
\end{itemize}
A map $f$ with these properties is called a map with a \emph{strong homoclinic tangency}, see Figure \ref{Fig1}.
\end{defin}
\begin{figure}[h]
\centering
\includegraphics[width=0.6\textwidth]{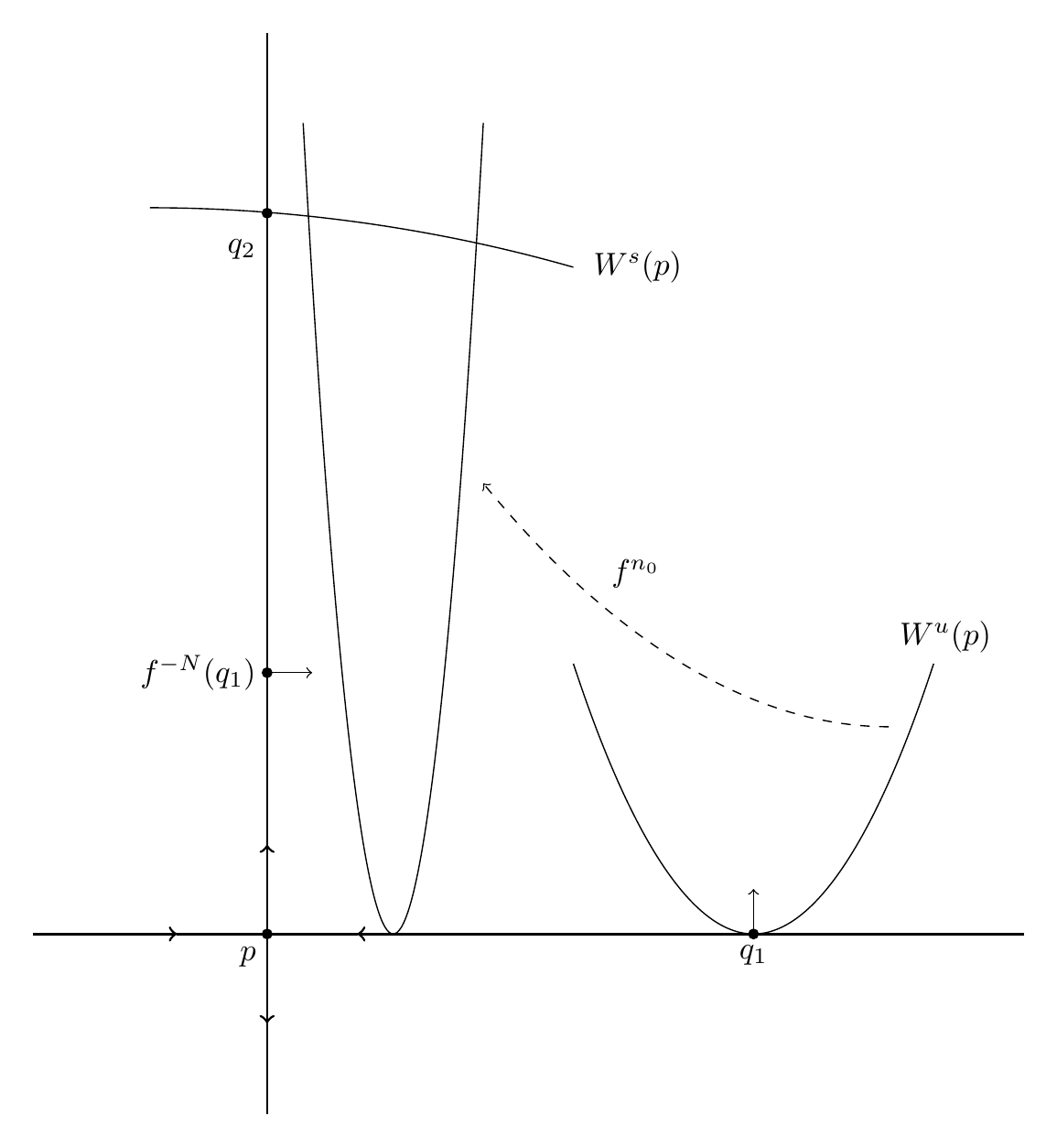}
\caption{Strong homoclinic tangency}
\label{Fig1}
\end{figure}

\begin{rem} The conditions defining a map with a strong homoclinic tangency are natural, except $(f2)$. This condition requires that the contraction at the saddle is strong enough. It plays a crucial role in many fundamental places such as the construction of the sink, (see Proposition \ref{sink}), and the transversality (see Proposition \ref{angle}).
\end{rem}

\begin{rem}
 Observe that all conditions are open in the space of maps with an homoclinic and transversal tangency. Also, except for $(f2)$, all conditions are dense.
\end{rem}

\begin{rem} If the unstable eigenvalue is negative, $\mu<-1$, then $(f8)$, $(f9)$, and $(f10)$ are redundant.
\end{rem}

\begin{rem} If \rm{dim}$(M)=2$ and the unstable eigenvalue is negative, $\mu<-1$, then $(f5)$, $(f6)$, $(f7)$, $(f8)$, $(f9)$, and $(f10)$ are redundant. The condition $(f4)$ reduces to the map having a non degenerate homoclinic tangency.
\end{rem} 
Following \cite{PT}, we construct now a family of unfolding of a map $f$ with a strong homoclinic tangency.
Let  $\mathcal P=[-r,r]\times [-r,r]$ with $r>0$.
Given a map $f$ with a strong homoclinic tangency, we consider a $\Cinf$ family $F:\mathcal P\times M\to M$ through $f$ with the following properties:
\begin{itemize}
\item[$(F1)$] $F_{0,0}=f$,
\item[$(F2)$] $F_{t, a}$ has a saddle point $p(t, a)$ with unstable eigenvalue $|\mu(t, a)|>1$, with largest stable eigenvalue $\lambda_1(t,a)$, and
$$
\frac{\partial \mu}{\partial t}\ne 0,
$$
\item[$(F3)$] let 
$\mu_{\text{max}}=\max_{(t,a)}|\mu(t,a)|$,
  $\lambda_{\text{max}}=\max_{(t,a)}|\lambda_1(t,a)|$ and assume $$\lambda_{\text{max}}\mu_{\text{max}}^3<1,$$
\item[$(F4)$] there exists a $\Cd$ function $[-r,r]\ni t\mapsto q_1(t)\in W^u(p(t,0))\cap W^s(p(t,0))$ such that $q_1(t)$ is a non degenerate homoclinic tangency and it is in general position, namely 
$$\lim_n\frac{1}{n}\log d(F_{t,0}^n(q_1(t)),p(t,0))=\log |\lambda_1(t,0)|,$$
\item[$(F5)$] the direction $0\neq B\in T_{q_1(t)}W^u(p(t,0))$ is in general position, namely
$$
\lim_n\frac{1}{n}\log |DF_{t,0}^n(B)|=\log |\lambda_1(t,0)|.
$$
\end{itemize}
According to Theorem \ref{familydependence} we may make a change of coordinates to ensure  that the family $F$ is $\Ct$ and for all $(t,a)\in [-r_0,r_0]\times [-r_0,r_0]$ with $0<r_0<r$ and by rescaling we can assume that $F_{t,a}$ is linear on the ball $[-2,2]^m$, namely 
\begin{equation}
\label{Flinear}
F_{t,a}=\left(\begin{matrix}
\lambda_1(t,a)&0&\dots&0&0\\
0&\vdots&\vdots&\vdots&\vdots\\
0&0&\dots&\lambda_{m-1}(t,a)&0\\
0&0&\dots&0&\mu(t,a)\\
\end{matrix}\right).
\end{equation}
Observe that, by $(F3)$ and by continuity, for $t_0$ small enough, \begin{equation}\label{thetacond0}
0<\frac{\log\mu_{\text{max}}}{\log{\frac{1}{\lambda_{\text{max}}}}}<\frac{3}{2}\frac{\log\mu_{\text{min}}}{\log{\frac{1}{\lambda_{\text{min}}}}}<\frac{1}{2},
\end{equation} where $\mu_{\text{min}}=\min_{(t,a)}|\mu(t,a)|$ and $\lambda_{\text{min}}=\min_{(t,a)}|\lambda_1(t,a)|$. Moreover the saddle point $p(t,a)=(0,0)$ and the local stable and unstable manifolds satisfy:
\begin{itemize}
\item[-] $W^s_{\text{loc}}(0)=[-2,2]^{m-1}\times \left\{0\right\}$,
\item[-] $W^u_{\text{loc}}(0)=\left\{0\right\}\times [-2,2]$ ,
\item[-] $q_1(t)\in [-1,1]^{m-1}\times \left\{0\right\}\subset W^s_{\text{loc}}(0)$ and it has first coordinate equal to $1$,
\item[-] $q_2(t,a)\in \left\{0\right\}\times \left(\frac{1}{\mu},1\right)\subset W^u_{\text{loc}}(0)$,
\item[-] there exists $N$ such that $f^{N}(q_3(t))=q_1(t)$ where $q_3(t)=(0,1)$,
\item[-]$Df^N_{q_3}(e_1)\notin T_{q_1}W^s(0)
$
and points in the positive $y$ direction,
\item[-] the direction $B=T_{q_1}W^u(0)$ has a non zero first coordinate.  
\end{itemize}

\bigskip

In the next lemma we prove that $q_3$ is contained in a curve of points whose vertical expanding tangent vectors are mapped by $DF^{N}$ to horizontal contractive ones. Let $(x,y)$ be in a neighborhood of $q_3$ and consider the point $$(X_{t,a}(x,y),Y_{t,a}(x,y))=F^N_{t,a}(x,y).$$
The following lemma says that $F^N_{t,a}(x,y)$ produces an unfolding in the sense of \cite{PT}. The formal definition of unfolding is given in Definition \ref{unfolding} below.
\begin{lem}\label{functionc}

There exist $x_0, a_0>0$, a $\Cd$ function $c:[-x_0,x_0]^{m-1}\times [-t_0,t_0]\times  [-a_0,a_0]\to\R $ and a positive constant $Q$ such that 
\begin{equation}\label{partialYpartialt}
\frac{\partial Y_{t,a}}{\partial y}\left(x, c(x,t,a)\right)=0,
\end{equation}
and 
$$\frac{\partial^2 Y_{t,a}}{\partial y^2}\left(x, c(x,t,a)\right)\geq Q.$$
Moreover 
\begin{eqnarray}\label{gamma}
|c(x,t,a)-c(0,t,a)|=O\left(|x|\right).
\end{eqnarray}
\end{lem}
\begin{proof}
Let $\Phi: [-\frac{1}{2},\frac{3}{2}]\times\left( [-1,1]^{m-1}\times [-t_0,t_0]\times  [-a_0,a_0]\right)\to\R$ be defined by 
$$\Phi(y, x, t,a)=\frac{\partial Y_{t,a}}{\partial y}\left(x, y\right).$$ Observe that $\Phi$ is $\Cd$. Let $q_3(t)=F^{-N}_{t,0}(q_1(t))$. Because $q_1(t)$ is an homoclinic tangency, see $(F4)$, we have $$\Phi (q_3(t),0, t,0)=0,$$ and because $q_1(t)$ is a non degenerate tangency, we get $$\frac{\partial\Phi}{\partial y}(q_3(t),0, t,0)>0.$$ For every $t\in[-t_0,t_0]$ there exist, by the implicit function theorem, $\epsilon >0$ and a unique $\Cd$ function $c:[-\epsilon ,\epsilon]^{m-1}\times [-t_0-\epsilon,t_0+\epsilon]\times  [-\epsilon,\epsilon]\to\R $ locally satisfying the requirements of the lemma. These local functions extend to global ones because of the compactness of the interval $[-t_0,t_0]$ and the local uniqueness.
\end{proof}
\begin{rem}\label{c0=1}
Without lose of generality, by a smooth change of coordinates in the $y$ direction, we may assume that $c(0,t,a)=1$.
\end{rem}
\begin{defin}\label{criticalpoints}
Let $\left(t,a\right)\in [-t_0,t_0]\times  [-a_0,a_0]$. We call the point $$c_{t,a}=\left(0, c(0,t,a)\right),$$ the {primary critical point} and 
 $$z_{t,a}=F_{t,a}^N\left(c_{t,a}\right)=(z_x(t,a),z_y(t,a)),$$ the {primary critical value} of $F_{t,a}$.
\end{defin}
Observe that, near the saddle point, vertical vectors are expanding and horizontal ones are contracting. The critical points are defined to have the property that the expanding vertical vector is sent to the contracting horizontal one under $DF^N$. Let us briefly discuss the concept of critical points for dissipative maps in order to compare our definition with the ones previously used. 

Let us start by recalling that critical points are fundamental in the study of one-dimensional dynamics and they are easy to detect: those are the points where the map is not locally a diffeomorphism. This definition has no meaning for diffeomorphism of higher dimensional manifolds. The formal definition in this setting is given in \cite{PRH}, where the authors define the critical points as homoclinic tangencies, i.e. a tangency between the stable and unstable manifold of a saddle point. One can also, which is the basis of the corresponding construction in \cite{BC}, define a critical point as a tangency between a local stable manifold and an unstable manifold of a periodic point.

Homoclinic tangencies play a crucial topological role in general. However, they are difficult to detect. That is why in varies studies, starting with \cite{BC}, there are notions of approximate critical points which share with homoclinic tangencies the property that expanding vectors are mapped into contracting ones.  In particular, in \cite{BC}, critical points are rather tangencies between the unstable manifold of the saddle point and an approximate local stable manifold, not necessarily of the saddle point.
In our situation, similarly the tangent vector of the unstable manifold at the critical point $c_{t,a}$ is mapped into the contractive horizontal vector at the critical value $z_{t,a}$.

The crucial fact, which comes from the fundamental property of the approximate critical points (and homoclinic tangencies), is that  in a neighborhood, the return map is an H\'enon-like map. The local 
H\'enon behavior is what  is important and allows the analysis. From a technical point of view, the notion of critical point itself is less important; the local H\'enon behavior is all what is needed.

Another instance where critical points occur but play only a secondary role is in the context of strongly dissipative H\'enon maps at the boundary of chaos. These maps have a period doubling Cantor attractor, see Definition \ref{cantorA} and \cite{CLM}. The Cantor attractors are studied using renormalization zooming in to the so-called {\it tip} of the Cantor set as in \cite{CLM}. Indeed, return maps to neighborhoods of the tip are 
H\'enon-like maps. A posteriori one shows that the stable manifold at the tip is tangent to the direction of the neutral Lyapunov exponent. The tip plays the role of a critical point. However, this fact did not play any role during the renormalization analysis. This is another instance where, from a technical point of view, the notion of critical point is not that important. 

%
%
%
%
%
%



\begin{defin}\label{unfolding}
A family $F_{t,a}$ is called an {unfolding} of $f$ if it can be reparametrized such that
\begin{itemize}
\item[$(P1)$] $z_y(t,0)=0$,
\item[$(P2)$] $\frac{\partial z_y(t,0)}{\partial a}\neq 0.$
\end{itemize}
\end{defin}
\begin{rem}\label{highzya}
Without lose of generality by a suitable coordinate change we may assume that if $F$ is an unfolding then  $z_y(t,a)=a$, the primary critical value is  at height $a$ and the primary critical point $c(t,a)=(0,1)$, see Figure \ref{Fig2}. 
\end{rem}
\begin{rem}
A generic $2$ dimensional family through $f$ can locally be re-parametrized to become an unfolding.
\end{rem}
\begin{figure}[h]
\centering
\includegraphics[width=.9\textwidth]{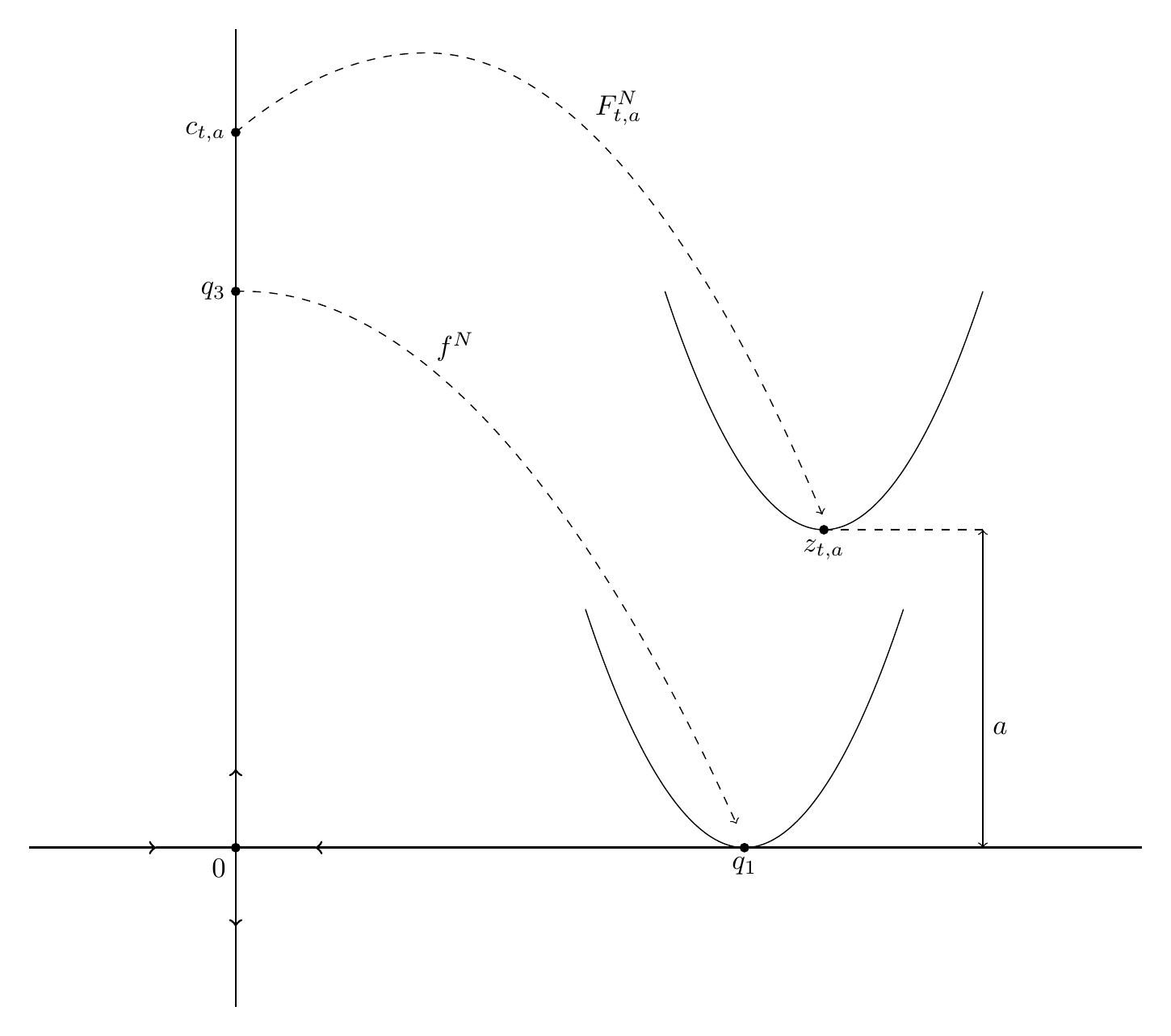}
\caption{Unfolding}
\label{Fig2}
\end{figure}
The condition $\lambda_{\text{max}}\mu_{\text{max}}^3<1$, see $(F3)$, allows us to choose $\theta\in(0,\frac{1}{2})$ such that 
\begin{equation}\label{thetacond}
1<\lambda_{\text{min}}^{2\theta}\mu_{\text{min}}^3 \text{ and } \lambda_{\text{max}}^{\theta}\mu_{\text{max}}< 1.
\end{equation}
We can choose any $\theta$ satisfying 
\begin{equation}\label{thetacond1}
0<\theta_0=\frac{\log\mu_{\text{max}}}{\log{\frac{1}{\lambda_{\text{max}}}}}<\theta<\frac{3}{2}\frac{\log\mu_{\text{min}}}{\log{\frac{1}{\lambda_{\text{min}}}}}=\theta_1<\frac{1}{2},
\end{equation}
where we used $(F3)$, the initial condition $\lambda_{\text{max}}\mu_{\text{max}}^3<1$ and (\ref{thetacond0}).

\section{Cascades of sinks}
In this section we are going to prove that, for a two dimensional set of parameters, an unfolding contains maps which have a sink of high period. Fix an unfolding $F$ and for each $\left(t, a\right)\in  [-t_0,t_0]\times  [-a_0,a_0]$ let $$\Gamma_{t, a}=\left\{(x,c(x,t, a))| x\in[-x_0,x_0]^{m-1}\right\}.$$ 
In the next lemma we build a curve $a_n$ of points, in the parameter space, whose corresponding critical values are mapped after $n$ steps  into $\Gamma_{t, a}$. 
\begin{lem}\label{Imtang}
For $n$ large enough, there exists a $\Cd$ function $a_n:[-t_0,t_0]\to (0,a_0]$ such that $$F^n_{t, a_n\left(t\right)}\left(z_{\left(t, a_n\left(t\right)\right)}\right)\in\Gamma_{\left(t, a_n\left(t\right)\right)}.$$ Moreover
\begin{equation}\label{dandt}
\frac{d a_n}{d t}=-n\frac{\partial\mu}{\partial t}\frac{1}{\mu^{n+1}}\left[1+O\left(|\lambda_1|^n\right)\right].
\end{equation}
\end{lem}
\begin{proof}
Let $\Gamma=\text{graph}(c)$, namely, $$\Gamma=\left\{\left(x, c(x,t,a), t,a\right)| \left( x,t,a\right)\in [-x_0,x_0]^{m-1}\times [-t_0,t_0]\times [-a_0,a_0] \right\}.$$ 
Then $\Gamma$ is a $\Cd$ codimension $1$ manifold transversal to $W^u(0)$. For $n\geq 0$ let
\begin{equation}\label{defgamma}
\Gamma_n=\left\{\left(x,y,t,a\right)| F^n_{t,a}\left( x,y\right)\in\Gamma\right\},
\end{equation}
and the limit of the $\Gamma_n$:s as $n\to\infty$ is, $$\Gamma_{\infty}=\left\{\left(x,0,t,a\right)| \left( x,t,a\right)\in [-x_0,x_0]^{m-1}\times [-t_0,t_0]\times [-a_0,a_0]  \right\}.$$
This follows since, for each $(t,a)\in [-t_0,t_0]\times [-a_0,a_0] $, $F_{t,a}$ is linear, $\Gamma_n$ converges to $\Gamma_{\infty}$ in the $\Cd$ topology. Namely for large $n$, $\Gamma_n$ is a graph of a function also denoted by $\Gamma_n$ and 
$$
\left\|\Gamma_n-\Gamma_{\infty}\right\|_{\Cd}\to 0.
$$
Let $z:[-t_0,t_0]\times [-a_0,a_0]\to\R^m\times [-t_0,t_0]\times [-a_0,a_0]$ be the $\Cd$ critical value function defined as
$$z(t,a)=\left(z_x(t,a), z_y(t,a),t,a\right)=\left(z_x(t,a), a,t,a\right),$$
where $z_{t,a}=\left(z_x(t,a),z_y(t,a)\right)$.
Observe that 

\begin{equation}\label{dz}
\frac{\partial z_y}{\partial a}=1\text{ and }z_y(t,0)=0. 
\end{equation}
Let $Z=\text{Image}(z)$. Because of (\ref{dz}), $Z$ is a manifold transversal to $\Gamma_{\infty}$. Hence, there exists $n_1>0$ such that, for all $n\geq n_1$, $Z$ is transversal to $\Gamma_n$. As consequence, for all $n\geq n_1$,
$$A_n=z^{-1}\left(\Gamma_n\right)$$ is a $\Cd$ codimension $1$ manifold. We define $a_n: t\mapsto a_n(t)$ as a function whose graph is $A_n$. 
Observe that, by Lemma \ref{functionc} and Remark \ref{c0=1}, the $\Cd$ function $c$ satisfies, $c(0,t,a)=1$. Hence $\partial c/\partial t=O(x)$ and   $\partial c/\partial a=O(x)$. 
Abusing the notation, we denote by $\lambda$ the diagonal matrix whose entries are the eigenvalues $\left\{\lambda_1,\dots,\lambda_{m-1}\right\}$ and we recall that
\begin{equation}\label{munan}
\mu^n a_n=c\left(\lambda^n z_x, t, a_n\right).
\end{equation}
By differentiating (\ref{munan}) and using $\partial c/\partial a, \partial c/\partial t=O(x)$ we have
$$
\mu^n\frac{d a_n}{dt}+na_n\mu^{n-1}\frac{\partial\mu}{\partial t}=O\left(\lambda_1^n\right).
$$
The lemma follows.
\end{proof}
Choose $\epsilon_0>0$ and define, for $n$ large enough,
$$
\mathcal{A}_n=\left\{(t,a)\in [-t_0,t_0]\times [-a_0,a_0] \left|\right. |a-a_n(t)|\leq\frac{\epsilon_0}{|\mu(t,a_n(t))|^{2n}}\right\}.
$$
The strip $\mathcal{A}_n$ is built to contain the parameters having a sink. The size has to be chosen carefully. The following remark describes the error obtained in the change of the eigenvalues while changing the parameters within $\mathcal{A}_n$.
\begin{rem}\label{muzeroothersforan}
Since $\partial\mu/\partial a$, $\partial\lambda_1/\partial t$ and $\partial\lambda_1/\partial a$ are all bounded we have that, for all $(t,\tilde a)\in\mathcal A_n$,
$$
\left(\frac{|\mu(t,a_n(t))|}{|\mu( t,\tilde a)|}\right)^n=1+O\left(\frac{n}{|\mu(t,a_n(t))|^{2n}}\right), \left(\frac{|\lambda_1(t,a_n(t))|}{|\lambda_1(\tilde t,\tilde a)|}\right)^n=1+O\left(\frac{n}{|\mu(t,a_n(t))|^{2n}}\right).
$$
\end{rem}

In the following we prove, for a properly chosen $\epsilon_0$, that for all parameters $(t,a)\in \mathcal{A}_n$, $F_{t,a}$ has a sink of period $n+N$ which we call {\it primary sink}. The method used appears already in \cite{ BP, Ro}. Namely, we find an invariant box in the phase space and we prove that the $F_{t,a}^{n+N}$ contracts this box. As a consequence, we get a sink. The following is a preliminary lemma.



\begin{lem}\label{Nderivatives}
There exist $x'_0<x_0$, $a'_0<a_0$, $b>0$ and $Q>0$, such that, for all $(t,a)\in [-t_0,t_0]\times [-a'_0,a'_0] $ and for every $(x,y)\in\Gamma_{t,a}$ with $|x|<| x'_0|$ the following holds. There exist matrices $A_{x,y}$, $B_{x,y}\neq 0$ and $C_{x,y}$ such that 
$F^N_{t,a}$ in coordinates centered in $(x,y)$ and $F^N_{t,a}(x,y)$ has the form
\begin{eqnarray}\label{NstepsTaylorformula}
\nonumber
F^N_{t,a}\left(\begin{matrix}
\Delta x\\\Delta y
\end{matrix}\right)&=&\left(\begin{matrix}
A_{x,y}& B_{x,y}\\C_{x,y} & 0
\end{matrix}\right)\left(\begin{matrix}
\Delta x\\\Delta y
\end{matrix}\right)\\&+&\left(\begin{matrix}
O_{1,1}\Delta x^2+O_{1,2}\Delta x\Delta y+O_{1,3}\Delta y^2\\ Q_{x,y}\Delta y^2+O_{2,1}\Delta x^2+O_{2,2}\Delta x\Delta y+O_{2,3} \Delta y^3
\end{matrix}\right),
\end{eqnarray} 
where the matrices $A_{x,y}$, $C_{x,y}$ and the vector $B_{x,y}$ are $\Cd$ dependent on $x$ and $y$, $Q_{x,y}>Q$, $|B_{x,y}|>b>0$ and the vector valued functions $O_{i,j}$ are $\Cd$ dependent on $x$ and $y$ and they are uniformly bounded.
\end{lem}
\begin{proof}
The lemma gives the Taylor expansions of $F^N_{t,a}$ when $(x,y)\in\Gamma_{t,a}$. It follows immediately from Lemma \ref{functionc} where we state that a vertical curve trough $(x,y)$ in $\Gamma_{t,a}$ is mapped to a curve with a non degenerate horizontal tangency. In particular $\partial Y_{t,a}/\partial y=0$ by (\ref{partialYpartialt}) and the horizontal tangency, $$DF^N_{t,a}(x,y)\left(\begin{matrix}
0\\1
\end{matrix}\right)=\left(\begin{matrix}
B_{x,y}\\0
\end{matrix}\right)$$ is not degenerate for all $(x,y)\in\Gamma_{t,a}$. This is a consequence of the following argument. Let $t\in[-t_0,t_0]$. Because $F^N_{t,0}(q_3(t))=q_1(t)$ is a non degenerate homoclinic tangency, we know that the vector $B_{q_{3}(t)}\neq 0$ and $Q_{q_3(t)}>0$. By taking $|a|<|a'_0|$, $|x|<|x'_0|$ small enough, the lower bounds on $Q_{x,y}$ and $B_{x,y}$ follow.
\end{proof}
For $n$ large enough, take $(t,a)\in A_n$ and denote by $c_n(t)=c_{t,a_n(t)}$ and by $z_n(t)=z_{t,a_n(t)}=(z_{n,x}(t),z_{n,y}(t))$. When the choice of $t$ is clear, we just use the notation $c_n$ and $z_n$.
Take $(t,a)\in\mathcal A_n$. For $\delta>0$ we define the box that is going to contain the sink,  
$$
B^n_{\delta}(t,a)=\left\{(x,y)| |x-z_{n,x}(t)|\leq\frac{1}{3},|y-z_{n,y}(t)|\leq\frac{\delta}{|\mu(t, a_n(t))|^{2n}}\right\}.
$$
 When the choice of $(t,a)$ is clear we just use the notation $B^n_{\delta}$.
 
In the next lemma we prove that $B^n_{\delta}$ returns into itself, see Figure \ref{Fig3}. Let $\tilde Q=\max Q_{x,y}$ where $Q_{x,y}$ as in Lemma \ref{Nderivatives}. 
\begin{figure}[h]
\centering
\includegraphics[width=0.9\textwidth]{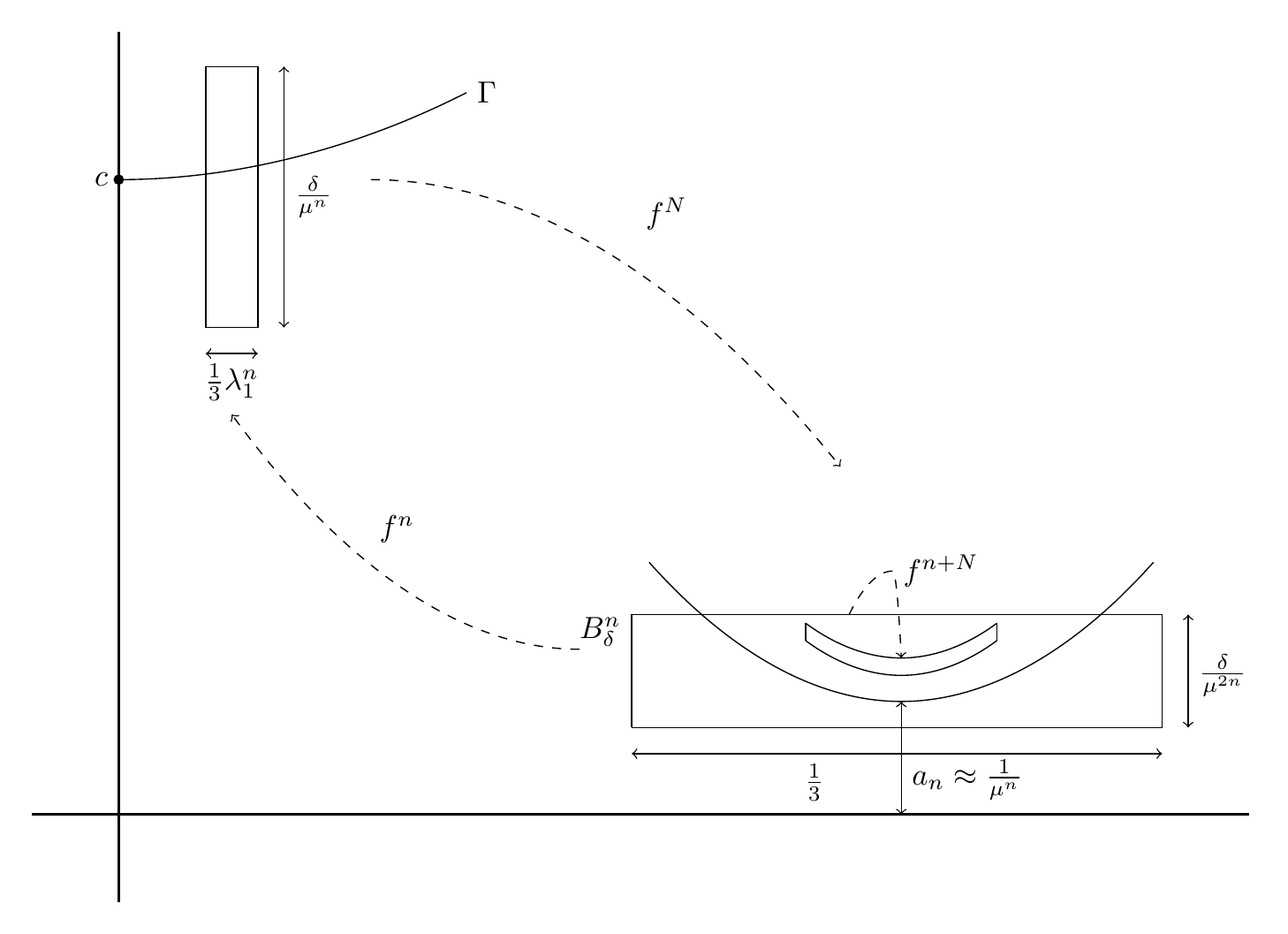}
\caption{Invariant Region}
\label{Fig3}
\end{figure}
\begin{lem}\label{periodicpoint}
Choose $\delta=\frac{1}{4\tilde Q}, \epsilon_0=\frac{1}{32\tilde Q}$. Then, for $n$ large enough and $(t,a)\in\mathcal A_n$,
$$
F^{n+N}_{t,a}\left(B^n_{\delta}\right)\subset B^n_{\delta}.
$$
\end{lem}
\begin{proof}
Fix $(t,a)\in\mathcal A_n$. For $n$ large enough, we write $F^{n+N}_{t,a}$ in coordinates centered at $z_n$, namely $F^{n+N}_{t,a}\left(\Delta x,\Delta y\right)=\left(\Delta \tilde x,\Delta \tilde y\right)$. Let $\left(\Delta x,\Delta y\right)$ be such that if $z_n+\left(\Delta x,\Delta y\right)\in B^n_{\delta}$ and $\Delta a=(a-a_n(t))$, then
\begin{equation}\label{deltas}
|\Delta x|\leq\frac{1}{3}, |\Delta y|\leq \frac{\delta}{|\mu(t, a_n(t))|^{2n}} \text{ and } |\Delta a|\leq \frac{\epsilon_0}{|\mu(t, a_n(t))|^{2n}}.
\end{equation}
Denote by $\left(\Delta x',\Delta y'\right)=F^n_{t,a}\left(\Delta x,\Delta y\right)-c_{t,a}$. Using that $F^{n}_{t,a}$ is linear on $B^n_{\delta}$, see (\ref{Flinear}), we get
\begin{eqnarray*}\label{deltaxprime}
|\Delta x'|\leq 2|\lambda_1(t,a)|^n,
\end{eqnarray*}
and 
\begin{eqnarray*}\label{deltayprime}
|\Delta y'|&\leq &\left|\mu(t, a)^{n}\right|\left|\Delta y\right|+\left|\left|\mu(t, a)^{n}\right|z_{n,y}-c(0,t,a)\right|\\&\leq&\left|\mu(t, a)^{n}\right|\left|\Delta y\right|+\left|\left(\left[1+O\left(\frac{n}{\mu(t, a_n(t))^{2n}}\right)\right]\left|\mu(t, a_n(t))^{n}\right|z_{n,y}\right)-c_{n,y}\right|\\&+&\left|c_{n,y}-c(0,t,a)\right|,
\end{eqnarray*}
where we used Remark \ref{muzeroothersforan}. Since $F^n_{t,a_n(t)}\left(z_n\right)\in\Gamma_{t,a_n(t)}$, then $$\left|\mu(t, a_n(t))^{n}\right|z_{n,y}=c_{n,y}+O\left(\lambda_1(t, a_n(t))^n\right).$$ Moreover $\left|c_{n,y}-c(0,t,a)\right|=O\left(\Delta a\right)$ (see Lemma \ref{functionc}). Hence, by (\ref{deltas}) and by Remark \ref{muzeroothersforan} we get
\begin{eqnarray*}\label{deltayprime}
|\Delta y'|
&\leq &\left|\mu(t, a)^{n}\right|\left|\Delta y\right|+O\left(\lambda_1(t, a_n(t))^n\right)+O\left(\frac{n}{\mu(t, a_n(t))^{2n}}\right)+O\left(\frac{1}{\mu(t, a_n(t))^{2n}}\right)\\&\leq &\frac{\delta}{\left|\mu(t, a_n(t))^{n}\right|}\left[1+O\left(\frac{n}{\mu(t, a_n(t))^{2n}}\right)\right]+O\left(\frac{n}{\mu(t, a_n(t))^{2n}}\right).
\end{eqnarray*}
By Lemma \ref{Nderivatives} (center $F^{N}_{t,a}$ in coordinates around $c_{t,a}$) extended to include also the Taylor expansion in $\Delta a$   
we get, for $n$ large enough
\begin{eqnarray*}
|\Delta\tilde x|&=&O\left(\Delta x'\right)+O\left(\Delta y'\right)+O\left(\Delta a\right)\leq\frac{1}{3},
\end{eqnarray*}
and 
\begin{eqnarray*}
|\Delta\tilde y|&\leq &O\left(|\lambda_1(t,a)|^n \right)+Q_{x,y}\left|\Delta y'\right|^2+O\left(\frac{1}{\mu(t, a_n(t))^{3n}}\right)+|\Delta a|\\
&\leq &\frac{\tilde Q \delta^2}{\left|\mu(t, a_n(t))^{2n}\right|}+O\left(\frac{1}{\mu(t, a_n(t))^{3n}}\right)+\frac{\epsilon_0}{\left|\mu(t, a_n(t))^{2n}\right|},
\end{eqnarray*}
where we also used $(F3)$ and Remark \ref{muzeroothersforan}.
By our choice of $\epsilon_0$ and $\delta$, for $n$ large enough, the lemma follows.
\end{proof}

We fix $\epsilon_0,$ and $\delta$ such that Lemma \ref{periodicpoint} holds. 
We are now ready to prove the existence of a sink. This is achieved in the next proposition by proving that $F^{n+N}$ contracts the box $B_{\delta}^n$.
\begin{prop}\label{sink}
For $n$ large enough and for all $(t,a)\in \mathcal A_n$, $ B^n_{\delta}(t)$ has a unique periodic point which is a sink. 
\end{prop}
\begin{proof}
Because $F_{t,a}^n$ is linear on $B_{\delta}^n$, the image $F_{t,a}^n\left(B_{\delta}^n\right)$ is contained in a neighborhood of $c_{t,a}$ of diameter smaller than  $\delta|\mu|^{-n}\left(1+O\left({n}|\mu|^{-n}\right)\right)
$. From this and by differentiating (\ref{NstepsTaylorformula}) with respect to $\Delta x$ and $\Delta y$ we get
\begin{eqnarray*}
DF_{t,a}^N&=&\left(\begin{matrix}
O(1) &O(1) \\
O(1) & 2{\delta|\mu|^{-n} Q_{x,y}}\left(1+O\left({n}{|\mu|^{-n}}\right)\right)\\
\end{matrix}\right).
\end{eqnarray*}
Note that $N$ is fixed. Using again that $F_{t,a}^n$ is linear on $B_{\delta}^n$ we obtain that 
\begin{eqnarray*}
DF_{t,a}^{n+N}&=&\left(\begin{matrix}
O(1) &O(1) \\
O(1) &{2\delta }{|\mu|^{-n}}Q_{x,y}\left(1+O\left({n}{|\mu|^{-n}}\right)\right)
\end{matrix}\right)\left(\begin{matrix}
O\left(|\lambda_{1}|^n\right) &0 \\
0& \mu^n
\end{matrix}\right)\\
&=&\left(\begin{matrix}
O\left(|\lambda_{1}|^n\right) & O\left(|\mu|^n\right) \\
O\left(|\lambda_{1}|^n\right)& 2\delta Q_{x,y}\left(1+O\left({n}{|\mu|^{-n}}\right)\right)
\end{matrix}\right).
\end{eqnarray*}
Let $D=\left(\begin{matrix}
O\left(|\lambda_{1}|^n\right) & O\left(|\mu|^n\right) \\
O\left(|\lambda_{1}|^n\right)& \frac{3}{4}
\end{matrix}\right)$ be a positive matrix and $\left(DF_{t,a}^{n+N}\right)^k(\Delta x,\Delta y)=(\Delta x_k,\Delta y_k)$, then by the choice of $\delta$
\begin{eqnarray*}
\left(\begin{matrix}
|\Delta x_{k+1}| \\
|\Delta y_{k+1}| 
\end{matrix}\right)\leq D\left(\begin{matrix}
|\Delta x_{k}| \\
|\Delta y_{k}| 
\end{matrix}\right).
\end{eqnarray*}
Observe that $\text{tr}(D)=\frac{3}{4}+O\left(|\lambda_{1}|^n\right)$ and $\text{det}(D)=O(\left(|\lambda_1||\mu|\right)^n)$. As consequence, for $n$ large enough, $|\Delta x_{k}|, |\Delta y_{k}|\to 0$ exponentially fast. Hence, the periodic point in $B^n_{\delta}$ is a sink whose basin of attraction contains $B^n_{\delta}$.
\end{proof}

\section{Curves of secondary tangencies}
In this section we are going to show that, for $n$ large enough,  there exist curves $b_n$ whose parameters have a new homoclinic tangency which is again a strong homoclinic tangency. The curves  $b_n$ intersect the strip $\mathcal A_n$ where the sinks occur, see Figure \ref{Fig6}. In particular, maps in this intersection have a sink and a strong homoclinic tangency. The following lemmas are a preparation for constructing these tangencies.
\begin{figure}
\centering
\includegraphics[width=1\textwidth]{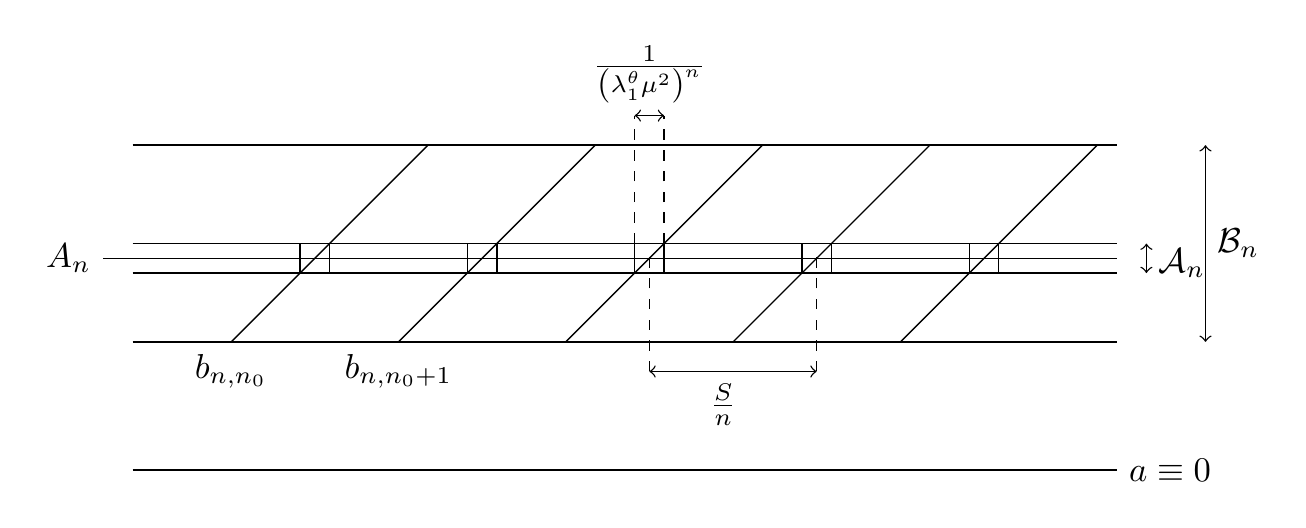}
\caption{Curves of Secondary Tangencies (Section $4$)}
\label{Fig6}
\end{figure}
\subsection{Existence of the secondary tangencies}
Remember the choice of $\theta$ in (\ref{thetacond}) and (\ref{thetacond1}).
Choose $\epsilon_1>0$ small enough and define, for large $n$,
$$
\mathcal{B}_n=\left\{(t,a)\in [-t_0,t_0]\times [-a_0,a_0] \left|\right. |a-a_n(t)|\leq{\epsilon_1}{|\lambda_1(t,a_n(t))|^{\theta n}}\right\}.
$$

The strip $\mathcal{B}_n$ is built to contain the curves of tangencies. The size has to be chosen carefully. Later the value of $\epsilon_1$ will be adjusted downward, see for example Lemma \ref{curvaturez3}. Observe that, by (\ref{thetacond}) $|\lambda_1(t,a_n(t))|^{\theta n}|\mu(t,a_n(t))|^{2 n}>1$ and for $n$ large enough $|\lambda_1(t,a_n(t))|^{\theta n}|\mu(t,a_n(t))|^{2 n}>\epsilon_0/\epsilon_1$. In particular, for $n$ large enough, $\mathcal{B}_n$ contains $\mathcal{A}_n$. The following remark points out the distortion of the eigenvalues when changing the parameters in $\mathcal{B}_n$.
\begin{rem}\label{muzeroothers}
Observe that if we take $n$ large enough, then for all $t\in [-t_0,t_0]$, $(\tilde t,\tilde a)$ with $|\tilde t- t|=O\left({|\lambda_1(t,a_n(t)))|^{\theta n}}\right)$ and $|\tilde a- a|=O\left({|\lambda_1(t,a_n(t)))|^{\theta n}}\right)$, we have 
$$
\left(\frac{|\mu(t,a_n(t))|}{|\mu(\tilde t,\tilde a)|}\right)^n=1+O\left(n|\lambda_1(t,a_n(t))|^{\theta n}\right), \left(\frac{|\lambda_1(t,a_n(t))|}{|\lambda_1(\tilde t,\tilde a)|}\right)^n=1+O\left(n|\lambda_1(t,a_n(t))|^{\theta n}\right)
$$
where we used the fact that $\partial\mu/\partial t$ and $\partial\mu/\partial a$ are bounded.
Moreover if $(t,a)\in\mathcal B_n$, $|a-a_n(t)|\leq{\epsilon_1}{|\lambda_1(t,a_n(t))|^{\theta n}} $, then
$$
\frac{|\mu(t,a_n(t))^n |a-a_n(t)||}{|\lambda_1(t,a)|^{\theta n}|\mu(t,a)|^{n}}\leq{2\epsilon_1}.
$$
\end{rem}
Fix $(t,a)\in\mathcal B_n$ and in the sequel we will suppress this choice in the notation, for example, $z=z_{t,a}$, $c=c_{t,a}$ (see Definition \ref{criticalpoints}), $\lambda_1=\lambda_1(t,a)$ and $\mu=\mu(t,a)$. By the initial conditions on the family, for $n$ large enough and $W^u_{\text{loc}}(z)$ small enough, $F^{\theta n}_{t,a}\left(W^u_{\text{loc}}(z)\right)$ intersects $\Gamma$ in exactly two points.  Choose one of these points
$$c'\in F^{\theta n}_{t,a}\left(W^u_{\text{loc}}(z)\right)\cap \Gamma,$$ and a local unstable manifold $W^u_{\text{loc}}(c')$ of diameter $L|\lambda_1|^{\theta n}$ with $L$ a big constant which is going to be chosen later. Let $z^{(1)}=(z_{x}^{(1)},z_{y}^{(1)})$ be the lowest point of $ F^N\left (W^u_{\text{loc}}(c')\right)$, see Figure \ref{Fig4a}. 
\begin{figure}
\centering
\includegraphics[width=0.9\textwidth]{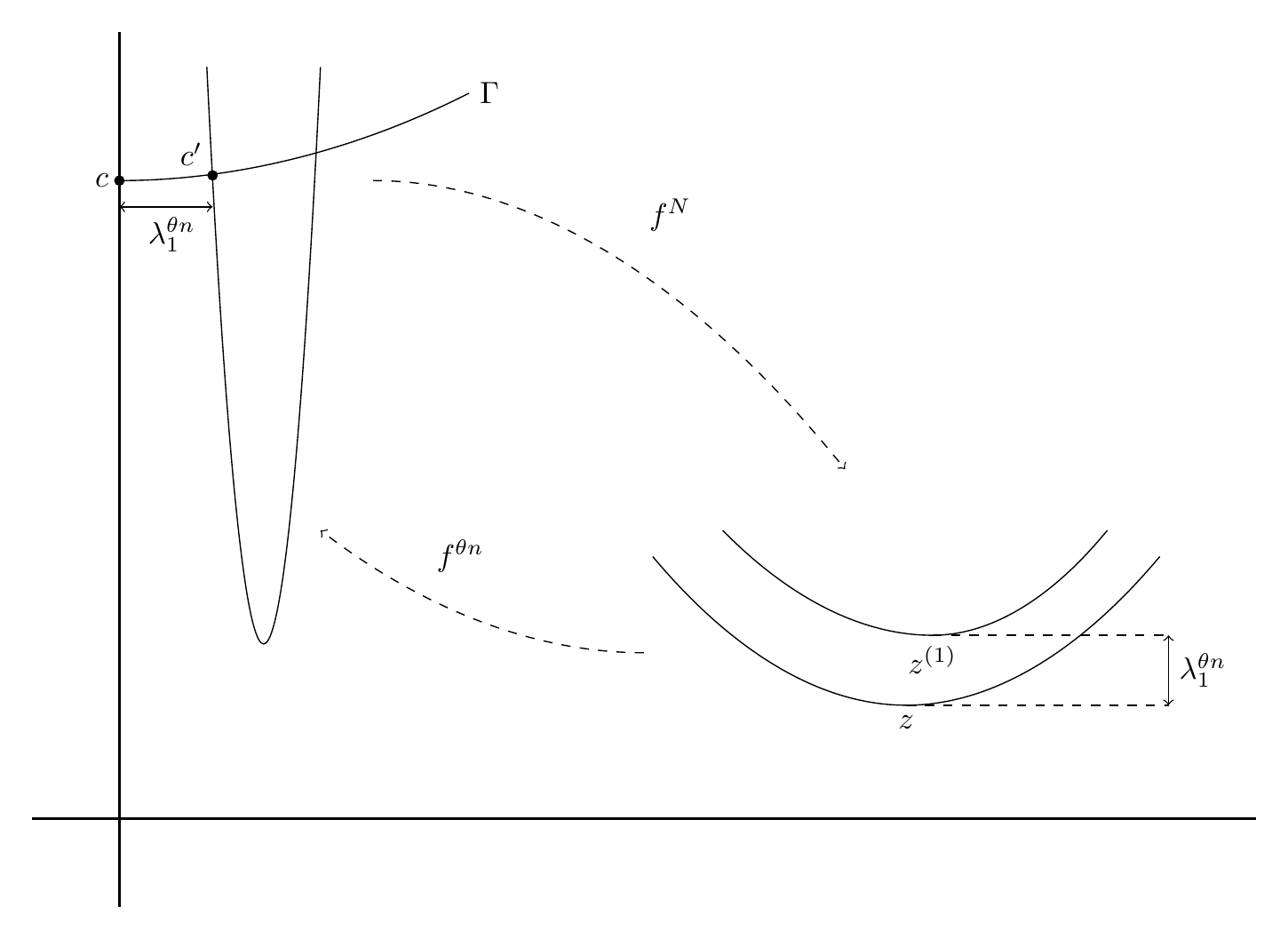}
\caption{$N+\theta n$ iterates}
\label{Fig4a}
\end{figure}
\begin{lem}\label{z1distz}
There exists a uniform constant $K>0$ such that, for $n$ large enough and for $(t,a)\in\mathcal B_n$,
$$\frac{1}{K}|\lambda_1|^{\theta n}\leq |z_{y}^{(1)}-z_{y}|\leq {K}|\lambda_1|^{\theta n},$$
and the tangent space
$T_{z^{(1)}}W^u_{\text{loc}}(z^{(1)})\subset\R^{m-1}\times \left\{0\right\}$ is horizontal. Moreover, the limit $\lim_{n\to\infty} { |z_{y}^{(1)}-z_{y}|}/{|\lambda_1|^{\theta n}}$ exists.
\end{lem}
\begin{proof}
Use coordinates centered at the critical point $c$  of the parameter $(t,a)$ and let $(\Delta x,\Delta y)\in W^u_{\text{loc}}(c')$ with $(\Delta \tilde x,\Delta\tilde y)\in F_{t,a}^N\left(W^u_{\text{loc}}(c')\right)$ centered in $z$. By Lemma \ref{Nderivatives} and the fact that $|\Delta y|=O\left(|\lambda_1|^{\theta n}\right) $we have
\begin{eqnarray*}
\Delta\tilde y&=&C\Delta x+Q(\Delta y)^2+O\left(|\Delta x|^2+|\Delta x||\Delta y|+|\Delta y|^3\right)\\
&=&C\Delta x+O\left(|\lambda_1|^{2\theta n}\right).
\end{eqnarray*}
Because the first coordinate of $q_1$ is equal to $1$, see $(F4)$ and the following normalization, we have that $\Delta x=\lambda_1^{\theta n}e_1+O\left(|\lambda_2|^{\theta n}\right)$ and because $Df^N_{q_3}(e_1)\notin T_{q_1}W^s(p)$ has a component in the positive $y$ direction we get that 
\begin{eqnarray*}
\frac{1}{K}|\lambda_1|^{\theta n}+O\left(|\lambda_1|^{2\theta n}+|\lambda_2|^{\theta n}\right)\leq \Delta\tilde y&\leq &K|\lambda_1|^{\theta n}+O\left(|\lambda_1|^{2\theta n}+|\lambda_2|^{\theta n}\right),
\end{eqnarray*}
with $K$ a positive uniform constant. The first claim of the lemma is then proved. For the second one, use the coordinates centered in $c'$. Let  $(\Delta x,\Delta y)\in W^u_{\text{loc}}(c') $ with $(\Delta \tilde x,\Delta\tilde y)= F_{t,a}^N\left((\Delta x,\Delta y)\right)$ written in coordinates centered in $F^N(c')$. Consider $W^u_{\text{loc}}(c')$ as a graph over the $y$-axis and take a tangent vector $(\Delta u,\Delta v)\in T_{c'}W^u_{\text{loc}}(c')$. Define the slope-vector  at $c'$ as $s=\Delta u/\Delta v$. Then \begin{eqnarray}\label{slopeespr}
s=O\left(\left(\frac{|\lambda_1|^2}{|\mu|}\right)^{\frac{\theta n}{2}}\right).
\end{eqnarray}
This estimate is obtained as follows. Let $(z_x+\Delta z_x, z_y+\Delta z_y)\in W^u_{\text{loc}}(z)$ such that $F^{\theta n}\left(z_x+\Delta z_x, z_y+\Delta z_y\right)=c'$ and $(\Delta u_0,\Delta v_0)$ be a tangent vector to $W^u_{\text{loc}}(z)$ at the point $\left(z_x+\Delta z_x, z_y+\Delta z_y\right)$. Then there exist uniform constants $K_1,K_2$ such that $${K_1^{-1}}{|\mu|^{-\theta n}}\leq \Delta z_y\leq K_1{|\mu|^{-\theta n}},$$ and because of the quadratic behavior of $W^u_{\text{loc}}(z)$,  $$K_2^{-1}{|\mu|^{-\frac{\theta n}{2}}}\leq |\Delta z_x|\leq K_2{|\mu|^{-\frac{\theta n}{2}}}.$$ In particular $$K_2^{-1}{|\mu|^{-\frac{\theta n}{2}}}|\Delta u_0|\leq |\Delta v_0|\leq K_2{|\mu|^{-\frac{\theta n}{2}}}|\Delta u_0|.$$ Observe that $$s=O\left(\frac{|\lambda_1|^{\theta n}}{|\mu|^{\theta n}}\frac{\Delta u_0}{\Delta v_0}\right).$$ The estimate for the slope follows. Use now coordinates $\left(\Delta x',\Delta y'\right)$ centered around $z$ with $F^{\theta n}\left(\Delta x',\Delta y'\right)=\left(\Delta x,\Delta y\right)$. Then there is a non zero smooth function $R$ such that
$$|\Delta x'|=R\left(\Delta y'\right)\left(\Delta y'\right)^{\frac{1}{2}},$$ and 
$$\Delta y'=O\left(\frac{1}{\mu^{\theta n}}\right).$$
In particular,
$$
\frac{d^2|\Delta x'|}{d \left(\Delta y'\right)^2}=O\left(\left(\Delta y'\right)^{-\frac{3}{2}}\right)=O\left(|\mu|^{\frac{3}{2}\theta n}\right).
$$
This implies that 
$$
\frac{d^2|\Delta x|}{d \left(\Delta y\right)^2}=O\left(|\mu|^{\frac{3}{2}\theta n}\frac{|\lambda_1|^{\theta n}}{|\mu|^{2\theta n}}\right)=O\left(\left(\frac{|\lambda_1|^2}{|\mu|}\right)^{\frac{\theta n}{2}}\right).
$$
Hence,
\begin{equation}\label{straightline}
\Delta x=s\Delta y+O\left(\left(\frac{|\lambda_1|^2}{|\mu|}\right)^{\frac{\theta n}{2}}\right)(\Delta y)^2.
 \end{equation}
  From (\ref{straightline})  and (\ref{slopeespr}) we have that $|\Delta x|=O\left(s\Delta y\right)=o\left(\Delta y\right)$. From Lemma \ref{Nderivatives}, 
\begin{eqnarray*}
\Delta\tilde x&=&B\Delta y+O_x\left(\Delta y\right) (\Delta y)^2\\
\Delta\tilde y&=&s\Delta y+Q(\Delta y)^2+O_{y}\left(\Delta y\right)(\Delta y)^3,
\end{eqnarray*}
where $B, Q\neq 0$, $O_x$ and $O_{y}$ are bounded differentiable functions in $\Delta y$. 
After differentiating one obtains
\begin{equation}\label{curveatz1}
\frac{d\tilde y}{d\tilde x}=\frac{2Q}{B}\Delta y+O\left(\left(\frac{|\lambda_1|^2}{|\mu|}\right)^{\frac{\theta n}{2}}\right),
\end{equation} 
where we used (\ref{slopeespr}). As consequence $z^{(1)}$ which is the point where ${d\tilde y}/{d\tilde x}=0$ corresponds to $$\Delta y=O\left(\left(\frac{|\lambda_1|^2}{|\mu|}\right)^{\frac{\theta n}{2}}\right)=o\left(|\lambda_1|^{\theta n}\right)\leq L|\lambda_1|^{\theta n}.$$ Indeed $W^u_{\text{loc}}(z^{(1)})$ has a horizontal tangency.
\end{proof}

 Let $z^{(2)}=(z_{x}^{(2)},z_{y}^{(2)})=F_{t,a}^n(z^{(1)})$ and take a local unstable manifold $W^u_{\text{loc}}(z^{(2)})$ of vertical size $L\left(|\lambda_1|^{\theta}|\mu|\right)^n$ with $L$ a large constant which is going to be chosen later. Observe that $z^{(2)}$ is located close to $c$, see Figure \ref{Fig4b}. The following lemma gives information on the size and the position of $W^u_{\text{\rm loc}}(z^{(2)})$. In particular Lemma \ref{ymaxymin} states how much $W^u_{\text{loc}}(z^{(2)})$ is shifted up with respect to the critical point $c$, see Figure \ref{Fig4b}. Its proof follows from Lemma \ref{z1distz} and from the fact that we are in the domain of linearization.
 \begin{lem}\label{ymaxymin}
There exists a uniform constant $K>0$ such that the following holds. For $n$ large enough and for $(x,y)\in W^u_{\text{\rm loc}}(z^{(2)})$,
 \begin{eqnarray*}
 \frac{1}{K}\left(|\lambda_1|^{\theta}|\mu|\right)^n\leq y- |\mu_0|^n z_{n,y}+O\left(|\mu_0|^n|a-a_n(t)|\right)\leq K\left(|\lambda_1|^{\theta}|\mu|\right)^n,
  \end{eqnarray*}
  and 
  $$x=O\left(|\lambda_1|^n\right),$$
  where $z_n=z(t,a_n(t))$, $\mu_0=\mu(t,a_n(t))$ and $\mu=\mu(t,a)$.
 \end{lem}
 The curves $W^u_{\text{\rm loc}}(z^{(1)})$, $W^u_{\text{\rm loc}}(z^{(2)})$ and $W^u_{\text{\rm loc}}(z^{(3)})$ (which will be introduced later) are graphs over smooth curves in the $x$-plane. We can describe these curves  in coordinates centered at $z^{(i)}$. In particular, for each component of $W^u_{\text{\rm loc}}(z^{(i)})\setminus z^{(i)}$, there exists a smooth function $w$ such that if $(\Delta x,\Delta y)\in W^u_{\text{\rm loc}}(z^{(i)})$ then $\Delta y=w\left(|\Delta x|\right)$.

 In the following lemma we study the curvature of $W^u_{\text{\rm loc}}(z^{(2)})$.
 \begin{lem}\label{shape}
For $n\ge 1$ large enough, in a coordinate system centered in $z^{(2)}$ the $\Cq$ curve $W^u_{\text{\rm loc}}(z^{(2)})$ is given by 
$$
\Delta y=Q_2\left(\frac{|\mu|}{|\lambda_1|^2}\right)^n|\Delta x|^2+O\left(|\Delta x|^3\left(\frac{|\mu|}{|\lambda_1|^3}\right)^n\right),
$$
where $(\Delta x,\Delta y)\in W^u_{\text{\rm loc}}(z^{(2)})$, $0\leq\Delta y<L\left(|\lambda_1|^{\theta}|\mu|\right)^n$, $|\Delta x|=O\left(|\lambda_1|^{2+\theta}\right)^{\frac{n}{2}}$ and the constant $Q_2>0$ is a uniformly bounded and bounded away from zero.
 \end{lem}
 \begin{proof}
 Consider the $W^u_{\text{loc}}(z^{(1)})$ of vertical size $L|\lambda_1|^{\theta n}$ and let  $B\in T_{z^{(1)}}W^u_{\text{loc}}(z^{(1)})$ be a unit vector. Then in coordinate centered in $z^{(1)}$ the $\Cq$ curve $W^u_{\text{loc}}(z^{(1)})$ is given by 
 \begin{eqnarray*}
  \Delta y &=& Q_1|\Delta x|^2+O(|\Delta x|^3).
 \end{eqnarray*}
where $Q_1$ is a uniform constant, see (\ref{curveatz1}).
 Given the fact that $B$ has a non zero first coordinate, see $(F5)$, then the linear map $F^n_{t,a}$ turns the curve $W^u_{\text{loc}}(z^{(1)})$ into a curve of the form as stated in the lemma. Namely, recall that the curves $W^u_{\text{loc}}(z^{(1)})$ and $W^u_{\text{loc}}(z^{(2)})$ are graphs of functions over smooth curves in the $x$-plane. The curve $W^u_{\text{loc}}(z^{(1)})$ is parametrized by $|\Delta x|$ and by iteration by the linear map we obtain that the curve $W^u_{\text{loc}}(z^{(2)})$ is parametrized by $\lambda^n_1\left(1+O\left(\left|\frac{\lambda_2}{\lambda_1}\right|^n\right)\right)|\Delta x|$.   \\
  Observe also that for $n$ large enough and for all $(t,a)\in\mathcal B_n$, the highest vertical point of $W^u_{\text{loc}}(z^{(2)})$ is still in the domain of linearization. The factor $\left({|\mu|}/{|\lambda_1|^2}\right)^n$ comes from the change of curvature by $n$ linear iterates.
  \end{proof}

Let $z^{(3)}=(z_{x}^{(3)},z_{y}^{(3)})$ be the lowest point of $F^N_{t,a}\left(W^u_{\text{loc}}(z^{(2)})\right)$. In the next lemma we prove that the new critical value $z^{(3)}$ is far from the critical value $z$ and from now on they will have independent behavior. In particular, while $z$ is contained in the basin of the sink, the new critical value $z^{(3)}$ it is outside of the basin and it will be involved in the creation of the new tangency.
\begin{lem}\label{hnbounds}
For $\epsilon_1$ small enough, there exists a uniform constant $H>0$ such that
$$
\frac{1}{H}\left(\left|\lambda_1\right|^{\theta}\left|\mu\right|\right)^{2n}\leq z_{y}^{(3)}-z_{y}\leq {H}\left(\left|\lambda_1\right|^{\theta}\left|\mu\right|\right)^{2n},
$$
and 
$T_{z^{(3)}}W^u_{\text{loc}}(z^{(3)})\subset\R^{m-1}\times \left\{0\right\}$ is horizontal.
\end{lem}
\begin{proof}
Use coordinates centered at $c$ and let $(\Delta x,\Delta y)\in W^u_{\text{loc}}(z^{(2)})=F_{t,a}^n\left(W^u_{\text{loc}}(z^{(1)})\right)$. From Lemma \ref{ymaxymin} we get 
\begin{eqnarray*}
\Delta y&=&y- |\mu_0|^n z_{n,y}+ |\mu_0|^n z_{n,y}-1\\&\geq &\frac{1}{K}\left(\left|\lambda_1\right|^{\theta}\left|\mu\right|\right)^{n}+O\left(|\mu_0|^n|a-a_n(t)|\right)+O\left(|\lambda_1 (t,a_n(t))|^n\right),
\end{eqnarray*}
where we used that $F^n_{t,a_n(t)}(z_{n})\in\Gamma$ which implies that $$| |\mu_0|^n z_{n,y}-1|=O\left(\left|F^n_{t,a_n(t)}(z_{n})_x\right|\right)=O\left(|\lambda_1 (t,a_n(t))|^n\right).$$
As a consequence, for $\epsilon_1$ small enough, the position of $W^u_{\text{loc}}(z^{(2)})$ satisfies
\begin{equation}\label{delatydist}
\tilde K\left(\left|\lambda_1\right|^{\theta}\left|\mu\right|\right)^{n}\leq\Delta y \leq (\tilde K+L)\left(\left|\lambda_1\right|^{\theta}\left|\mu\right|\right)^{n},
\end{equation}

\begin{equation}\label{delatxdist}
|\Delta x|=O\left(\left|\lambda_1\right|^{n}\right),
\end{equation}
where $\tilde K$ is a uniform constant. We apply Lemma \ref{Nderivatives} in image coordinates $(\Delta\tilde x, \Delta\tilde y)$ centered at $z$ and we find 
$$
\Delta\tilde y=O\left(\left|\lambda_1\right|^{n}\right)+Q\left(\Delta y\right)^2+O\left(\left(\left|\lambda_1\right|^{\theta}\left|\mu\right|\right)^{3n}\right)=Q\left(\Delta y\right)^2+O\left(\left(\left|\lambda_1\right|^{\theta}\left|\mu\right|\right)^{3n}\right),
$$ 
where we also used $(\ref{thetacond})$.
Here $Q$ is the initial curvature of Lemma \ref{Nderivatives}. From (\ref{delatydist}) we get the stated bounds for $ z_{y}^{(3)}-z_{y}$. Moreover for $L$ large enough the minimum is obtained in the interior of $W^u_{\text{loc}}(z^{(2)})$ and 
$T_{z^{(3)}}W^u_{\text{loc}}(z^{(3)})\subset\R^{m-1}\times \left\{0\right\}$ is horizontal.
\end{proof}
In the following lemma we study the curvature of $W^u_{\text{\rm loc}}(z^{(3)})$. One can observe that the curvature of $W^u_{\text{\rm loc}}(z^{(3)})$ has grown considerably compared to that of $W^u_{\text{\rm loc}}(z)$. 
\begin{lem}\label{curvaturez3} 
For $\epsilon_1$ small enough the following holds. Let $(t,a)\in\mathcal B_n$ then $W^u_{\text{loc}}(z^{(3)}(t,a))$ is the graph of a function and its curvature satisfies
$$
\text{\rm curv}\left(W^u_{\text{loc}}(z^{(3)}(t,a))\right)\geq C\left(\frac{|\mu|^4}{|\lambda_1|^{2-3\theta}}\right)^n,
$$
with $C>0$ a uniform constant and ${|\mu|^4}/{|\lambda_1|^{2-3\theta}}>1$.
\end{lem}
\begin{proof} Use coordinates centered at $z^{(2)}$ and let $m_5=(\Delta x,\Delta y)\in W^u_{\text{\rm loc}}(z^{(2)})$ be the preimage of $z^{(3)}$ under $N$ iterates, i.e. $F^N_{t,a}(\Delta x,\Delta y)=z^{(3)}$.
Then $0\leq\Delta y\leq L\left(|\lambda_1|^{\theta}|\mu|\right)^{n}$. Let $v=\left(\begin{matrix}
v_1\\v_2
\end{matrix}\right)$ be the tangent vector at $m_5$ to $ W^u_{\text{loc}}(z^{(2)})$. We use Lemma \ref{shape} and we obtain that 
\begin{equation}\label{eq1}
|v_2|=\left[2Q_2\left(\frac{|\mu|}{|\lambda_1|^2}\right)^n|\Delta x|+O\left(|\Delta x|^2\left(\frac{|\mu|}{|\lambda_1|^3}\right)^n\right)\right]|v_1|.
\end{equation}
By Lemma \ref{functionc} we have that ${\partial^2 Y}/{\partial y^2}$ is bounded away from zero in a neighbourhood of $c$. As a consequence, by Lemma \ref{Nderivatives} we have $$DF^N_{t,a}(m_5)=\left(\begin{matrix}
A&B\\C&D
\end{matrix} \right),$$ where $A$, $B$, $C$ are matrices and $D$ is a number. Observe that, from Lemma \ref{functionc}, there exists a uniform constant $d'>0$ such that  $$D=d'\left(m_{5,y}-1\right)+O(m_{5,x}),$$ and by Lemma \ref{ymaxymin}, $m_{5,x}=O\left(\lambda_1^n\right)$.  By using again the Lemma \ref{ymaxymin} and the bound on $m_{5,x}$, by taking $\epsilon_1$ small enough, we get that $D\geq d\left(|\lambda_1|^{\theta}|\mu|\right)^{n}$ where $d>0$ is a uniform constant. Because $F^N_{t,a}(m_5)$ is the minimum of the curve $W^u_{\text{loc}}(z^{(3)})$, then 
\begin{equation}\label{eq2}
Cv_1+Dv_2=0.
\end{equation}
Use coordinates $\left(\Delta x',\Delta y'\right)$ along $W^u_{\text{loc}}\left(z^{(1)}\right)$ centered around $z^{(1)}$ such that $F^n\left(\Delta x',\Delta y'\right)=\left(\Delta x,\Delta y\right)$. Then, from (\ref{delatydist}), $\Delta y'=O\left(|\lambda_1|^{\theta n}\right)$ and by the quadratic behavior, $|\Delta x'|=O\left(|\lambda_1|^{\frac{\theta}{2}n}\right)$. This implies that 
\begin{equation}\label{nightbeforetheday}
|\Delta x|=O\left(|\lambda_1|^{\left(1+\frac{\theta}{2}\right)n}\right).
\end{equation}
By (\ref{eq1}), (\ref{eq2}), Lemma \ref{shape}, the fact that $D\geq d\left(|\lambda_1|^{\theta}|\mu|\right)^{n}$ and that $C$ is a uniform constant, we have
$$
2dQ_2|\Delta x|+O\left(|\Delta x|^2\left(\frac{1}{|\lambda_1|}\right)^n\right)=O\left(\left(\frac{|\lambda_1|^{2-\theta}}{|\mu|^2}\right)^n\right).
$$
Moreover by (\ref{nightbeforetheday}), $|\Delta x|\left(\frac{1}{|\lambda_1|}\right)^n=O\left(|\lambda_1|^{\frac{\theta}{2}n}\right)$ and because $|\lambda_1|^{\frac{\theta}{2}n}<1$ we obtain
\begin{equation}\label{eq3}
|\Delta x|=O\left(\left(\frac{|\lambda_1|^{2-\theta}}{|\mu|^2}\right)^n\right).
\end{equation}
Observe that we can reduce the higher dimensional problem to a purely $2$-dimensional one by considering the projection of $ W^u_{\text{loc}}(z^{(2)})$ to the $(x_1,y)$-plane. In fact the curve $ W^u_{\text{loc}}(z^{(2)})$ projects to the $(x_1,y)$-plane and this projection distorts the curvature by an exponential (with respect to $n$) small amount, see $(f1)$, $(F4)$ and $(F5)$. As consequence it is enough to estimate the curvature of the projection of the curve $ W^u_{\text{loc}}(z^{(2)})$ to the $(x_1,y)$-plane. The curvature of this projected curve is proportional to the curvature of the curve $ W^u_{\text{loc}}(z^{(2)})$. Because $F^n$ is linear, we can estimate the curvature of the projection of the curve $ W^u_{\text{loc}}(z^{(2)})$ to the $(x_1,y)$-plane by calculating how the linear map, which preserves the $(x_1,y)$-plane, changes the curvature of $ W^u_{\text{loc}}(z^{(1)})$ on the $(x_1,y)$-plane. We denote the $x_1$-coordinate by $x$. 
We calculate now the second derivatives using again Lemma \ref{shape} and the fact that the curve $ W^u_{\text{loc}}(z^{(2)})$ is in $n$ $\Cd$ exponentially close to the projected curve to the $x_1y$-plane. We get 
\begin{equation}\label{m5secondderivative}
\frac{d^2y}{dx^2}\ge \frac32 Q_2\left(\frac{|\mu|}{|\lambda_1|^2}\right)^n+O\left(|\Delta x|\left(\frac{|\mu|}{|\lambda_1|^3}\right)^n\right)\geq Q_2\left(\frac{|\mu|}{|\lambda_1|^2}\right)^n,
\end{equation}
where we used (\ref{eq3}). 
Observe that $W^u_{\text{loc}}(m_5)$ is a graph over an $x$-direction with second derivative at $m_5$ given by (\ref{m5secondderivative}). By (\ref{eq1}) and (\ref{eq3}), the slope satisfies
\begin{equation}\label{eq5}
\frac{dy}{dx}=O\left(\left(\frac{1}{|\lambda_1|^{\theta}|\mu|}\right)^n\right).
\end{equation}
Consider now the same curve as a graph over the $y$-axis. Then, by (\ref{m5secondderivative}), (\ref{eq5}) and using $d^2x/dy^2=-1/\left(dy/dx\right)^3d^2y/dx^2$, we have
$$
\frac{d^2x}{dy^2}\geq C\left(\frac{|\mu|^4}{|\lambda_1|^{2-3\theta}}\right)^n.
$$
The map $F^N_{t,a}$ will preserve this order of curvature. As final remark observe that, by (\ref{thetacond}), ${|\mu|^4}/{|\lambda_1|^{2-3\theta}}=\left({|\mu|}/{|\lambda_1|^{2-\theta}}\right){|\mu|^3}{|\lambda_1|^{2\theta}}>1$.
\end{proof}
\begin{figure}
\centering
\includegraphics[width=0.9\textwidth]{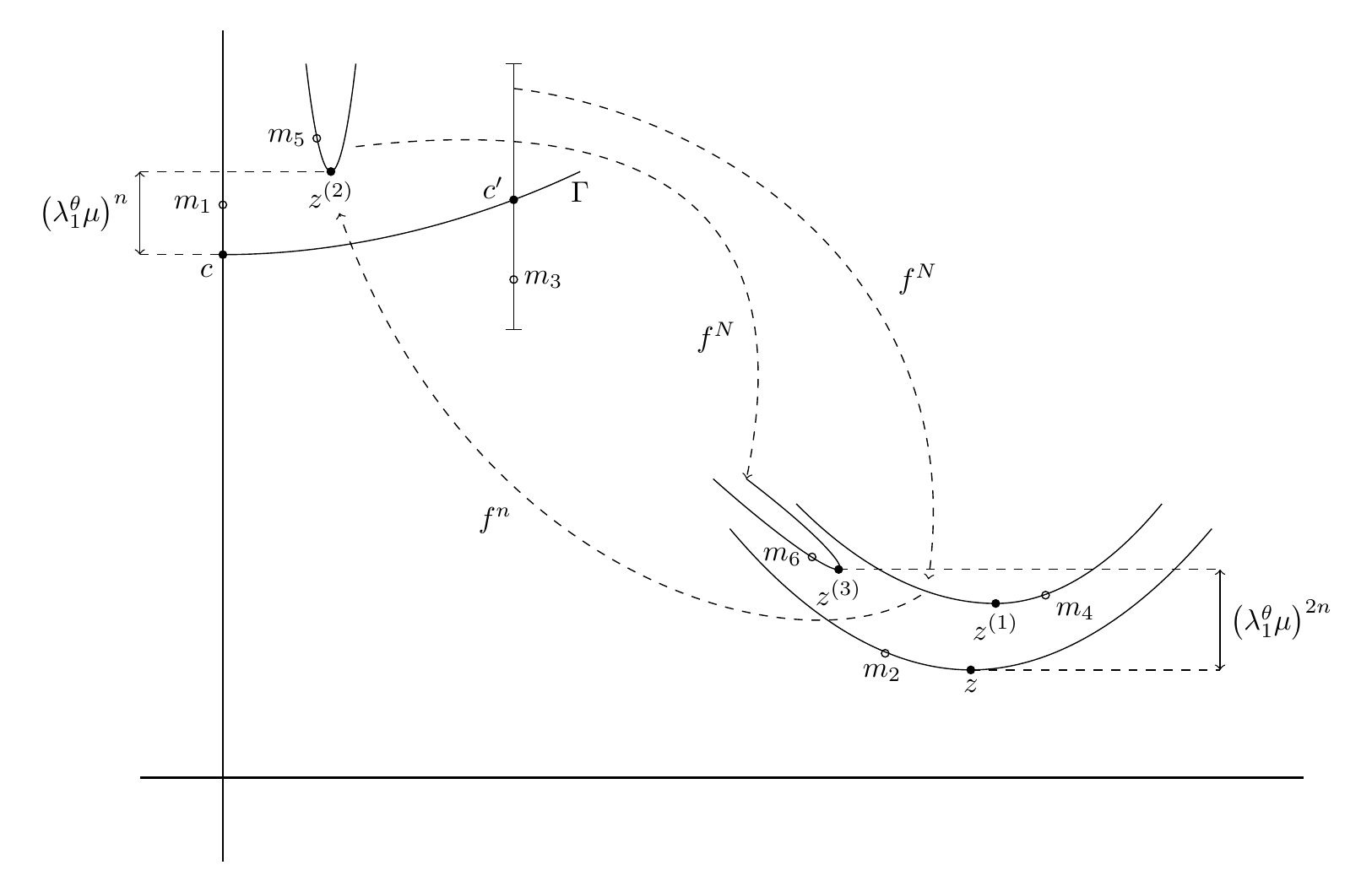}
\caption{$2N+n$ iterates}
\label{Fig4b}
\end{figure}
In the next proposition we estimate the variation of $z_{y}^{(3)}$ with respect to the parameters $a$ and $t$. This gives the speed of $W^u_{\text{\rm loc}}(z^{(3)})$ in the phase space. In order to achieve this we take a vector at a point $m_1\in W^u_{\text{\rm loc}}\left(c\right)$ and we calculate its variation along its orbit at crucial moments, step by step, until we get to a vector at the point $F^{3N+\theta n+n}(m_1)\in W^u_{\text{\rm loc}}(z^{(3)})$.  The steps will be defined precisely in the proof of the proposition. The reader can refer to Figure \ref{Fig4b}. 

A crucial application of Proposition \ref{speed} is the speed of 
$W^u_{\text{\rm loc}}(z^{(3)})$ when varying the parameters along the curve $A_n$. The dominating terms in the estimates for the partial derivatives will cancel in the calculation of this speed, see (\ref{upperineq}). This is the reason why we can not suppress more terms in the Taylor development. This cancellation is also crucial in the proof of Proposition \ref{angle}.


\begin{prop}\label{speed}
For $n\ge 1$ large enough 
\begin{eqnarray*}
\frac{\partial z_{y}^{(3)}}{\partial t}&=&D_5K \left(\lambda_1^{\theta}\mu\right)^n  \frac{n}{\mu}\frac{\partial\mu}{\partial t}+ D_5Kn \left(\lambda_1^{\theta}\mu\right)^{2n}\left[C_3X_2\frac{\theta }{\lambda_1} \frac{\partial \lambda_1}{\partial t}+\frac{K}{\mu}\frac{\partial\mu}{\partial t}\right]  +O\left(\left(|\lambda_1|^{\theta}|\mu|\right)^{2n}\right),\\
 \frac{\partial z_{y}^{(3)}}{\partial a}&=&D_5K\left(\lambda_1^{\theta}\mu^2\right)^n +1+ D_5K \left(\lambda_1^{\theta}\mu\right)^n  \frac{n}{\mu}\frac{\partial\mu}{\partial a}+D_5Kn \left(\lambda_1^{\theta}\mu\right)^{2n}\left[C_3X_2\frac{\theta }{\lambda_1} \frac{\partial \lambda_1}{\partial a}+\frac{K}{\mu}\frac{\partial\mu}{\partial a}\right] \\&+& O\left(\left(|\lambda_1|^{\theta}|\mu|\right)^{2n}\right),
\end{eqnarray*}
where $|\lambda_1|^{\theta}|\mu|^2>1$, $C_3X_2, K,D_5\neq 0$ and they converge to a non zero limit as $n\to\infty$.
\end{prop}
\begin{proof}
Choose $ t_1\in[-t_0,t_0]$. Observe that, by Lemma \ref{Nderivatives}, the definition of unfolding, by using coordinates $(x,y, t, a)$ centered around $(0,1, t_1, 0)$ in the domain and coordinates centered around $F_{t_1,0}^N(0, 1)$ we get the following expression for $F_{t,a}^N(x, y)$, 
\begin{eqnarray}\label{FNexpr}
F_{t,a}^N(x, y)&=&\left(
\begin{matrix}
Ax + By+tE+ aF\\
Cx+Q y^2+a\left[1+x\varphi_x +y^2\varphi_y\right] 
\end{matrix}
\right),
\end{eqnarray}
where the uniformly bounded coefficients satisfy
\begin{itemize}
\item[-]$A(x,y,t,a)$ is a $(n-1)\times (n-1)$ matrix valued $\Ct$ function,
\item[-]$B(x,y,t,a),E(x,y,t,a),F(x,y,t,a)$ are $n-1$ vector valued $\Ct$ functions,
 \item[-] $C(x,y,t,a)$ is a $1\times (n-1)$ matrix valued $\Ct$ function,
\item[-] $Q(x,y,t,a),\varphi_x(x,y,t,a),\varphi_y(x,y,t,a)$ are $\Ct$ functions.
\end{itemize} Moreover $B,Q\neq 0$. The $x$ component of $F_{t,a}^{N}(x,y)$ is a general expression and the fact that $B\neq 0$ follows from $(f5)$. For the $y$ component of $F_{t,a}^{N}(x,y)$ observe that the linear terms in $t$ and $y$ are absent because in $a=0$, the point $F^N_{t,0}(0,1)$ is a non-degenerate homoclinic tangency, see Remark \ref{highzya}. The $a$-dependence of the second component of $F^N_{t,a}$ follows from Remark \ref{highzya}.
\\
Recall that $c'\in F^{\theta n}_{t,a}\left(W^u_{\text{loc}}(z)\right)\cap\Gamma$, $z^{(1)}$ is defined as the lowest point of $F^N_{t,a}\left(W^u_{\text{loc}}(c')\right)$, $z^{(2)}$ as the lowest point of $F^n_{t,a}\left(W^u_{\text{loc}}(z^{(1)})\right)$ and $z^{(3)}$ as the lowest point of $F^N_{t,a}\left(W^u_{\text{loc}}(z^{(2)})\right)$.
We fix the points $m_6, m_5,m_4, m_3, m_2, m_1$ so that they satisfy the following, see Figure \ref{Fig4b}. 
\begin{itemize}
\item[-] $m_6\in W^u_{\text{loc}}(z^{(3)})$ sufficiently close to $z^{(3)}$.
\item[-] $m_5=(m_{5,x},m_{5,y})\in W^u_{\text{loc}}(z^{(2)})$, $F_{t,a}^{-N}(m_6)=m_5$. By Lemma \ref{ymaxymin} and Remark \ref{muzeroothers}, there exists a uniform constant $K_5>0$ such that, using that $a\in\mathcal B_n$ and for $\epsilon_1$ small enough,
\begin{equation}\label{m5}
\frac{1}{K_5}\left(|\lambda_1|^{\theta}|\mu|\right)^n\leq |m_{5,y}- c_{n,y}|\leq K_5\left(|\lambda_1|^{\theta}|\mu|\right)^n,
\end{equation}
\begin{equation}\label{m5x}
|m_{5,x}|=O\left(|\lambda_1|^{n}\right),
\end{equation}
and by (\ref{eq3})
\begin{equation}\label{m5x1}
|m_{5,x}-z^{(2)}_x|=O\left(\left(\frac{|\lambda_1|^{2-\theta}}{|\mu|^2}\right)^n\right).
\end{equation}
\item[-] $m_4=(m_{4,x},m_{4,y})\in W^u_{\text{loc}}(z^{(1)})$, $F_{t,a}^{-n}(m_5)=m_4$. By Lemma \ref{z1distz} and the fact that $z_{n,y}$ can be estimated by $|\mu|^{-n}$, see Remark \ref{muzeroothers},
\begin{equation}\label{m4}
|m_{4,y}|=O\left(\frac{1}{|\mu|^n}\right),
\end{equation}
and by (\ref{m5x1}), using the linear map backward and $(F5)$ (which ensures that the vector $m_{4,x}-z^{(1)}_x$ has a dominating first component) we get, 
\begin{equation}\label{m41}
|m_{4,x}-z^{(1)}_x|=O\left(\left(\frac{|\lambda_1|^{1-\theta}}{|\mu|^2}\right)^n\right).
\end{equation}
\item[-] $m_3=(m_{3,x},m_{3,y})\in W^u_{\text{loc}}(c')$, $F_{t,a}^{-N}(m_4)=m_3$. By Lemma \ref{z1distz} (see Figure \ref{Fig4a}) and (\ref{m41}),
\begin{eqnarray}\label{m3} 
|m_{3}-c|&=&O\left({|\lambda_1|^{\theta{n}}}\right),\\
|m_{3}-c'|&=&O\left(\left(\frac{|\lambda_1|^{1-\theta}}{|\mu|^2}\right)^n\right).
\end{eqnarray}
\item[-] $m_2=(m_{2,x},m_{2,y})\in W^u_{\text{loc}}(z)$, $F_{t,a}^{-\theta n}(m_3)=m_2$. Therefore, there exists a uniform constant $K_2>0$ such that
\begin{equation}\label{m2}
\frac{1}{K_2}\frac{1}{|\mu|^{\theta{n}}}\leq |m_{2,y}|\leq K_2\frac{1}{|\mu|^{\theta{n}}}.
\end{equation}
\item[-] $m_1=(m_{1,x},m_{1,y})\in W^u_{\text{loc}}(c)$, $F_{t,a}^{-N}(m_2)=m_1$. By (\ref{m2}) and because of the quadratic behavior, there exists a uniform constant $K_1>0$ such that
\begin{equation}\label{m1}
\frac{1}{K_1}\frac{1}{|\mu|^{\theta\frac{n}{2}}}\leq |m_{1,y}- c_{y}|\leq K_1\frac{1}{|\mu|^{\theta\frac{n}{2}}}.
\end{equation}
\end{itemize}
Let us recall that, $m_2=F^N(m_1)$, $m_3=F^{\theta n}(m_2)$,  $m_4=F^N(m_3)$,  $m_5=F^n(m_4)$, and  $m_6=F^N(m_5)$. 
Moreover, $t_1\in[-t_0,t_0]$ was chosen arbitrary. Take $(t_1+t,a)\in\mathcal B_n$ and consider the $\Cq$ map $$(\Delta t, \Delta a, \Delta y)\longmapsto F^{3N+(\theta+1)n}_{t+\Delta t, a+\Delta a}{( m_1+(0,\Delta y))}=m_6+(\Delta m_{6,x}, \Delta m_{6,y}).$$  We are interested in the partial derivatives of $m_6$. 
Observe that for $i=1,3,5$,
\begin{eqnarray}\label{mi+1deltami+1}
d m_{i+1}&=&DF^N_{t,a}\left(m_{i}\right)d m_i+\frac{\partial F^N_{t,a}\left(m_{i}\right)}{\partial t}dt+\frac{\partial F^N_{t,a}\left(m_{i}\right)}{\partial a}da,
\end{eqnarray}
for $i=4$,
\begin{eqnarray}\label{mi+1deltami+1fn}
d m_{i+1}&=&DF^n_{t,a}\left(m_{i}\right) dm_i+\frac{\partial F^n_{t,a}\left(m_{i}\right)}{\partial t}dt+\frac{\partial F^n_{t,a}\left(m_{i}\right)}{\partial a}da,
\end{eqnarray}
and for $i=2$,
\begin{eqnarray}\label{mi+1deltami+1fthetan}
dm_{i+1}&=&DF^{\theta n}_{t,a}\left(m_{i}\right)dm_i+\frac{\partial F^{\theta n}_{t,a}\left(m_{i}\right)}{\partial t}dt+\frac{\partial F^{\theta n}_{t,a}\left(m_{i}\right)}{\partial a}da.
\end{eqnarray}
All partial derivatives in (\ref{mi+1deltami+1}) are uniformly bounded. However we will need more careful estimate for $\Delta m_{i+1,y}$. Namely, for $i=1,3,5$,
\begin{eqnarray*}
DF^N_{t,a}\left(m_{i}\right)&=&\left(\begin{matrix}
A_i&B_i\\
C_i&O\left(|m_{i,x}|\right)+D_i (m_{i,y}-1)
\end{matrix}\right),
\end{eqnarray*}
which follows from (\ref{FNexpr}). Similarly from (\ref{FNexpr}) one obtains
\begin{eqnarray*}
\frac{\partial m_{i+1,y}}{\partial t}&=&\frac{\partial C}{\partial t}m_{i,x}+\frac{\partial Q}{\partial t}\left(m_{i,x}-1\right)^2+a m_{i,x}\frac{\partial\varphi_x}{\partial t}+a \left(m_{i,y}-1\right)^2\frac{\partial\varphi_y}{\partial t}\\
&=&d_i \left| m_{i,y}-1\right |^2+O\left(|m_{i,x}|\right),
\end{eqnarray*}
where $d_i\neq 0$ is uniformly bounded away from zero. By differentiation of (\ref{FNexpr}) with respect of $\Delta a$ we obtain
\begin{eqnarray*}
\frac{\partial m_{i+1,y}}{\partial a}&=&1+O\left(\left| m_{i,y}-1\right |^2\right)+O\left(|m_{i,x}|\right).
\end{eqnarray*}
For $i=3$ we refine the estimate 
\begin{eqnarray*}
DF^N_{t,a}\left(m_{3}\right)&=&\left(\begin{matrix}
A_3&B_3\\
C_3&O\left(|m_{3}-c'|\right)
\end{matrix}\right),
\end{eqnarray*}
where we estimate $DF^N(m_3)$ with $DF^N(c')$ evaluated at a point at distance $|m_{3}-c'|$. For $i=1,5$ we get in an analogous way

\begin{eqnarray}\label{DFN}
d m_{i+1}&=&\left(\begin{matrix}
A_i&B_i\\
C_i&O\left(|m_{i,x}|\right)+D_i (m_{i,y}-1)
\end{matrix}\right)d m_i\\&+&\left(\begin{matrix}
O\left(d t\right)+O\left(d a\right)\\
\left[d_i |m_{i,y}-1|^2+O\left(|m_{i,x}|\right)\right]d t+\left[1+O\left(|m_{i,y}-1|^2+|m_{i,x}|\right)\right]d a
\end{matrix}\right),\nonumber
\end{eqnarray}
and  
\begin{eqnarray}\label{DFN3}
d m_{4}&=&\left(\begin{matrix}
A_3&B_3\\
C_3& O\left(|m_{3}-c'|\right)
\end{matrix}\right)d m_3\\&+&\left(\begin{matrix}
O\left(d t\right)+O\left(d a\right)\\
\left[d_3 |m_{3,y}-1|^2+O\left(|m_{3,x}|\right)\right]d t+\left[1+O\left(|m_{3,y}-1|^2+|m_{3,x}|\right)\right]d a
\end{matrix}\right),\nonumber
\end{eqnarray}
where $D_i,B_i,C_i\neq 0$ (because $q_1$ is a non degenerate tangency in general direction and $DF^N(0,1)$ is non singular). From $d m_4$ to $d m_5$ we use the linear map $F^n$. From (\ref{mi+1deltami+1fn}) and using that $m_{5,y}=\mu^n m_{4,y}$ we get 
\begin{eqnarray}\label{DFn}
d m_{5}&=&\left(\begin{matrix}
\lambda^{n}&0\\
0&\mu^{n}
\end{matrix}\right)d m_4+\left(\begin{matrix}
X_4\frac{n\lambda_1^n}{\lambda_1}\left[ \frac{\partial \lambda_1}{\partial t} d t+\frac{\partial \lambda_1}{\partial a}d a\right]\\
m_{5,y}\frac{n}{\mu}\left[\frac{\partial \mu}{\partial t}d t+\frac{\partial \mu}{\partial a}d a \right]
\end{matrix}\right),
\end{eqnarray}
with $0\neq X_4=m_{4,x}\approx \left(\begin{matrix}
1&0&\hdots&0
\end{matrix} \right)$ and the diagonal matrix $\lambda=(\lambda_1, \lambda_2,\cdots,\lambda_{m-1})$ has the stable eigenvalues along the diagonal. 
A similar formula holds going from $d m_2$ to $d m_3$ where we use the linear map $F^{\theta n}$, (\ref{mi+1deltami+1fthetan}) and the fact that $m_{3,y}=\mu^{\theta n} m_{2,y}$. Namely,
\begin{eqnarray}\label{DFthetan}
d m_{3}&=&\left(\begin{matrix}
\lambda^{\theta n}&0\\
0&\mu^{\theta n}
\end{matrix}\right)d m_2+\left(\begin{matrix}
X_2\frac{\theta n\lambda_1^{\theta n}}{\lambda_1}\left[ \frac{\partial \lambda_1}{\partial t} d t+\frac{\partial \lambda_1}{\partial a}d a\right]\\
m_{3,y}\frac{\theta n}{\mu}\left[\frac{\partial \mu}{\partial t}d t+\frac{\partial \mu}{\partial a}d a \right]
\end{matrix}\right)
\end{eqnarray}
where  $0\neq X_2=m_{2,x}\approx \left(\begin{matrix}
1&0&\hdots&0
\end{matrix} \right)$.
\\
Observe that $m_{1,x}=0$ and recall that $\left(d m_{1,x},d m_{1,y}\right)=\left(0,d y\right)$. By (\ref{DFN}) and by (\ref{m1}),
\begin{eqnarray*}
\left(
\begin{matrix}
d m_{2,x}\\
d m_{2,y}
\end{matrix}\right)=\left(\begin{matrix}
A_1&B_1\\
C_1&D_1\tilde K_1{|\mu|^{\frac{-\theta{n}}{2}}}
\end{matrix}\right)\left(
\begin{matrix}
0\\
d y
\end{matrix}\right)
+\left(\begin{matrix}
O\left(d t\right)+O\left(d a\right)\\
\left[d_1 \tilde K_1^2{|\mu|^{{-\theta{n}}}}\right]d t+\left[1+O\left({|\mu|^{{-\theta{n}}}}\right)\right]d a
\end{matrix}\right),
\end{eqnarray*}
where $\tilde K_1$ is the constant which gives an equality in (\ref{m1}). As a consequence 
\begin{eqnarray*}
d m_{2,x}&=&O\left(d t\right) +O\left(d a\right)+O\left(d y\right),\\
d m_{2,y}&=&\frac{\tilde d_1}{|\mu|^{\theta {n}}}d t+\left[1+O\left(\frac{1}{|\mu|^{\theta{n}}}\right)\right]d a+\frac{\tilde D_1}{|\mu|^{\theta\frac{n}{2}}}d y,
\end{eqnarray*}
where $\tilde d_1=d_1 \tilde K_1^2$ and $\tilde D_1=D_1\tilde K_1$. By using the fact that $F^{\theta n}_{t,a}$ is linear and using (\ref{DFthetan}),
\begin{eqnarray*}
\left(\begin{matrix}
d m_{3,x}\\ d m_{3,y}
\end{matrix}
\right)&=&\left(\begin{matrix}
\lambda^{\theta n}&0\\
0&\mu^{\theta n}
\end{matrix}\right)\left(\begin{matrix}
d m_{2,x}\\ d m_{2,y}
\end{matrix}
\right)+\left(\begin{matrix}
X_2\frac{\theta n\lambda_1^{\theta n}}{\lambda_1}\left[ \frac{\partial \lambda_1}{\partial t} d t+\frac{\partial \lambda_1}{\partial a}d a\right]\\
\left(\left(m_{3,y}-c_{n,y}\right)+c_{n,y}\right)\frac{\theta n}{\mu}\left[\frac{\partial \mu}{\partial t}d t+\frac{\partial \mu}{\partial a}d a \right]
\end{matrix}\right),
\end{eqnarray*}
 and using $(\ref{m3})$ and Remark \ref{highzya} (recall that $c$ is fixed at height $1$, i.e. $c_{n,y}=1$)
\begin{eqnarray*} 
d m_{3,x}&=&X_2\frac{\theta n\lambda_1^{\theta n}}{\lambda_1}\left[ \frac{\partial \lambda_1}{\partial t} d t+\frac{\partial \lambda_1}{\partial a}d a\right]+O\left(|\lambda_1|^{\theta n}d y\right)\\
&+&O\left(|\lambda_1|^{\theta n}d t\right)+O\left(|\lambda_1|^{\theta n}d a\right),\\
d m_{3,y}&=&\left[\frac{\theta n}{\mu}\frac{\partial\mu}{\partial t}+\tilde d_1+O\left(n|\lambda_1|^{\theta n}\right)\right]d t\\&+&\left[|\mu|^{\theta n}+\frac{\theta n}{\mu}\frac{\partial\mu}{\partial a}+O\left(1\right)\right]d a+\tilde D_1{|\mu|^{\theta\frac{n}{2}}}d y,
\end{eqnarray*}
where $X_2\neq 0$ is pointing in the direction of $e_1$, see Figure \ref{Fig4a}.
By (\ref{DFN3}) and (\ref{m3}) (observe that $\left\{|m_{3,y}-c_{n,y}|,|m_{3,x}|\right\}=O\left(|\lambda_1|^{\theta n}\right)$), we get
\begin{eqnarray*}
\left(
\begin{matrix}
d m_{4,x}\\
d m_{4,y}
\end{matrix}\right)=\left(\begin{matrix}
A_3&B_3\\
C_3&O\left(\left(\frac{|\lambda_1|^{1-\theta}}{|\mu|^2}\right)^n\right)
\end{matrix}\right)\left(
\begin{matrix}
d m_{3,x}\\
d m_{3,y}
\end{matrix}\right)
+\left(\begin{matrix}
O\left(d t\right)+O\left(d a\right)\\
O\left(|\lambda_1|^{\theta n}\right)d t+\left[1+O\left(|\lambda_1|^{\theta n}\right)\right]d a
\end{matrix}\right).
\end{eqnarray*}
As consequence
\begin{eqnarray*}
d m_{4,x}&=&\left[B_{3}\frac{\theta n}{\mu}\frac{\partial\mu}{\partial t}+O\left(1\right)\right]d t\\&+&\left[B_{3}|\mu|^{\theta n}+B_{3}\frac{\theta n}{\mu}\frac{\partial\mu}{\partial a}+O\left(1\right)\right]d a\\&+&\left[B_{3}\tilde D_1{|\mu|^{\theta\frac{n}{2}}}+O\left({|\lambda_1|^{\theta{n}}}\right)\right]d y,\\
d m_{4,y}&=&\left[C_3X_2\frac{\theta n\lambda_1^{\theta n}}{\lambda_1} \frac{\partial \lambda_1}{\partial t} + O\left(|\lambda_1|^{\theta n}\right)\right]d t \\&+&\left[1+C_3X_2\frac{\theta n\lambda_1^{\theta n}}{\lambda_1} \frac{\partial \lambda_1}{\partial a} + O\left(|\lambda_1|^{\theta n}\right)\right]d a\\&+&O\left(|\lambda_1 |^{\theta{n}}d y \right),
\end{eqnarray*}
where $C_3X_2\neq 0$, see $(f5)$. By (\ref{DFn}) we get 
\begin{eqnarray*}
d m_{5,x}&=&\left[B_{3}\frac{\theta n \lambda_1^n}{\mu}\frac{\partial\mu}{\partial t}+X_4\frac{n\lambda_1^{n}}{\lambda_1} \frac{\partial \lambda_1}{\partial t}+O\left(|\lambda_1|^n\right)\right]d t\\&+&\left[B_3\left(\lambda_1\mu^{\theta}\right)^n+B_{3}\frac{\theta n \lambda_1^n}{\mu}\frac{\partial\mu}{\partial a}+X_4\frac{n\lambda_1^{n}}{\lambda_1} \frac{\partial \lambda_1}{\partial a}+O\left(|\lambda_1|^n\right)\right]d a\\&+&\left[B_{3}\tilde D_1\left(\lambda_1\mu^{\frac{\theta}{2}}\right)^n+O\left(\left({|\lambda_1|^{\theta+1}}\right)^n\right)\right]d y,\\
d m_{5,y}&=& \left[\frac{n}{\mu}\frac{\partial\mu}{\partial t}+n\left(\lambda_1^{\theta}\mu\right)^n\left[C_3X_2\frac{\theta }{\lambda_1} \frac{\partial \lambda_1}{\partial t}+\frac{K}{\mu}\frac{\partial\mu}{\partial t}\right] +O\left(\left(|\lambda_1|^{\theta}|\mu|\right)^n\right)\right] d t\\&+&
\left[\mu^n+\frac{n}{\mu}\frac{\partial\mu}{\partial a}+n\left(\lambda_1^{\theta}\mu\right)^n\left[C_3X_2\frac{\theta }{\lambda_1} \frac{\partial \lambda_1}{\partial a}+\frac{K}{\mu}\frac{\partial\mu}{\partial a}\right] +O\left(\left(|\lambda_1|^{\theta}|\mu|\right)^n\right)\right] d a\\&+&
O\left(\left(|\lambda_1 |^{\theta}|\mu|\right)^{{n}}d y \right),
\end{eqnarray*}
where we used that $m_{5,y}=(m_{5,y}-c_{n,y})+c_{n,y}$ and the point $c$ is at height $1$, see Remark \ref{highzya}. Moreover $K$ is the constant giving equality in (\ref{m5}).
 By (\ref{DFN}), (\ref{m5}) and (\ref{m5x})  we get 
 \begin{eqnarray*}
\left(
\begin{matrix}
d m_{6,x}\\
d m_{6,y}
\end{matrix}\right)&=&\left(\begin{matrix}
A_5&B_5\\
C_5&O\left(|\lambda_1|^{n}\right)+D_5K\left(|\lambda_1 |^{\theta}|\mu|\right)^{{n}}
\end{matrix}\right)\left(
\begin{matrix}
d m_{5,x}\\
d m_{5,y}
\end{matrix}\right)\\
&+&\left(\begin{matrix}
O\left(d t\right)+O\left(d a\right)\\
\left[d_5K\left(|\lambda_1 |^{\theta}|\mu|\right)^{{2n}}+O\left(|\lambda_1|^{ n}\right)\right]d t+\left[1+O\left(\left(|\lambda_1 |^{\theta}|\mu|\right)^{{2n}}\right)\right]d a
\end{matrix}\right),
\end{eqnarray*}
where $K$ is the constant giving equality in (\ref{m5}).
As consequence

\begin{eqnarray}\label{dm6x}
d m_{6,x}&=& \left [B_5\frac{n}{\mu}\frac{\partial\mu}{\partial t} +O\left(1\right)\right] d t+
\left[B_5\mu^n+B_5\frac{n}{\mu}\frac{\partial\mu}{\partial a}+O\left(1\right)\right]d a\\&+&O\left(\left(|\lambda_1 |^{\theta}|\mu|\right)^{{n}}d y \right),\nonumber
\end{eqnarray}
\begin{eqnarray}\label{dm6y}
d m_{6,y}&=&\nonumber\left[D_5K \left(\lambda_1^{\theta}\mu\right)^n  \frac{n}{\mu}\frac{\partial\mu}{\partial t}+ D_5Kn \left(\lambda_1^{\theta}\mu\right)^{2n}\left[C_3X_2\frac{\theta }{\lambda_1} \frac{\partial \lambda_1}{\partial t}+\frac{K}{\mu}\frac{\partial\mu}{\partial t}\right]\right.
\\ &+& 
\left.O\left(\left(|\lambda_1|^{\theta}|\mu|\right)^{2n}\right)\right]d t
\\\nonumber&+&\left[D_5K\left(\lambda_1^{\theta}\mu^2\right)^n +1+ D_5K \left(\lambda_1^{\theta}\mu\right)^n  \frac{n}{\mu}\frac{\partial\mu}{\partial a}\right.
\\\label{dm6yda} &+&\left. D_5Kn \left(\lambda_1^{\theta}\mu\right)^{2n}\left[C_3X_2\frac{\theta }{\lambda_1} \frac{\partial \lambda_1}{\partial a}+\frac{K}{\mu}\frac{\partial\mu}{\partial a}\right] +O\left(\left(|\lambda_1|^{\theta}|\mu|\right)^{2n}\right)\right]d a\\\nonumber &+&O\left(\left(|\lambda_1 |^{\theta}|\mu|\right)^{{2n}} \right)d y,
\end{eqnarray}
where $D_5, B_{5}, K\ne 0$ and $|\lambda_1|^{\theta}|\mu|^2>1$ (see (\ref{thetacond})). Refer to Figure \ref{Fig4b}. 
The formula for $m_{6,y}$ holds in general for all $\Delta t,\Delta a$ and $\Delta y$. However at the point $(t_1,a)$ when $\Delta t=\Delta a=0$ we have ${\partial m_{6,y}}/{\partial y}=0$. Hence, the Taylor polynomial of second order for $\Delta m_{6,y}$ does not contain a linear term in $\Delta y$. As consequence ${\partial z^{(3)}_{y}}/{\partial t}={\partial m_{6,y}}/{\partial t}$ and ${\partial z^{(3)}_{y}}/{\partial a}={\partial m_{6,y}}/{\partial a}$. 
It is left to prove that $C_3X_2, K,D_5$ converge. Observe that $D_5$ and $C_3$ converge because they are part of the derivative $DF^N$ which converges to $DF^N(0,1)$. Moreover $\lim_{n\to \infty} X_2=\lim_{n\to \infty}\lambda_1^{-\theta n}F^{\theta n}(q_1)$ which is not zero because $q_1$ is in general position, see $(f4)$. Finally $K$ converges because of Lemma \ref{z1distz}. The proposition follows.
\end{proof}

%
%



Consider the transversal homoclinic intersection $q_2$ defined in $(f6)$. Let $W=W^s_{\text{loc}}(q_2)$ and for all $n\in\N$ let $W_n=F^{-n}_{t,a}(W)$. Because $W\pitchfork W^u(p)$ we can apply the $\lambda$-Lemma which implies that $W_n$ converges to $W^s_{\text{loc}}(0)$. In particular $W_n$ is the graph of a function which will also be denoted by $W_n$.  Moreover, because 
\begin{equation}\label{wnformule}
W_n(x)=\mu^{-n}W\left(\lambda_1^nx_1,\lambda_2^nx_2,\dots,\lambda_{m-1}^nx_{m-1}\right),
\end{equation}
 then
\begin{equation}\label{wn}
\frac{1}{2\mu}|\mu|^{-n}\leq |W_n|\leq 2 |\mu|^{-n},
\end{equation}
\begin{equation}\label{dwn}
\left|\frac{\partial W_n}{\partial x}\right|=O\left(\frac{|\lambda_1|}{|\mu|}\right)^{n},
\end{equation}
and
\begin{equation}\label{ddwn}
\left|\frac{\partial^2 W_n}{\partial x^2}\right|=O\left(\frac{|\lambda_1|^2}{|\mu|}\right)^{n}.
\end{equation}
We estimate now the speed of the stable manifold $W_n$ as the graph of a function  in the phase space. The comparison with the speed of $W^u_\text{\rm loc}(z^{(3)})$ will give us a new  tangency.
\begin{lem}\label{partialwn}
$$\frac{\partial W_n}{\partial t}=-\frac{n}{\mu}\frac{\partial\mu}{\partial t}W_n+\sum_i\left(\frac{\lambda_{i}}{\mu}\right)^n\frac{n}{\lambda_i}\frac{\partial\lambda_{i}}{\partial t}\frac{\partial W}{\partial x_i}x_i+\frac{1}{\mu^n}\frac{\partial W}{\partial t}=O\left(\frac{n}{|\mu|^n}\right),$$
$$\frac{\partial W_n}{\partial a}=-\frac{n}{\mu}\frac{\partial\mu}{\partial a}W_n+\sum_i\left(\frac{\lambda_{i}}{\mu}\right)^n\frac{n}{\lambda_i}\frac{\partial\lambda_{i}}{\partial a}\frac{\partial W}{\partial x_i}x_i+\frac{1}{\mu^n}\frac{\partial W}{\partial a}=O\left(\frac{n}{|\mu|^n}\right).$$
\end{lem}
\begin{proof}
Fix a point $x=(x_i)\in [-2,2]^{m-1}$ and a parameter $(t,a)\in [-r_0,r_0]^2$. We denote by $W_n$ the manifold corresponding to $(t,a)$ and by $W_n+\Delta W_n$ the manifold to $(t,a+\Delta a)$. For all $i=1,\dots, m-1$ let $\Delta\lambda_i={\partial\lambda_i}/{\partial a} \text{ } \Delta a$ and $\Delta\mu={\partial\mu}/{\partial a}\text{ } \Delta a$. Then, because the maps $F^n_{t,a}$ are linear, by differentiating (\ref{wnformule}) we obtain
$$
\left(\mu+\Delta\mu\right)^n\left(W_n+\Delta W_n\right)(x)=\mu^n W_n(x)+\frac{\partial W}{\partial x}\left(\left(\left(\lambda_i+\Delta\lambda_i\right)^n-\lambda_i^n\right)x_i\right)+\frac{\partial W}{\partial a}\Delta a.
$$
Similarly, one gets the same bound for ${\partial W_n}/{\partial t}$.
\end{proof}
A proof similar to that of the previous lemma gives the following.
\begin{lem}\label{partialgamman}
Let $\Gamma_n$ be as in (\ref{defgamma}), then 
$$\frac{\partial\Gamma_n}{\partial t}=-\frac{n}{\mu}\frac{\partial\mu}{\partial t}\Gamma_n+\sum_i\left(\frac{\lambda_{i}}{\mu}\right)^n\frac{n}{\lambda_i}\frac{\partial\lambda_{i}}{\partial t}\frac{\partial \Gamma}{\partial x_i}x_i+\frac{1}{\mu^n}\frac{\partial \Gamma}{\partial t}=O\left(\frac{n}{|\mu|^n}\right),$$
$$\frac{\partial \Gamma_n}{\partial a}=-\frac{n}{\mu}\frac{\partial\mu}{\partial a}\Gamma_n+\sum_i\left(\frac{\lambda_{i}}{\mu}\right)^n\frac{n}{\lambda_i}\frac{\partial\lambda_{i}}{\partial a}\frac{\partial \Gamma}{\partial x_i}x_i+\frac{1}{\mu^n}\frac{\partial \Gamma}{\partial a}=O\left(\frac{n}{|\mu|^n}\right).$$
\end{lem}
Let 
$$
S_n=\left\{(x,y)\in [-2,2]^m| y\geq W_n(x)\right\}.
$$
Observe that $z\in S_n\setminus S_{n-1}$. 
Let $n_0$ be the maximal $n\in\N$ such that 
\begin{equation}\label{nminusn0}
W^u_{\text{loc}}(z^{(3)})\subset S_{n-n_0}.
\end{equation}
Observe that $n_0$ gives information on the position of $W^u_{\text{loc}}(z^{(3)})$ with respect to the pull-backs of the stable manifold, see Figure \ref{Fig5}. A more precise estimate on the size of $n_0$ comes from the following lemma.
Let $\alpha={\log\left(|\lambda_1|^{2\theta}|\mu|^3\right)}/{\log |\mu|}$. By (\ref{thetacond1}), $\alpha\in \left(0,{1}\right)$. 
\begin{lem}\label{n0}
The integer $n_0$ satisfies the following:
$$
n_0=n\alpha+O(1).
$$ 
\end{lem}
\begin{proof}
From Lemma \ref{hnbounds}, (\ref{wn}) and the fact that $z\in S_n\setminus S_{n-1}$, there exists a uniform constant $K>0$ such that 
$$
\frac{1}{K}\left(\frac{1}{|\mu|^n}+\left(|\lambda_1|^{\theta}|\mu|\right)^{2n}\right)\leq z_{y}^{(3)}\leq K\left(\frac{1}{|\mu|^n}+\left(|\lambda_1|^{\theta}|\mu|\right)^{2n}\right).
$$
Moreover the definition of $n_0$ implies that 
$$
\frac{1}{K}\frac{1}{|\mu|^{n-n_0}}\leq  z_{y}^{(3)}\leq {K}\frac{1}{|\mu|^{n-n_0}}.
$$
The two previous inequalities imply that 
\begin{equation}\label{alphaofuse}
\frac{1}{K^2}\leq\frac{1}{|\mu|^{n_0}}\left(1+\left(|\lambda_1|^{2\theta}|\mu|^3\right)^{n}\right)\leq K^2.
\end{equation}
The lemma follows from (\ref{thetacond}).
\end{proof}
\begin{defin}
A tangency between $W^u_{\text{loc}}(z^{(3)})$ and $ W_{n-n_0}$ is called a \emph{secondary tangency of type} $n_0$. We define 
$$
T_{n,n_0}=\left\{(t,a)\in\mathcal B_n | F_{t,a}\text{ has a secondary tangency of type }n_0\right\}.
$$
\end{defin}
\begin{figure}
\centering
\includegraphics[width=0.6\textwidth]{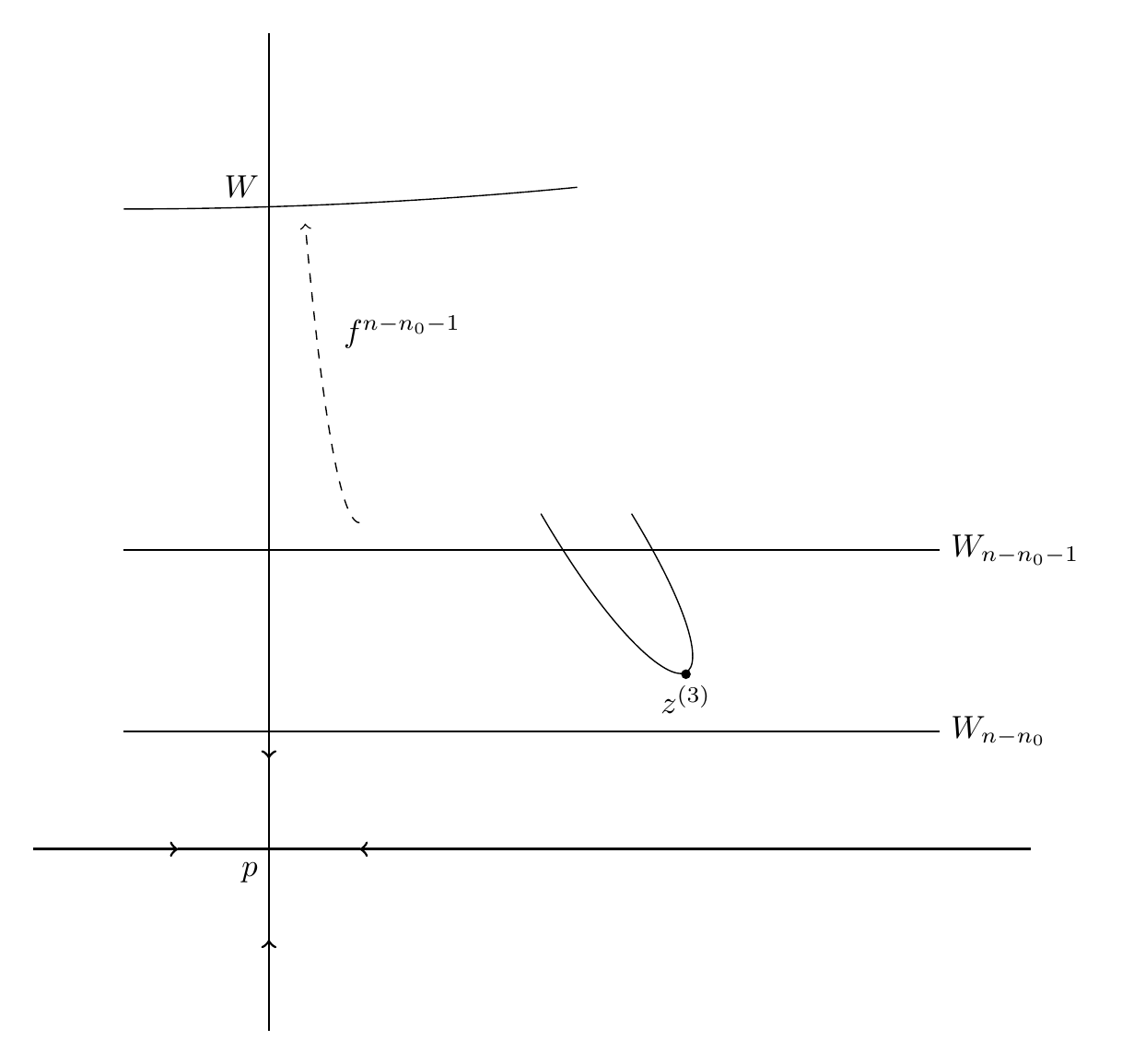}
\caption{Position of $z^{(3)}$}
\label{Fig5}
\end{figure}
In the next proposition we are going to prove that secondary tangencies exist for certain parameters in $A_n$, see Figure \ref{Fig6} and they are at distance of order $1/n$ to each other. The result is achieved by comparing the rate of speed of $W^u_{\text{loc}}(z^{(3)})$ and $ W_{n-n_0}$ when changing parameters in the phase space and it relies on the cancellation of the main term.

\begin{prop}\label{newtangency}
Let $C$ be an arbitrarly large positive constant. For all $t\in\left(-t_0,t_0\right)$ there exists $\theta\in\left(\theta_1,\theta_2\right)$ such that for $n$ large enough there exists $t_n\in\left(-t_0,t_0\right)$ such that the following holds:
\begin{itemize}
\item[-]  $|t- t_n| =O\left(\frac{1}{n}\right)$,
\item[-] $F_{t_n,  a_n(t_n)}$ has a secondary tangency $q_{1,n,n_0}$ and $\left(t_n,  a_n(t_n)\right)\in A_n\cap T_{n,n_0}$.
\end{itemize}
Moreover at the tangency point \begin{equation}\label{speedwnminusn0}
\left|\frac{\partial W_{n-n_0}}{\partial t}(q_{1,n,n_0})\right|\geq C n\left(\lambda_1^{\theta}\mu\right)^{2n}.
\end{equation}

\end{prop}

\begin{proof}
Fix $C>0$ and choose $t^*\in\left(-t_0,t_0\right)$. 
Use the notation of Proposition \ref{speed} and observe that the function 
$$
\theta\mapsto {D_5K}\left[C_3X_2\frac{\theta }{\lambda_1} \frac{\partial \lambda_1}{\partial t}+\frac{K}{\mu}\frac{\partial\mu}{\partial t}\right]
$$
is affine. From $(F2)$, we know that $\left({K}/{\mu}\right){\partial\mu}/{\partial t}\neq 0$. Hence, there exist a uniform $v>0$ and $\theta\in\left(\theta_1,\theta_2\right)$ such that for all $(t,a)$ close to $(t^*,0)$,  $$\left|{D_5K}\left[C_3X_2\frac{\theta }{\lambda_1} \frac{\partial \lambda_1}{\partial t}+\frac{K}{\mu}\frac{\partial\mu}{\partial t}\right]\right|\geq v>0.$$
 Without loss of generality we may assume that $D_5K\left[C_3X_2\frac{\theta }{\lambda_1} \frac{\partial \lambda_1}{\partial t}+\frac{K}{\mu}\frac{\partial\mu}{\partial t}\right]>v>0$. The opposite case is treated by reversing the direction of $t$. Take $n$ large enough and a point $(t, a_n(t))\in A_n$ near $(t^*,a_n(t^*))\in A_n$. By Proposition \ref{speed} and Lemma \ref{Imtang}  
\begin{eqnarray}\label{upperineq}\nonumber
z^{(3)}_y\left(t,a_n(t)\right)&=&z^{(3)}_y\left(t^*,a_n(t^*)\right)+\int_0^{(t- t^*)}\left[\frac{\partial z^{(3)}}{\partial t}+\frac{\partial z^{(3)}}{\partial a}\frac{d a_n}{d t}\right]dt=z^{(3)}_y\left(t^*,a_n(t^*)\right)\\\nonumber &+&\left\{D_5K \left(\lambda_1^{\theta}\mu\right)^n  \frac{n}{\mu}\frac{\partial\mu}{\partial t}+ D_5Kn \left(\lambda_1^{\theta}\mu\right)^{2n}\left[C_3X_2\frac{\theta }{\lambda_1} \frac{\partial \lambda_1}{\partial t}+\frac{K}{\mu}\frac{\partial\mu}{\partial t}\right]\right.\\\nonumber &-&\left.D_5K \left(\lambda_1^{\theta}\mu\right)^n  \frac{n}{\mu}\frac{\partial\mu}{\partial t} +O\left(\left(|\lambda_1|^{\theta}|\mu|\right)^{2n}\right)\right\}(t-t^*)
\\&\geq &z^{(3)}_y\left(t^*,a_n(t^*)\right)+n\left(\lambda_1^{\theta}\mu\right)^{2n}\frac{v}{2}(t-t^*),
\end{eqnarray}
where we used that $\lambda_1^{2\theta}\mu^3>1$ and $\lambda_1^{\theta}\mu^2>1$  see (\ref{thetacond}). Observe that we have a cancellation of the dominant terms in the partial derivatives obtained by combining Proposition \ref{speed} and Lemma \ref{Imtang}.
Let $n^*_0$ be such that $W^u_{\text{loc}}\left(z^{(3)}\left(t^*,a_n(t^*)\right)\right)\subset S_{n-n_0^*}$, then, by (\ref{wn}), 
\begin{equation}\label{ztildeheight}
z^{(3)}_y\left(t^*,a_n(t^*)\right)\geq \frac{2}{\mu}\frac{1}{\mu^{n-n^*_0}}.
\end{equation}
By (\ref{upperineq}) and (\ref{ztildeheight})
\begin{eqnarray}\label{upperineq1}
z^{(3)}_y\left(t,a_n(t)\right)\geq \frac{2}{\mu}\frac{1}{\mu^{n-n^*_0}}+n\left(\lambda_1^{\theta}\mu\right)^{2n}\frac{v}{2}(t-t^*).
\end{eqnarray}
Choose $\kappa\geq 2$. Then by (\ref{wn}), in a neighborhood of $(z^{(3)}\left(t^*,a_n(t^*)\right),t^*,a_n(t^*))$ 
\begin{eqnarray*}
\max W_{n-n_0^*-\kappa}&\leq &\frac{2}{\mu^{n-n^*_0-\kappa}}.
\end{eqnarray*}
From the previous inequality, (\ref{upperineq1}), we get that $W^u_{\text{loc}}(z^{(3)}( t, a_n(t)))$ is above $W_{n-n_0^*-\kappa}$ if
\begin{eqnarray*}
 n\left(\lambda_1^{\theta}\mu\right)^{2n}\frac{v}{2}(t-t^*)\geq \frac{2}{\mu^{n-n^*_0}}\left[\mu^{\kappa}-\frac{2}{\mu}\right]\geq 2C'\left(\lambda_1^{\theta}\mu\right)^{2n}\left[\mu^{\kappa}-\frac{2}{\mu}\right],
\end{eqnarray*}
where we also used (\ref{alphaofuse}) and $C'$ is a uniform constant.
As consequence, if 
$$
t-t^*\geq \frac{4C'}{n}\frac{\mu^{\kappa}}{v},$$
then, $W^u_{\text{loc}}(z^{(3)}( t, a_n(t)))$ is above $W_{n-n_0^*-\kappa}$. Because $W^u_{\text{loc}}( z^{(3)}( t^*, a_n(t^*)))$ contains a point below $ W_{n-n_0^*-1}$, there exists a parameter between $( t, a_n(t))$ and $( t^*, a_n(t^*))$ for which a secondary homoclinic tangency of type $n_0=n_0^*+\kappa$ occurs and $|t^*-t|=O\left(\frac{1}{n}\right)$. Moreover by Lemma \ref{partialwn}, (\ref{wn}), Lemma \ref{n0} and the definition of $\alpha$ we get
\begin{equation}\label{dWn-n0dt}
\frac{\partial W_{n-n_0}}{\partial t}(q_{1,n,n_0})=-2n\left(\lambda_1^{\theta}\mu\right)^{2n}\mu^{\kappa-1}\left(1-\alpha-\frac{\kappa}{n}\right)\frac{\partial\mu}{\partial t}+O\left(\left(\lambda_1^{\theta}\mu\right)^{2n}\right).
\end{equation}
From $(F2)$ it follows that ${\partial\mu}/{\partial t}\neq 0$.
The last statement of the lemma follows by taking $\kappa$ large enough and $n$ large enough.
\end{proof}

\subsection{Existence of tangency curves}
In the previous section we proved the existence of tangency points on the curve $A_n$ whose curvature is estimated in the following lemma. In the sequel we extend these points to create curves in $\mathcal{B}_n$. All points along these curves have a secondary tangency, refer to Figure \ref{Fig6}. 
\begin{lem}\label{curvature}
Let $(t,a)\in\mathcal B_n$ such that $W^u_{\text{\rm loc}}(z^{(3)}(t,a))$ has a secondary tangency $q_{1,n,n_0}$ of type $n_0$, then $q_{1,n,n_0}$ is non degenerate.  Namely, $W^u_{\text{\rm loc}}(q_{1,n,n_0})$ is the graph of a function and its curvature satisfies
$$
\text{\rm curv}\left(W^u_{\text{\rm loc}}( q_{1,n,n_0})\right)\geq C\left(\frac{|\mu|^4}{|\lambda_1|^{2-3\theta}}\right)^n,
$$
with $C>0$ a uniform constant.
\end{lem}
\begin{proof} Use coordinates centered at $z^{(2)}$ and let $(\Delta x,\Delta y)\in W^u_{\text{loc}}(z^{(2)})$ such that $F^N_{t,a}(\Delta x,\Delta y)=q_{1,n,n_0}$.
Then $0\leq\Delta y\leq L\left(|\lambda_1|^{\theta}|\mu|\right)^{n}$. Let $v=\left(\begin{matrix}
v_1\\v_2
\end{matrix}\right)$ be the tangent vector at $(\Delta x,\Delta y)$ to $ W^u_{\text{loc}}(z^{(2)})$. We use Lemma \ref{shape} and we obtain that 
\begin{equation}\label{eqn1}
|v_2|=\left[2Q_2\left(\frac{|\mu|}{|\lambda_1|^2}\right)^n|\Delta x|+O\left(|\Delta x|^2\left(\frac{|\mu|}{|\lambda_1|^3}\right)^n\right)\right]|v_1|.
\end{equation}
By Lemma \ref{functionc} we have that ${\partial^2 Y}/{\partial y^2}$ is bounded away from zero in a neighborhood of $c$. As consequence, by Lemma \ref{ymaxymin} and Lemma \ref{Nderivatives} we have $DF^N_{t,a}(\Delta x,\Delta y)=\left(\begin{matrix} 
A&B\\C&D
\end{matrix} \right)$ where $D\geq d\left(|\lambda_1|^{\theta}|\mu|\right)^{n}$ and $d>0$ is a uniform constant. Because $F^N_{t,a}(\Delta x,\Delta y)$ is a tangency at $q_{1,n,n_0}$, then 
\begin{equation}\label{eqn2}
Cv_1+Dv_2=O\left(\left(\frac{|\lambda_1|}{|\mu|}\right)^{n-n_0}\left(|v_1|+|v_2|\right)\right),
\end{equation}
where we used (\ref{dwn}). The proof of the lemma is completed by following exactly that of Lemma \ref{curvaturez3}.  
\end{proof}
In the next proposition we prove the existence of curves of secondary tangencies, see Figure \ref{Fig6}.
Let $C=2\max_{(t,a,\theta)}{D_5K}\left[C_3X_2\frac{\theta }{\lambda_1} \frac{\partial \lambda_1}{\partial t}+\frac{K}{\mu}\frac{\partial\mu}{\partial t}\right]$ and 
$$T^*_{n,n_0}=\left\{(t,a)\in T_{n,n_0}\left|\right. \left|\frac{\partial W_{n-n_0}}{\partial t}(q_{1,n,n_0})\right|\geq C n\left(\lambda_1^{\theta}\mu\right)^{2n}\right\}.$$
Observe that the secondary tangencies in $T^*_{n,n_0}$ are the ones for which the stable manifold moves faster than the local unstable manifold at the tangency when varying the $t$ parameter. As consequence of Proposition \ref{newtangency}, we get the following.
\begin{cor}\label{secontangnonempty}
For $n$ large enough, the set of types of secondary tangencies, $$\mathfrak N_n=\left\{(n,n_0)\left|\right. T^*_{n,n_0}\neq\emptyset\right\}$$ is non empty and $$\overline{\bigcup_n\bigcup_{\mathfrak N_n}T^*_{n,n_0}}\supset [-t_0,t_0]\times\left\{0\right\}.$$
\end{cor}

In the following proposition we show that the secondary tangencies form curves in parameter space. These curves intersect the curves $A_n$ transversally with small but controlled angle. The precise estimates of Proposition \ref{speed} are used again in a cancellation as in Proposition \ref{newtangency}.

\begin{prop}\label{angle} For each $T^*_{n,n_0}$ there exist  a $\Cd$ function  $b=b_{n,n_0}:[t^{-}_{n,n_0},t^+_{n,n_0}]\mapsto\R$ and a constant $V$ which is bounded away from zero, such that $T^*_{n,n_0}$ is the graph of $b$. In particular, $T^*_{n,n_0}$ is connected, 
$$\frac{db}{dt}=-\frac{n}{\mu^{n+1}}\frac{\partial\mu}{\partial t} +V {n \lambda_1^{\theta n}}+O\left(|\lambda_1|^{\theta n}\right),$$
and, if $(t,a)\in A_n\cap T^*_{n,n_0}$,
$$\frac{db}{dt}=\frac{da_n}{dt}+V{n \lambda_1^{\theta n}}+O\left(|\lambda_1|^{\theta n}\right).$$  Moreover the following holds: 
\begin{itemize}
\item[-]  $\partial T^*_{n,n_0}\subset\partial\mathcal B_n$,
\item[-] each $T^*_{n,n_0}$ has a unique transversal intersection with $A_n$.
\end{itemize}
\end{prop}
\begin{proof}
Let $(t,a)\in T^*_{n,n_0}$. We start by constructing a local function whose graph is contained in $T^*_{n,n_0}$. Let $m_1\in W^u_{\text{loc}}(c)$ such that $F_{t,a}^{3N+(1+\theta)n}(m_1)=q_{1,n,n_0}$. 
To describe the perturbation of $W^u_{\text{loc}}(q_{1,n,n_0})$ we define the following function by choosing coordinates centered in the image at $q_{1,n,n_0}$.
Take some $\epsilon >0$ and consider the $\Cq$ function $
(\tilde x,\tilde y):(-\epsilon,\epsilon)^3\mapsto\R^{m-1}\times\R $ defined by 
$$
\left(\tilde x(\Delta y,\Delta t, \Delta a),\tilde y(\Delta y,\Delta t, \Delta a)\right)=F_{t+\Delta t, a+\Delta a}^{3N+(1+\theta) n}\left(m_1+(0,\Delta y)\right),
$$
which describes $W^u_{\text{loc}}\left(F_{t+\Delta t, a+\Delta a}^{3N+(1+\theta)n}(m_1)\right)$. Observe that the manifolds $W_{n-n_0}$ of $F_{t+\Delta t, a+\Delta a}$ are described locally, near $q_{1,n,n_0}$, as the graph of a $\Cq$ function 
$$
w_{n-n_0}:(\Delta x, \Delta t,\Delta a)\longmapsto w_{n-n_0}(\Delta x, \Delta t,\Delta a)\in\R.
$$
The curves of secondary tangencies will be constructed using the implicit function theorem. For this aim define the $\Cd$ function $\Psi:[-\epsilon,\epsilon]^3\to\R^2$ as 
$$
\Psi\left(\Delta y, \Delta t,\Delta a\right)=\left(
\begin{matrix}
\tilde y(\Delta y,\Delta t, \Delta a)-w_{n-n_0}\left(\tilde x(\Delta y, \Delta t, \Delta a),\Delta t, \Delta a\right)\\ \frac{d\tilde y}{d\tilde x}\left(\Delta y, \Delta t,\Delta a\right)-\frac{dw_{n-n_0}}{d\tilde x}\left(\tilde x(\Delta y, \Delta t, \Delta a),\Delta t, \Delta a\right)\end{matrix}\right)=\left(\begin{matrix}
\psi_1\\ \psi_2\end{matrix}\right).
$$
Observe that $\Psi^{-1}(0)$ describes locally the perturbation of the secondary tangency $q_{1,n,n_0}$ and that
\begin{equation}\label{psimap}
D\Psi({0,0,0})=\left(\begin{matrix}
0&\left(\lambda_1^{\theta}\mu\right)^n  \Psi_{1,2}& \left(\lambda_1^{\theta}\mu^2\right)^n \Psi_{1,3}\\
  \left(\frac{|\mu|^3}{|\lambda_1|^{2-2\theta}}\right)^n\Psi_{2,1}&  \Psi_{2,2} &  \Psi_{2,3}
\end{matrix}\right).
\end{equation}
In particular, by (\ref{dm6y}), (\ref{dwn}), (\ref{dm6x}), (\ref{alphaofuse}) (in the estimation of the order term)
\begin{eqnarray}\label{psi12}
\nonumber
\Psi_{1,2}&=&\frac{1}{\left(\lambda_1^{\theta}\mu\right)^n}\left[\frac{\partial m_{6,y}}{\partial t}-\frac{\partial W_{n-n_0}}{\partial x}\frac{\partial m_{6,x}}{\partial t}-\frac{\partial W_{n-n_0}}{\partial t}\right]
\\\nonumber &=&D_5K \frac{n}{\mu}\frac{\partial\mu}{\partial t}+ D_5Kn \left(\lambda_1^{\theta}\mu\right)^{n}\left[C_3X_2\frac{\theta }{\lambda_1} \frac{\partial \lambda_1}{\partial t}+\frac{K}{\mu}\frac{\partial\mu}{\partial t}\right]
\\\nonumber & +& 
O\left(\left(|\lambda_1|^{\theta}|\mu|\right)^{n}\right)+O\left(n\left(\frac{\lambda_1}{\mu}\right)^{n-n_0}\frac{1}{\left(\lambda_1^{\theta}\mu\right)^n }\right)-\frac{1}{\left(\lambda_1^{\theta}\mu\right)^n }\frac{\partial W_{n-n_0}}{\partial t}
\\\nonumber &=&D_5K  \frac{n}{\mu}\frac{\partial\mu}{\partial t}+ n \left(\lambda_1^{\theta}\mu\right)^{n}\left\{{D_5K}\left[C_3X_2\frac{\theta }{\lambda_1} \frac{\partial \lambda_1}{\partial t}+\frac{K}{\mu}\frac{\partial\mu}{\partial t}\right]-\frac{1}{n\left(\lambda_1^{\theta}\mu\right)^{2n}}\frac{\partial W_{n-n_0}}{\partial t}\right\}\\\nonumber &+&O\left(\left(|\lambda_1|^{\theta}|\mu|\right)^{n}\right)\\
&=&D_5K  \frac{n}{\mu}\frac{\partial\mu}{\partial t}+D_5K V n \left(\lambda_1^{\theta}\mu\right)^{n} +O\left(\left(|\lambda_1|^{\theta}|\mu|\right)^{n}\right),
\end{eqnarray}
where 
\begin{equation}\label{Vespr}
V=\left[C_3X_2\frac{\theta }{\lambda_1} \frac{\partial \lambda_1}{\partial t}+\frac{K}{\mu}\frac{\partial\mu}{\partial t}\right]-\frac{1}{n\left(\lambda_1^{\theta}\mu\right)^{2n}}\frac{\partial W_{n-n_0}}{\partial t}\frac{1}{D_5K}
\end{equation}
is a large uniform constant, see (\ref{speedwnminusn0}), and by (\ref{dm6yda}), (\ref{dwn}), (\ref{dm6x}), Lemma \ref{partialwn} and (\ref{alphaofuse}) (in the estimation of the order term)
\begin{eqnarray}\label{psi13}
\nonumber
\Psi_{1,3}&=&\frac{1}{\left(\lambda_1^{\theta}\mu^2\right)^n}\left[\frac{\partial m_{6,y}}{\partial a}-\frac{\partial W_{n-n_0}}{\partial x}\frac{\partial m_{6,x}}{\partial a}-\frac{\partial W_{n-n_0}}{\partial a}\right]
\\\nonumber &=&\frac{1}{\left(\lambda_1^{\theta}\mu^2\right)^n}\left\{D_5K\left(\lambda_1^{\theta}\mu^2\right)^n +1+ D_5K \left(\lambda_1^{\theta}\mu\right)^n  \frac{n}{\mu}\frac{\partial\mu}{\partial a}\right.
\\ \nonumber &+&\left. D_5Kn \left(\lambda_1^{\theta}\mu\right)^{2n}\left[C_3X_2\frac{\theta }{\lambda_1} \frac{\partial \lambda_1}{\partial a}+\frac{K}{\mu}\frac{\partial\mu}{\partial a}\right] +O\left(\left(|\lambda_1|^{\theta}|\mu|\right)^{2n}\right)\right.\\\nonumber&+&\left.O\left(\mu^n\left(\frac{\lambda_1}{\mu}\right)^{n-n_0}\right)+O\left(\frac{n-n_0}{\mu^{n-n_0}}\right)\right\}
\\&=&D_5K\left[1+O\left({\left({|\lambda_1|}^{1+\theta-\alpha}{|\mu|}\right)^{n}}\right)\right].
\end{eqnarray}
Moreover $ \Psi_{1,2},\Psi_{1,3}\ne 0$ and $ \Psi_{2,1}\ne 0$, see Lemma \ref{curvature}, (\ref{dm6x}) and (\ref{ddwn}). Because $D\Psi({0,0,0})$ is onto we get that the set of secondary tangencies $\Psi^{-1}(0)$ is locally the graph of a $\Cd$ function $b$. Moreover the $T_{(0,0,0)}\Psi^{-1}(0)=\text{Ker} D\Psi$. Hence 
$$
\left(\lambda_1^{\theta}\mu\right)^n  \Psi_{1,2}\Delta t+\left(|\lambda_1|^{\theta}|\mu|^2\right)^n\Psi_{1,3}\Delta b=0,
$$
and $$\frac{db}{dt}=-\frac{\Psi_{1,2}}{\Psi_{1,3}}\frac{1}{\mu^n}=-\frac{n}{\mu^{n+1}}\frac{\partial\mu}{\partial t} -{n \lambda_1^{\theta n}}V +O\left(|\lambda_1|^{\theta n}\right).$$ Moreover, by the previous estimate on the slope of $b$ and by (\ref{dandt}), we have
\begin{equation}\label{dbminusda}
\frac{db}{dt}=\frac{da_n}{dt}-{n \lambda_1^{\theta n}}V +O\left(|\lambda_1|^{\theta n}\right),
\end{equation}
 for $(t,a)\in  A_n\cap T^*_{n,n_0}$. By comparing $\mu^n(t,a)$ and $\frac{\partial\mu}{\partial t}(t,a)$ with $\mu^n(t,a_n(t))$ and $\frac{\partial\mu}{\partial t}(t,a_n(t))$ on $\mathcal B_n$, see Remark \ref{muzeroothers}, one gets the same estimate as in (\ref{dbminusda}) for any other point $(t,a)\in\mathcal B_n$. This uniform bound on the difference of the slopes allows to extend $b$ globally up to both boundaries of $\mathcal B_n$. In particular the length of a component of $T^*_{n,n_0}$ is proportional to $\frac{1}{n}$. Moreover each component of $T^*_{n,n_0}$ intersects $A_n$ transversally, see (\ref{dbminusda}), in a unique point. 
 \\
It is left to show that $T^*_{n,n_0}$ has only one connected component. This follows from the fact that the function $h:t\mapsto\psi_1\left(\Delta y, t, a_n(t)\right)$, where $\Psi=\left(
\psi_1,\psi_2\right)$, is strictly monotone for every $\Delta y$. Namely,
\begin{eqnarray*}
\frac{dh}{dt}&=&\left(\lambda_1^{\theta}\mu\right)^n  \Psi_{1,2}+ \left(\lambda_1^{\theta}\mu^2\right)^n \Psi_{1,3}\frac{da_n}{dt}
\\&=&\left(\lambda_1^{\theta}\mu\right)^{n}D_5K  \frac{n}{\mu}\frac{\partial\mu}{\partial t}+D_5K V n \left(\lambda_1^{\theta}\mu\right)^{2n}-\left(\lambda_1^{\theta}\mu\right)^{n}D_5K  \frac{n}{\mu}\frac{\partial\mu}{\partial t}+O\left(\left(\lambda_1^{\theta}\mu\right)^{2n}\right)
\\&=&\left(\lambda_1^{\theta}\mu\right)^{2n}nD_5KV+O\left(\left(\lambda_1^{\theta}\mu\right)^{2n}\right)\neq 0.
\end{eqnarray*}
 where we used (\ref{psi12}), (\ref{psi13}) and (\ref{dbminusda}).
  
\end{proof}
\begin{rem}\label{anglebound}
Observe that the constant $V$, see (\ref{Vespr}), is mainly determined by ${\partial W_{n-n_0}}/{\partial t}$, see Proposition \ref{newtangency} and the definition of $T^*_{n,n_0}$. Moreover ${\partial W_{n-n_0}}/{\partial t}$ is controlled by the condition $\partial\mu/\partial t\neq 0$, see (\ref{dWn-n0dt}). The angle is determined by infinitesimal properties at the saddle point. In particular, by taking $n$ large enough we can assure that the curves $b$ and $a_n$ intersect transversally. Observe that, in the estimate for the angle, see (\ref{dbminusda}), the error term, which is of the order $\lambda_1^{\theta n}$, is dominated by the term $nV\lambda_1^{\theta n}$.
\end{rem}
\section{Newhouse phenomenon}
In this section we select the parameters corresponding to maps having infinitely many sinks. The proof is done by induction on what we call the "Newhouse boxes". In the first generation, the Newhouse boxes are essentially rectangles in $\mathcal A_n$ whose boundary are defined by the curves of secondary tangencies, see Figure \ref{Fig6}. The family restricted to the Newhouse boxes of first generation have one sink and it is an unfolding of a new homoclinic tangency. The propositions and the lemmas proved in the previous sections apply then to these families creating Newhouse boxes of second generation. As a consequence, the family restricted to the Newhouse boxes of second generation have two sinks and it is an unfolding of a new homoclinic tangency. We proceed by induction.

Let $(t_{n,n_0},a_{n,n_0})$ be the parameters at the intersection point $ A_n\cap T^*_{n,n_0}$, see Proposition \ref{angle}, and $$f_{n,n_0}=F_{t_{n,n_0},a_{n,n_0}}.$$ Recall that $T^*_{n,n_0}$ is the graph of a function  $b=b_{n,n_0}:[t^{-}_{n,n_0},t^+_{n,n_0}]\mapsto\R$. The domains
$$
\mathcal {P}_{n,n_0}=\left\{(t,a)\left|\right.  t\in[t^{-}_{n,n_0},t^+_{n,n_0}], |a-a_n(t)|\leq\frac{\epsilon_0}{|\mu(t,a_n(t))|^{2n}}\right\}
$$
are called the Newhouse boxes of first generation. 
\bigskip

The construction in the previous sections started with a map $f:M\to M$ with a strong homoclinic tangency and an unfolding $F:\mathcal P\times M\to M$. The following inductive construction will repeat the discussion of the previous sections starting with the map $f_{n,n_0}:M\to M$ and an unfolding $F:\mathcal P_{n,n_0}\times M\to M$ which is the restriction of the original family.
\begin{lem}\label{disjP}
The domains $\mathcal {P}_{n,n_0}$ are pairwise disjoint for $n$ large enough and the diameter goes to zero.
\end{lem}
\begin{proof}
By Proposition \ref{angle} the curve $ T^*_{n,n_0} $ is the graph of the function $b$ and there exists a uniform constant $K>0$ such that
\begin{equation}\label{horizsize}
\frac{1}{K}\frac{1}{n\left(|\lambda_1|^{\theta}|\mu|^2\right)^n}\leq \left| \left\{t\left|\right. (t,b(t))\in\mathcal A_n\right\}\right|\leq K\frac{1}{n\left(|\lambda_1|^{\theta}|\mu|^2\right)^n},
\end{equation}
and the proof of Proposition \ref{newtangency} gives $\text{dist}\left(f_{n,n_0},f_{n,n_0+1}\right)$ is proportional to ${1}/{n}$. The disjointness follows from this estimates, the fact that $|\lambda_1|^{\theta}|\mu|^2\geq \left(|\lambda_1|^{2\theta}|\mu|^3\right)^{1/2}>1$, see (\ref{thetacond}) and the fact that $\mathcal A_n$ are pairwise disjoint.
\end{proof}
The next proposition ensures that the family restricted to $\mathcal P_{n,n_0}$ is again an unfolding of a strong homoclinic tangency. This allows an inductive procedure.
\begin{prop}\label{induction}  
For $n$ large enough, the map $f_{n,n_0}$ has a strong homoclinic tangency and the restriction $F:\mathcal P_{n,n_0}\times M\to M$ can be reparametrized to become an unfolding. Moreover each map in $\mathcal P_{n,n_0}$ has a sink of period $n+N$.
\end{prop}
\begin{proof}
Observe the new family is a restriction of our original family and the secondary tangency curve $b_{n,n_0}$ describes homoclinic tangencies associated to the original saddle point. The transversal intersection $q_2$ it is still present in the new family. As a consequence, the conditions $(f1), (f2), (f3),(f6),(f7), (f8), (F1), (F2), (F3)$ are automatically satisfied. We need only to check the conditions involving the secondary tangency $q_{1,n,n_0}$. Let $T$ be the time such that $f^T(q_2)\in [-2,2]^{m-1}\times\left\{0\right\}$. From $(f6)$ we know that the first coordinate of $f^T(q_2)$ is non zero. Recall now that $q_{1,n,n_0}\in W_{n-n_0}$. As consequence, for $n$ large enough, $f_{n,n_0}^{n-n_0}(q_{1,n,n_0})$ is close to $q_2$. Hence,
$f_{n,n_0}^{n-n_0+T}(q_{1,n,n_0})\in [-2,2]^{m-1}\times\left\{0\right\}$ has first coordinate non zero. This proves $(f4)$ for $q_{1,n,n_0}$.
\\
 Use the notation $m_1,m_2,m_3,m_4,m_5$ from Proposition \ref{speed} and observe that $m_1\in W^u_{\text{loc}}(0)$ close to $q_3$. This proves $(f10)$ for $q_{1,n,n_0}$.
 \\
Let $B\in T_{q_1}W^u(p)$ the unit vector. From $(f5)$ we know that $B$ has a non zero first coordinate. For $n$ large, the direction of $T_{m_2}W^u(p)$ is close to $B$, hence it has a non zero first coordinate. This implies that the direction of $T_{m_3}W^u(p)$ is close to $E_{q_3}$ and again from $(f5)$, the direction of $T_{m_4}W^u(p)$ is close to $B$ and it has a non zero first coordinate. Hence the direction of $T_{m_5}W^u(p)$ is close to $E_{q_3}$ and from $(f5)$, the direction of $T_{q_{1,n,n_0}}W^u(p)$ is close to $B$ and it has a non zero first coordinate. The direction of $T_{f^{n-n_0}_{n,n_0}\left(q_{1,n,n_0}\right)}W^u(p)$ is close to $E_{q_2}\cap T_{q_2}W^s(p)$.  Property $(f7)$ implies $(f5)$ for $q_{1,n,n_0}$.
\\
Observe that $f_{n,n_0}^{n-n_0}\left(q_{1,n,n_0}\right)$ converges to $q_2$. Property $(f8)$ implies $(f9)$ for $q_{1,n,n_0}$.
\\
For proving $(F4)$ use the $\Cd$ function $b_{n,n_0}$ and observe that the maps on the graph of this curve, $T^*_{n,n_0}$, have a non degenerate homoclinic tangency, see Lemma \ref{curvature}. The proof that these tangencies are in general position, it is the same as the one that we use to prove that $q_{1,n,n_0}$ is in general position. Similarly one proves $(F5)$. 
\\
Observe that $(P1)$ follows by the fact that $b_{n,n_0}$ is the curve of tangencies and $(P2)$ from the fact that $\Psi_{1,3}\neq 0$, see (\ref{psimap}).
\end{proof}
Inductively we are going to construct parameters with multiple sinks of higher and higher periods and a strong homoclinic tangency.
Let $\mathfrak N^1=\left\{(n,n_0)\left|\right. n_0\in\mathfrak N_n\right\}\subset\N^2$ be the set of labels of the Newhouse boxes of first generation $\mathcal P_{n,n_0}$. As inductive hypothesis assume that there exist sets $\mathfrak N^k\in\left(\N^2\right)^k$ such that, each  $\mathfrak N^k$ is the set of labels of the Newhouse boxes of generation $k$. Moreover the natural projections $\mathfrak N^1\leftarrow\mathfrak N^2\leftarrow\dots\leftarrow\mathfrak N^g $ with $\mathfrak N^k\in\left(\N^2\right)^k$ are countable to $1$ and correspond to inclusion of Newhouse boxes of successive generations. They satisfy the following inductive hypothesis.

 For $\underline n=\left\{(n^{(k)},n_0^{(k)})\right\}_{k=1}^g\in\mathfrak N^g$, there exist sets $\mathcal P^k_{n^{(k)},n_0^{(k)}}\subset\mathcal P$, where $n^{(k)}$ labels the sink and  $n_0^{(k)}$ labels the secondary tangency, such that
\begin{itemize}
\item $\text{diam}\left(\mathcal P^k_{n^{(k)},n_0^{(k)}}\right)\leq\frac{1}{k}$,
\item $\mathcal P^{k+1}_{n^{(k+1)},n_0^{(k+1)}}\subset \mathcal P^k_{n^{(k)},n_0^{(k)}}$ for $k=1,\dots,g-1$,
\item there exists a map $f^k_{n^{(k)},n_0^{(k)}}\in\mathcal P^k_{n^{(k)},n_0^{(k)}}$ which has a strong homoclinic tangency,
\item the restriction $F:{\mathcal P^k_{n^{(k)},n_0^{(k)}}}\times M\to M$ can be reparametrized to become and unfolding of $f^k_{n^{(k)},n_0^{(k)}}$,
\item every map in $\mathcal P^k_{n^{(k)},n_0^{(k)}}$ has at least $k$ sinks of different periods.
\end{itemize}
By induction, using Proposition \ref{induction},  we get an infinite sequence of sets $\mathfrak N^k$. 
\begin{figure}
\centering
\includegraphics[width=0.73\textwidth]{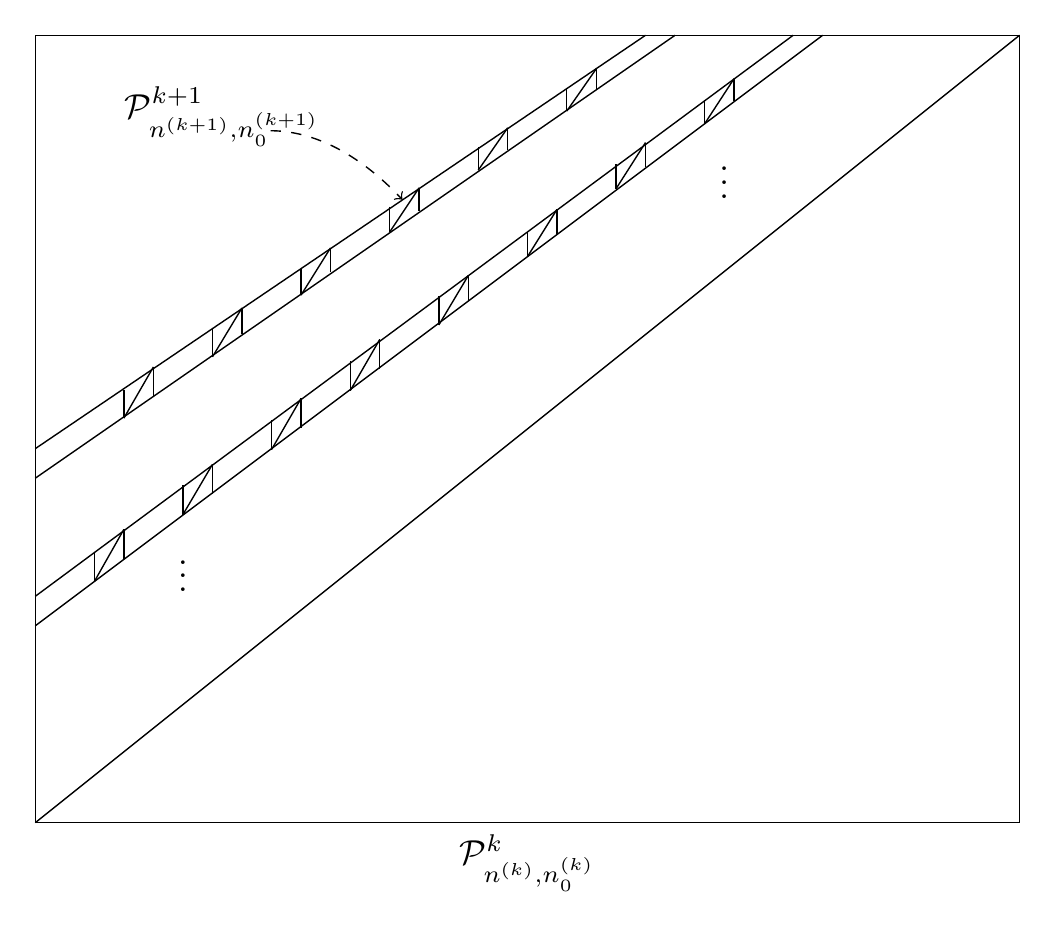}
\caption{Newhouse boxes}
\label{Fig7}
\end{figure}
\begin{defin}
The sets $\mathcal P^g_{n^{(g)},n_0^{(g)}}$ are called Newhouse boxes of generation $g$, see Figure \ref{Fig7}, and 
$$NH=\bigcap_g\bigcup_{\underline n\in\mathfrak N^g}\mathcal P^g_{n^{(g)},n_0^{(g)}}.$$ 
\end{defin}

The set $NH$, consisting of parameters for which the corresponding maps have infinitely many sinks, accumulates on the curve of the original tangency. The inductive construction of these parameters implies that the set $NH$ accumulates on all tangency curves given by $b_{n,n_0}$. We have the following lemma.

\begin{lem}\label{pacc}
\begin{equation*}
\overline{NH}\supset [-t_0,t_0]\times\left\{0\right\}.
\end{equation*}
In particular,
\begin{equation}\label{NHinter}
\overline{NH}\cap \mathcal {P}_{n,n_0}\supset\text{\rm graph}(b_{n,n_0}).
\end{equation}
\end{lem}
\begin{proof}
Given $(t,0)$, by Proposition \ref{newtangency}, for every $n$ large enough, $A_n$ has a secondary tangency at $(t_n,a_n(t_n))$ with $|t-t_n|=O\left(\frac{1}{n}\right)$. Hence, there exists a sequence of Newhouse boxes in $\mathfrak{N_1}$ accumulating at $(t,0)$. By construction, each box in $\mathfrak{N_1}$ contains points of $NH$. The lemma follows.
\end{proof}
Given any family $F$ of diffeomorphisms, we define the Newhouse set $NH_F$ as the set of parameters having infinitely many sinks.
The upper Minkowski dimension is denoted by $MD$.

%
\vskip .2 cm
\paragraph{\bf Theorem A.}\label{Newhousepoints}
Let $F:\mathcal P\times M\to M$ be an unfolding of a map $f$ with a strong homoclinic tangency, then 
\begin{itemize}
\item[-] $NH\subset NH_F$, every map in $NH$ has infinitely many sinks,
\item[-] $NH$ is homeomorphic to $\R\setminus\Q$,
\item[-]$MD(NH)\geq\frac{1}{2}$. 

\end{itemize}
\vskip .2 cm
\begin{proof} The inductive construction, using  Proposition \ref{induction}, implies that all maps in $NH$ have infinitely many sinks. From  Lemma \ref{disjP} and Corollary \ref{secontangnonempty} we know that 
$$
\bigcup_{\underline n\in\mathfrak N^g}\mathcal P^g_{n^{(g)},n_0^{(g)}},
$$
consists of countably many disjoint boxes. Each box $P^g_{n^{(g)},n_0^{(g)}}$ contains countably many pairwise disjoint boxes of the next generation. Hence, the nested intersection $NH$ is homeomorphic to $\R\setminus\Q$, see \cite{Gaal}. For the last property, let $(t^*,0)$ such that ${\log{\lambda(t^*,0)^{-1}}}/{\log\mu(t^*,0)}$ is the maximum of ${\log{\lambda(t,0)^{-1}}}/{\log\mu(t,0)}$. Consider a sequence of first generation Newhouse boxes $\mathcal {P}_{n,n_0}\in\mathfrak{N_1}$ accumulating at $(t^*,0)$. This is possible because of Lemma \ref{pacc}. Choose $\epsilon>0$ and let $n$ be maximal such that $\epsilon\leq {\epsilon_0}/{\mu(t^*,a^*)^{2n}}$. 
 Because (\ref{horizsize}), (\ref{NHinter}) and  the fact that the vertical size of $\mathcal {P}_{n,n_0}$ is ${\epsilon_0}/{\mu(t^*,a^*)^{2n}}$, we need at least ${K}/{\lambda(t^*,a^*)^{\theta n}}$ balls of radius $\epsilon$ to cover $\overline{NH}\cap \mathcal {P}_{n,n_0}$. As consequence $$MD(NH)\geq \frac{\theta}{2}\max_{NH}\left(\frac{\log\frac{1}{\lambda}}{\log\mu}\right),$$
 and $MD(NH)\ge\frac{1}{2}$, 
  where we used  (\ref{thetacond1}).
\end{proof}
\begin{rem}
Observe that the estimate for the upper Minkowski dimension is not sharp.  Other dimension estimates for maps with infinitely many sinks were obtained in \cite{BDS, TY, W}. 
\end{rem}
%

The construction of $NH$ involves the Newhouse boxes. These boxes are constructed using the curves $a_n$ and $b_{n,n_0}$. The transversality of these curves implies the stability of these boxes. By considering higher dimensional families, with more than two parameters, the boxes will move smoothly. The intersection $NH$ will create a lamination in higher dimensional unfoldings. This shows that the Newhouse phenomenon has a codimension 2 nature.

\vskip .2 cm
\paragraph{\bf Theorem B.}\label{Newhouselamination}
Let $M$, $\mathcal P$ and $\mathcal T$ be $\Cinf$ manifolds and $F:\left(\mathcal P\times\mathcal T\right)\times M\to M$ a $\Cinf$ family with $\text{dim}(\mathcal P)=2$ and $\text{dim}(\mathcal T)\geq 1$. If $F_0:\left(\mathcal P\times\left\{\tau_0\right\}\right)\times M\to M$ is an unfolding of a map $f_{\tau_0}$ with a strong homoclinic tangency, then 
\begin{itemize}
\item[-] $NH_F$ contains a codimension $2$ lamination $L_F$,
\item[-] $L_F$ is homeomorphic to $\left(\mathbb R\setminus\mathbb Q\right)\times\mathbb R^{\dim (\mathcal T)}$,
\item[-] the leaves of $L_F$ are $\Cuno$ codimension $2$ manifolds,
\item[-] infinitely many sinks persist along each leave of the lamination.
\end{itemize}
\vskip .2 cm
\begin{proof}
Observe that there exists a small neighborhood $\tau_0\in U\subset\mathcal T$ and a $\Cinf$ function $U\ni\tau\to f_t$ such that, for all $\tau\in U$, $f_{\tau}$ has a strong homoclinic tangency and the family $F_{\tau}:\left(\mathcal P\times\left\{\tau\right\}\right)\times M\to M$ is an unfolding of $f_{\tau}$.
 Let $D_5,K,C_3,X_2$ be the constants as in Proposition \ref{speed} and
 $$C=2\max_{(t,\tau, a,\theta)}{D_5K}\left[C_3X_2\frac{\theta }{\lambda_1} \frac{\partial \lambda_1}{\partial t}+\frac{K}{\mu}\frac{\partial\mu}{\partial t}\right],$$ and choose a secondary tangencies at $(t_0,\tau_0,a_0)$ from
$$\left\{(t,\tau, a)\in T_{n,n_0}\left|\right. \left|\frac{\partial W_{n-n_0}}{\partial t}(q_{1,n,n_0})\right|\geq C n\left(\lambda_1^{\theta}\mu\right)^{2n}\right\}\cap A_n.$$ In the corresponding unfolding $F_{\tau_0}$, there is the tangencies curve $b$ which intersects transversally the sinks curve $A_n$ in the point $(t_0,\tau_0,a_0)$. From Proposition \ref{angle}  we get a lower bound for the angle which is independent on the parameter $\tau$, see Remark \ref{anglebound}. This transversality implies that this intersection persists in a neighborhood of $\tau_0$ as the graph of a smooth function. The uniform lower bound of the angle implies that this smooth function extends globally. In particular the secondary tangency $(t_0,\tau_0,a_0)$ on the sink curve has its smooth continuation in all unfoldings $F_{\tau}$ for any given $\tau$ creating Newhouse boxes.
\\
Fix $\tau\in U$ and denote the Newhouse boxes of the family $F_{\tau}$ by $\left\{\mathcal N_{\underline n}(\tau)\right\}^g$. 
These boxes are defined in terms of the smooth functions $a$ and $b$. Because the angle between $a$ and $b$ is uniformly bounded by a constant independent of $\tau$, see Proposition \ref{angle}, the boxes $P^k_{n^k,n^k_0}(\tau)$ move smoothly with $\tau$. Let 
\begin{equation}\label{alongated}
P^k_{n^{(k)},n^{(k)}_0}(U)=\bigcup_{\tau\in U}P^k_{n^{(k)},n^{(k)}_0}(\tau),
\end{equation}
 then the sets  $P^k_{n^{(k)},n^{(k)}_0}(U)$ approximate the lamination, in the sense that they are homeomorphic to $$P^k_{n^{(k)},n^{(k)}_0}(\tau_0)\times U.$$ Let $$L_F=\bigcap_g\bigcup_{\underline n\in\mathfrak N^g}\mathcal P^g_{n^{(g)},n_0^{(g)}}(U),$$ and observe that $L_F$ is homeomorphic to $U\times NH(\tau_0)$.
\\
It is left to prove that the leaves $L_F$ are $\Cuno$. Let $\text{Tang}^k_{n^{(k)},n^{(k)}_0}$ be the codimension one surface of tangencies contained in $P^k_{n^{(k)},n^{(k)}_0}(U)$.  Let $P^k_{n^{(k)},n^{(k)}_0}(U)\subset P^{k-1}_{n^{(k-1)},n^{(k-1)}_0}(U)$ and reparametrize $P^{k-1}_{n^{(k-1)},n^{(k-1)}_0}(U)$ in coordinates, also denoted by $(t,a,\tau)$, such that the restriction of this reparametrization to a slice at $\tau$, using only the coordinates $(t,a)$, is an unfolding.  In particular, $\text{Tang}^{k-1}_{n^{(k-1)},n^{(k-1)}_0}=\left\{a=0\right\}$. Observe that there is no difference between parameters $t$ and parameters $\tau$ and Proposition \ref{angle} applies to both.
According to Proposition \ref{angle} we have, in the coordinates of $ P^{k-1}_{n^{(k-1)},n^{(k-1)}_0}(U)$, with $\tau=(\tau_i)$
\begin{eqnarray*}
\frac{db}{dt}&=&-\frac{1}{\mu^{n^{(k)},}}\left[\frac{{n^{(k)}}}{\mu}\frac{\partial\mu}{\partial t} +O\left({n^{(k)}}\left(|\lambda_1|^{\theta}|\mu|\right)^{n^{(k)}}\right)\right],\\
\frac{db}{d\tau_i}&=&-\frac{1}{\mu^{n^{(k)}}}\left[\frac{{n^{(k)}}}{\mu}\frac{\partial\mu}{\partial \tau_i} +O\left({n^{(k)}}\left(|\lambda_1|^{\theta}|\mu|\right)^{n^{(k)}}\right)\right].
\end{eqnarray*}
The codimension one surface $\text{Tang}^k_{n^{(k)},n^{(k)}_0}$ can be described as a graph of a smooth function over a domain in $\text{Tang}^{k-1}_{n^{(k-1)},n^{(k-1)}_0}$, denoted as $\text{Tang}^k_{n^{(k)},n^{(k)}_0}$. 
As a consequence of the previous estimates, the graph $\text{Tang}^k_{n^{(k)},n^{(k)}_0}$ is $O\left(\frac{{n^{(k)}}}{|\mu|^{n^{(k)}}}\right)$ $\Cuno$ close to $\text{Tang}^{k-1}_{n^{(k-1)},n^{(k-1)}_0}$. 
By Lemma \ref{Imtang} we have
\begin{eqnarray}\label{dadtnew}
\frac{d a}{d t}&=&O\left(\frac{{n^{(k)}}}{|\mu|^{n^{(k)}}}\right),\\\label{dadtaunew}
\frac{d a}{d \tau_i}&=&O\left(\frac{{n^{(k)}}}{|\mu|^{n^{(k)}}}\right).
\end{eqnarray}
By Proposition \ref{angle}, the graphs of $a$ and $b$ intersect transversally in a manifold $\ell$. Notice that $\ell$ is the graph of a $\Cd$ function $\ell:\tau\mapsto\ell(\tau)\in P^k_{n^{(k)},n^{(k)}_0}(\tau)$. In particular $\ell$ is a codimension $2$ manifold.  According to Proposition \ref{speed}, we have
\begin{eqnarray}\label{dbdtminusdadt1}
 \frac{db}{d t}-\frac{da}{d t}&=&V n\lambda_1^{\theta n}+O\left(\left|\lambda_1\right|^{\theta n}\right),\\\label{dbdtminusdadt2}
\frac{db}{d\tau_i}-\frac{da}{d \tau_i}&=&V_i n\lambda_1^{\theta n}+O\left(\left|\lambda_1\right|^{\theta n}\right),
\end{eqnarray}
where $V$ and $V_i$ are continuous and uniformly away from zero. Let $(\Delta t,\Delta\tau_i,\Delta a)$ be a tangent vector to $\ell$. Then 
$$
\frac{\partial (b-a)}{\partial t}\Delta t+\sum_i\frac{\partial (b-a)}{\partial \tau_i}\Delta \tau_i=0,
$$
and 
$$
\frac{\partial a}{\partial t}\Delta t+\sum_i\frac{\partial a}{\partial \tau_i}\Delta \tau_i=\Delta a.
$$
By (\ref{dbdtminusdadt1}), (\ref{dbdtminusdadt2}), (\ref{dadtnew}) and  (\ref{dadtaunew}), the tangent space of $\ell$ is given by 
\begin{eqnarray*}
\left(1+O\left(\frac{1}{n}\right)\right)\left(V\Delta t+\sum V_i\Delta\tau_i \right)&=&0,\\
O\left(\frac{{n^{(k)}}}{|\mu|^{n^{(k)}}}\left(\Delta t+\sum\Delta\tau_i\right)\right)&=&\Delta a.
\end{eqnarray*}
By restricting $\mathfrak N^k$ to large values of $n^{(k)}$ the tangent spaces of $\ell$ converge to
$\Delta a=0$ and 
$
V\Delta t+\sum V_i \Delta\tau_i=0 .
$
As consequence, the leaves of the lamination are smooth manifolds and the tangent spaces to the leaves of the lamination vary continuously.
\end{proof}

%
%
%
\section{Real polynomial families}
Consider the real H\' enon family $F:\mathbb{R}^2\times\mathbb{R}^2\to\mathbb{R}^2$, 
$$
F_{a,b}\left(\begin{matrix}
x\\y
\end{matrix}
\right)=\left(\begin{matrix}
a-x^2-by\\x
\end{matrix}
\right),
$$
a two parameter family. In this section we are going to prove that the H\' enon family is an unfolding of a map with a strong homoclinic tangency. In particular we can apply Theorem A and Theorem B to get a Newhouse lamination in the space of polynomial maps, see Theorem C. Actually we could have started with any two parameter polynomial family which is an unfolding of a map with a strong homoclinic tangency and we would have got Newhouse laminations scattered throughout the space of polynomial maps. 
\vskip .2 cm
\paragraph{\bf Theorem C.}
The real H\' enon family contains a set $NH$, homeomorphic to $\mathbb{R}\setminus \mathbb{Q}$, of maps with infinitely many sinks.  Moreover the space $\text{Poly}_d({\mathbb{R}^n})$ of real polynomials of $\mathbb{R}^n$ of degree at most $d$, with $d\ge 2$, contains a codimension $2$ lamination of maps with infinitely many sinks. The lamination is homeomorphic to $\left(\mathbb R\setminus\mathbb Q\right)\times\mathbb R^{\text{D}-2}$ where $\text{D}$ is the dimension of $\text{Poly}_d({\mathbb{R}^n})$ and the leaves of the lamination are $\Cuno$ smooth. The sinks persist along each leave of the lamination. 
\vskip .2 cm
\begin{proof}
Consider the map $f(x)=2-x^2$. Then $x=-2$ is an expanding fixed point and $f^2(0)=-2$ where $0$ is the critical point. For given $b>0$ and $a$ large enough, the H\' enon map $F_{a,b}$ has an horse-shoe. By decreasing $a$ to $a(b)$, one arrives at the first homoclinic tangency. Hence, there exists an analytic curve $b\mapsto a(b)$ with $a(0)=2$ such that the parameter $(a(b),b)$ corresponds to a H\' enon map with an homoclinic tangency of the saddle point $p(b)$  which is a continuation of $p(0)=(-2,-2)$. For all $b$ positive and small enough $F_{a(b),b}$ has a strong homoclinic tangency and the H\' enon family $F_{a,b}$ restricted to a small neighborhood of $(a(b),b)$ is an unfolding. The theorem follows for $\text{Poly}_d({\mathbb{R}^2})$ from Theorem A and Theorem B applied to the unfolding $F_{a,b}$. 
\\
Extend now the family $F_{a,b}$ to $F:\mathbb R^2\times\mathbb R^n\mapsto\mathbb R^n$ as 
$$
F_{a,b}\left(\begin{matrix}
x\\y\\y_3\\\vdots\\y_n
\end{matrix}
\right)=\left(\begin{matrix}
a-x^2-by\\x\\b^3y_3\\\vdots\\b^ny_n
\end{matrix}
\right).
$$
Observe that this higher dimensional H\' enon family is again an unfolding in the same neighborhood of $(a(b),b)$. As consequence Theorem $A$ and Theorem $B$ apply and the proof is complete.
 \end{proof}
 In the rest of the section we prove that the leaves of the Newhouse laminations in the space of real polynomial maps of $\mathbb R^2$ are in fact real analytic. Besides this property of the Newhouse laminations has its own interest, it will also be used in the next section. We would like to stress that the result is presented only in two dimension because of its application in the next section, see Theorem F. However it is valid in higher dimension.
 \begin{prop}\label{analyticleaves} The space $\text{\rm Poly}_d({\mathbb{R}^2})$ of real polynomials of $\mathbb{R}^2$ of degree at most $d$, with $d\ge 2$, contains a codimension $2$ lamination $L_F$ of maps with infinitely many sinks. The lamination is homeomorphic to $\left(\mathbb R\setminus\mathbb Q\right)\times\mathbb R^{\text{D}-2}$ where $\text{D}$ is the dimension of $\text{\rm Poly}_d({\mathbb{R}^2})$. The leaves of the laminations are analytic and the sinks move simultaneously along the leaves.
 \end{prop}
  \vskip .2 cm

The proof needs some preparation which is given in the following lemmas. We first explain the steps of the proof. The main obstacle is that the curve of sinks $A_n$ is only smooth and not analytic. This is a consequence of the fact that its construction relies on the linearization of the saddle. We replace the curve $A_n$ by the graph of an analytic function $sa_n$ characterized by the property that each sink corresponding to parameters in $sa_n$ has trace zero. The construction of $sa_n$, defined on a fixed parameters domain $\mathbb D\times \mathbb D^{T}$, relies on Lemma \ref{datrace}. Moreover its analyticity is due to fact that it is defined by the equation $\text{trace}=0$. In Lemma \ref{dsandt} we prove that the curve $sa_n$ is very close to the curve $A_n$ and as corollary we obtain that it intersects transversally each curve of secondary tangencies, $b_{n,n_0}$. Moreover the curves $b_{n,n_0}$ naturally extend to holomorphic functions. By adding any number of new parameters $\tau=(\tau_i)$, the intersection point between $sa_n$ and $b_{n,n_0}$ is the graph of an holomorphic function defined on a fixed domain $\mathbb D^{T}$. The real parts of the limits of these graphs are the real-analytic leaves of the Newhouse lamination. 

\bigskip

Consider a polynomial unfolding $F_{t,a}$ of a strong homoclinic tangency in $\text{Poly}_d({\mathbb{R}^2})$, say with $(t,a)\in (-1,1)\times (-1,1)$. Consider this unfolding as a family 
$$\mathbb{D}\times \mathbb{D}\ni (t,a)\mapsto F_{t,a}\in \text{Poly}_d({\mathbb{C}^2}).
$$
Assume also that the unfolding is contained in a larger polynomial family
$$\mathbb{D}\times \mathbb{D}\times \mathbb{D}^T\ni (t,a, \tau)\mapsto F_{t, a, \tau}\in \text{Poly}_d({\mathbb{C}^2}).
$$
There is a local holomorphic change of coordinates such that the saddle point becomes $(0,0)$, the local stable manifold contains the unit disc in the $x$-axis, and the local unstable manifold contains the unit disc in the $y$-axis. Moreover, the restriction of the map to the invariant manifolds is linearized, that is 
\begin{equation}\label{semilinearization}
F(x,0)=(\lambda_1 x,0) \text{ and } F(0,y)=(0, \mu y).
\end{equation}
The domain $\mathbb{D}\times \mathbb{D}$ where (\ref{semilinearization}) holds, is called the domain of semi-linearization.
The change of coordinates depends holomorphically on the parameters.  Observe that, when $(x,y)\in \mathbb{D}\times \mathbb{D}$, then
\begin{equation}\label{DFnCC}
DF(x,y)=\left(\begin{matrix}
\lambda_1+O(y) & O(x)\\
O(y) & \mu +O(x)
\end{matrix}\right).
\end{equation}
\begin{rem}\label{semilin} By shrinking the semi-linearization domain we may assume that the error terms are small, say $\max\left\{O(x), O(y)\right\}\le 1/2$.
\end{rem}
\begin{lem}\label{DFnC} If $(x,y)\in \mathbb{D}\times \mathbb{D}$ and $F^i(x,y)\in \mathbb{D}\times \mathbb{D}$, for $i\le n$, then
$$
DF^n(x,y)=\left(\begin{matrix}
a_{11} \lambda_1^n \mu^n& a_{12} \lambda_1^n \mu^n\\
a_{21} & a_{22} \mu^n
\end{matrix}\right),
$$
where $a_{kl}$ are uniformly bounded holomorphic functions and $a_{22}\ne 0$ and uniformly away from zero.
\end{lem}

\begin{proof} Let $k\le n$ and 
$$
DF^k(x,y)=\left(\begin{matrix}
a_{11}(k) \lambda_1^k\mu^k& a_{12}(k) \lambda_1^k \mu^k\\
a_{21}(k) & a_{22}(k) \mu^k
\end{matrix}\right),
$$
and
$F^k(x,y)=(x_k,y_k)$. Observe, $x_k=O(\lambda_1^k)$ and $y_k=O\left(\frac{1}{\mu^{n-k}}\right)$.
Then
\begin{equation}\label{a11}
a_{11}(k+1)=\frac{1}{\mu}\left(1+O(y_k)\right)a_{11}(k)+O(\mu^{-k}) a_{21}(k),
\end{equation}
and
\begin{equation}\label{a21}
a_{21}(k+1)= O(y_k(\lambda_1\mu)^k) a_{11}(k)+ \mu(1+O(\lambda_1^k))a_{21}(k).
\end{equation}
We may restrict ourselves to $k_0\le k\le k_1=n-n_0$ with $k_0, n_0$ fixed but large. From (\ref{a21}) we get, using that $a_{21}(0)=0$, 
\begin{equation}\label{a21O}
a_{21}(k)=O\left(\sum_{i<k} \mu^{k-i} y_i (\lambda_1\mu)^ia_{11}(i)\right) =O\left(\sum_{i<k} \frac{1}{\mu^{n-k}}  (\lambda_1\mu)^ia_{11}(i)\right).
\end{equation}
Let $M_k=\max_{i\le k} |a_{11}(i)|$. Then, using the last estimate for $a_{21}(k)$ and (\ref{a11}),
$$
M_{k+1}\le \left[\frac{1}{\mu}(1+O(y_k))+ 
O\left(\sum_{i<k}  (\lambda_1\mu)^i\right) \frac{1}{\mu^{n}}\right] M_k\le M_k,
$$
when $n_0$ is large enough. Hence,   the entries $a_{11}$ and $a_{21}$, using (\ref{a21O}), are uniformly bounded. 
Similarly,
\begin{equation}\label{a12}
a_{12}(k+1)=\frac{1}{\mu}\left(1+O(y_k)\right)a_{12}(k)+
O(a_{22}(k)),
\end{equation}
and
\begin{equation}\label{a22}
a_{22}(k+1)= O(y_k\lambda_1^k) a_{12}(k)+ (1+O(\lambda_1^k))a_{22}(k).
\end{equation}
From (\ref{a22}) and the fact that $a_{22}(0)=1$, we get
\begin{equation}\label{a22O}
a_{22}(k)=O\left( 1+\frac{1}{\mu^n} \sum_{i<k}  (\lambda_1\mu)^i    a_{12}(i) \right).
\end{equation}
This estimate and (\ref{a12}) imply that $\max_{i\le k} |a_{12}(i)|$ is bounded. In particular, the entries $a_{12}(k)$ are uniformly bounded. 
The estimate (\ref{a22O}) implies that also $a_{22}(k)$ are uniformly bounded. By shrinking the semi-linearization domain, see Remark \ref{semilin}, and by taking $n$ large enough, (\ref{a22}) assures that $a_{22}=a_{22}(n)$ stays away from zero.
\end{proof}

Choose  a parameter $(t,a,\tau)$ and assume that there is a periodic point $p$ in the domain of semi-linearization which returns in $N$ steps into the domain of semi-linearization and then needs $n$ steps inside to return to itself. Let $$
(t,a, \tau)\mapsto \text{tr} DF^{N+n}_p.
$$
Observe, if $\text{\rm tr} DF^{N+n}_p=0$ then, for $n\ge 1$ large enough,  the periodic orbit of $p$ is attractive, called {\it strong sink}. Here we use that the dimension of the phase space is two. The trace at a periodic point is invariant under smooth coordinate change. The proof of the following lemma will be in the coordinates of semi-linearization.

\begin{lem}\label{datrace} There exists $K>0$ such that the following holds. If 
$
\text{\rm tr} DF^{N+n}_p=0
$
then
$$
\frac{1}{K}|\mu|^{2n}\le \left|\frac{\partial }{\partial a}\left(\text{\rm tr} DF^{N+n}_p\right)\right|\le K|\mu|^{2n}.
$$
\end{lem}

\begin{proof} From Lemma \ref{DFnC}, using the fact that the trace is zero we get 
$$
DF^{N+n}_p=DF^{n}_{F^N\left(p\right)}DF^{N}_p=\left(\begin{matrix}
O\left((\lambda_{1}\mu)^n\right) & O\left((\lambda_{1}\mu)^n\right)\\
O\left(\mu ^n\right)& O\left((\lambda_{1}\mu)^n\right)
\end{matrix}\right).
$$
The periodic point $p=(p_x, p_y)$ has coordinates $\left(p_x, p_y\right)\in \mathbb{D}\times \mathbb{D}$. We claim, by differentiating with respect to $a$ the $x$-component of the equation $F^{N+n}(p_x, p_y)=(p_x, p_y)$ that
\begin{equation}\label{dpx}
(1+O((\lambda_1\mu)^n))\frac{\partial p_x}{\partial a}= O((\lambda_1\mu)^n)\frac{\partial p_y}{\partial a}+ \frac{\partial F^{N+n}_x}{\partial a}=
O\left((\lambda_1\mu)^n\right)\frac{\partial p_y}{\partial a}+O\left((\lambda_1\mu)^n\right).
\end{equation}
Correspondingly, for the $y$-component, we claim
\begin{equation}\label{dpy}
(1+O((\lambda_1\mu)^n))\frac{\partial p_y}{\partial a}=
O(\mu^n) \frac{\partial p_x}{\partial a}+\frac{\partial F^{N+n}_y}{\partial a}=
O(\mu^n) \frac{\partial p_x}{\partial a}+K\mu^n,
\end{equation}
where $K>0$ and bounded away from zero. The equations for ${\partial F^{N+n}_x}/{\partial a}$ and ${\partial F^{N+n}_y}/{\partial a}$ are obtained as follows. 
Observe that 
\begin{eqnarray}\label{dFNnday}
\nonumber
\frac{\partial F^{N+n}_y}{\partial a}&=&a_{21}\frac{\partial F^{N}_x}{\partial a}(p)+a_{22}\mu^n\frac{\partial F^{N}_y}{\partial a}(p)+\frac{\partial F^{n}_y}{\partial a}(F^N(p))\\\nonumber
&=&O(1)+a_{22}\mu^n+\frac{\partial }{\partial a}\int_0^{F_y^N(p)}\left[DF^n\left(F^N_x(p)\right)\left(\begin{matrix}
0\\1
\end{matrix}\right)     
\right]_ydy\\\nonumber
&=&O(1)+a_{22}\mu^n+\frac{\partial }{\partial a}\int_0^{F_y^N(p)}a_{22}\left(F^N_x(p),y\right)\mu^n dy
\\\nonumber
&=&O(1)+a_{22}\mu^n+O\left(n\mu^n F_y^N(p)\right)\\
&=&K\mu^n,
\end{eqnarray}
where we used that $ F^i(F_y^N(p))$ for $i<n$ is in the domain of linearization, namely $F_y^N(p)=O\left(1/\mu^n\right)$ and $O\left(n\mu^n F_y^N(p)\right)=O(n)$.
Similarly,
\begin{eqnarray}\label{dFNndax}
\nonumber
\frac{\partial F^{N+n}_x}{\partial a}&=&a_{11}(\lambda_1\mu)^n\frac{\partial F^{N}_x}{\partial a}(p)+a_{12}(\lambda_1\mu)^n\frac{\partial F^{N}_y}{\partial a}(p)+\frac{\partial F^{n}_x}{\partial a}(F^N(p))\\\nonumber
&=&O((\lambda_1\mu)^n)+\frac{\partial }{\partial a}\left(\int_0^{F_y^N(p)}a_{12}\left(F^N_x(p),y\right)(\lambda_1\mu)^n dy+\lambda_1^n F^N_x(p)\right)
\\\nonumber
&=&O((\lambda_1\mu)^n)+O(n\lambda_1^n)\\
&=&O((\lambda_1\mu)^n),
\end{eqnarray}
where we used that $ F_y^N(p)=O(1/\mu^n)$.
From (\ref{dpx}), (\ref{dpy}) and the fact that $\lambda_1\mu^2<1$, see $(F3)$, we have
\begin{equation}\label{partial pya}
\frac{1}{2K} |\mu|^n\le \left|\frac{\partial p_y}{\partial a}\right|\le 2K |\mu|^n,
\end{equation} 
and
\begin{equation}\label{partial pxa}
\left|\frac{\partial p_x}{\partial a}\right|=O\left(\left(\lambda_1\mu^2\right)^n\right).
\end{equation}
Observe that, by Lemma \ref{DFnC},
$$
\text{tr} DF^{N+n}_p=\tilde{A}(p_x,p_y,t,a)(\lambda_1\mu)^n+a_{22}(p_x,p_y,t,a)D(p_x,p_y,t,a)\mu^n+a_{21}B(p_x,p_y,t,a),
$$
where $D$ is the entry $(DF^N_p)_{22}$, which tends to zero when $n$ gets large, $B=(DF^N_p)_{12}$ and the factors $\tilde A$ are $\Cuno$ uniformly bounded when $n$ gets large. 
Hence,
\begin{eqnarray}\label{tracefnN}
\nonumber \frac{\partial }{\partial a}\left(\text{tr} DF^{N+n}_p\right)&= 
 &\left[\frac{\partial \tilde{A}}{\partial x}\frac{\partial p_x}{\partial a}+ \frac{\partial \tilde{A}}{\partial y}\frac{\partial p_y}{\partial a}+\frac{\partial \tilde{A}}{\partial a}\right] (\lambda_1\mu)^n+ n\tilde{A}(\lambda_1\mu)^{n-1} \frac{\partial \lambda_1\mu}{\partial a}\\\nonumber &+
 &\left[\frac{\partial \left(a_{22} D\right)}{\partial x}\frac{\partial p_x}{\partial a}+ \frac{\partial \left(a_{22}D\right)}{\partial y}\frac{\partial p_y}{\partial a}+\frac{\partial \left(a_{22}D\right)}{\partial a}\right] \mu^n\\\nonumber &+
  &
 na_{22}D\mu^{n-1} \frac{\partial \mu}{\partial a}+\frac{\partial \left(a_{21}B\right)}{\partial a}\\
 &=& O(n|\mu|^n) +\frac{\partial \left(a_{22}D\right)}{\partial y}\frac{\partial p_y}{\partial a}\mu^{n},
 \end{eqnarray}
 where we used (\ref{partial pya}), (\ref{partial pxa}) and $a_{22}D\mu^n=-\tilde{A}(\lambda_1\mu)^n$. Observe, 
 $$\frac{\partial \left(a_{22}D\right)}{\partial y}=
 \frac{\partial a_{22}}{\partial y}D+a_{22}\frac{\partial D}{\partial y}
 $$ 
 is bounded away from zero because of the following properties. First, $D$ tends to zero, $a_{22}$ is bounded away from zero, see Lemma \ref{DFnC}, and ${\partial D}/{\partial y}$ is away from zero 
 because the family is an unfolding of a non-degenerate tangency. The lemma follows from (\ref{tracefnN}) and (\ref{partial pya}).
 \end{proof}

Choose a real parameter $\tau$ and $(t, a_n(t))\in \mathcal{A}_n$. Let $\tilde{p}=(\tilde p_x,\tilde p_y)$ be the corresponding sink of period $N+n$ near $z$. Observe that, for parameters $(t, a_n(t))$, Lemma \ref{periodicpoint} can be refined to obtain an invariant square of size of order $\lambda_1^n$. Hence $|\tilde p_y-z_y|=O(|\lambda_1|^n)$. 
This implies, by calculating the $\text{\rm tr} DF^{N+n}_p$ in the smooth linearization, that
\begin{equation}\label{trdfnN}
\text{\rm tr} DF^{N+n}_p=O\left(\left(\lambda_1 \mu\right)^n\right).
\end{equation}
Lemma \ref{datrace} and (\ref{trdfnN}) implies that, for some 
$$
\Delta a=O\left(\left(\lambda_1 \mu\right)^n\right) \frac{\epsilon_0}{|\mu(t, a_n(t)|^{2n}},
$$ 
the sink at $(t, a_n(t)+\Delta a)\in \mathcal{A}_n$  is a strong sink. In particular, there is an analytic global function $t\mapsto sa_n(t)$ such that $(t, sa_n(t))\in \mathcal{A}_n$ and the sink is a strong sink.
\\
Consider the periodic point of period $N+n$ at parameter $(t,sa_n(t))$. This periodic orbit spend $n$ step in the domain of semi linearization and needs $N$ steps to return to the domain of semi linearization. Moreover it has trace zero. Consider the maximal domain of parameters in $\mathbb{D}\times \mathbb{D}^T$ to which the function $sa_n(t,\tau)$ extends holomorphically. On the boundary of this domain the map still has a periodic point of period $n$ with trace zero and which spend $n$ step in the domain of semi linearization and needs $N$ steps to return to it. Lemma \ref{datrace} implies that the function $sa_n$ has a local holomorphic extension in this boundary point. This implies that the map $sa_n$ has an holomorphic extension
\begin{equation}\label{sandescr}
sa_n: \mathbb{D}\times \mathbb{D}^T\to \mathbb{C}.
\end{equation}
Along the graph of $sa_n$, in parameters of the form $(t,sa_n(t,\tau),\tau )$ the periodic point is a strong sink.

\begin{lem} \label{dsandt}
The slope of the strong sink curve is of the form
\begin{equation}
\frac{d sa_n}{d t}=-n\frac{\partial\mu}{\partial t}\frac{1}{\mu^{n+1}}\left[1+O\left(\frac{n}{\mu^n}\right)\right].
\end{equation}
\end{lem}
\begin{proof}
In the proof of Lemma \ref{datrace} we used an holomorphic coordinate change. However the trace of a periodic point is invariant under coordinate change and we will use the smooth linearizing coordinates to calculate the slope of the curve of strong sinks. Choose a parameter pair $(t,a)$ on the curve $sa_n$ and denote the corresponding periodic point by $p=(p_x, p_y)$. Observe that 
\begin{equation}\label{py}
p_y=\left[1+O\left(\left(\lambda_1\mu\right)^n\right)\right].
\end{equation}
Moreover
$$
DF^{N+n}_p=\left(\begin{matrix}
A\lambda_1^n & B\lambda_1^n\\
C\mu^n& D\mu^n
\end{matrix}\right),
$$
where, because the trace is zero,
\begin{equation}\label{tracezero}
D\mu^n=-A\lambda_1^n,
\end{equation}
and because (\ref{FNexpr}), $D=O\left(\Delta x,\Delta y\right)$ in coordinate centered in $(0,1)$  and hence 
\begin{equation}\label{dDdt}
\frac{\partial D}{\partial t}=O\left(\left(\lambda_1\mu\right)^n\right),
\end{equation}
where we used that $p_x=O\left(\lambda_1^n\right)$ and (\ref{py}).
From $F^{N+n}(p_x, p_y)=(p_x, p_y)$ one obtains, by differentiating with respect to $a$ in the $x$-direction,
$$
\left(1-A\lambda_1^n\right )\frac{\partial p_x}{\partial a}= B\lambda_1^n\frac{\partial p_y}{\partial a}+ \frac{\partial F^{N+n}_x}{\partial a}=
B\lambda_1^n\frac{\partial p_y}{\partial a}+O\left(\left(\lambda_1\mu\right)^n\right),
$$
and in the $y$-direction, using (\ref{tracezero})
$$
\left(1+A\lambda_1^n\right )\frac{\partial p_y}{\partial a}=
C\mu^n\frac{\partial p_x}{\partial a}+\frac{\partial F^{N+n}_y}{\partial a}=
C\mu^n\frac{\partial p_x}{\partial a}+\mu^n+O\left(n\right).
$$
where we used (\ref{dFNndax}) and (\ref{dFNnday}).
 This implies
\begin{equation}\label{partial pysa}
\frac{\partial p_y}{\partial a}=\mu^n+O\left(n\right),
\end{equation}
and
\begin{equation}\label{partial pxsa}
\frac{\partial p_x}{\partial a}=O\left(\left(\lambda_1\mu\right)^n\right).
\end{equation}
Observe that,
$$
\text{tr} DF^{N+n}_p=A\lambda_1^n+D\mu^n.
$$
Hence,
\begin{equation}\label{dtraceda}
\begin{aligned}
 \frac{\partial }{\partial a}\left(\text{tr} DF^{N+n}_p\right)= 
 &\left[\frac{\partial {A}}{\partial x}\frac{\partial p_x}{\partial a}+ \frac{\partial {A}}{\partial y}\frac{\partial p_y}{\partial a}+\frac{\partial {A}}{\partial a}\right] \lambda_1^n+ n{A}\lambda_1^{n-1} \frac{\partial \lambda_1}{\partial a}\\+
 &\left[\frac{\partial D}{\partial x}\frac{\partial p_x}{\partial a}+ \frac{\partial D}{\partial y}\frac{\partial p_y}{\partial a}+\frac{\partial D}{\partial a}\right] \mu^n+
 nD\mu^{n-1} \frac{\partial \mu}{\partial a}\\
 =& \frac{\partial D}{\partial y}\mu^{2n}+O(n\mu^n),
 \end{aligned}
 \end{equation}
 where we used (\ref{partial pysa}), (\ref{partial pxsa}), and (\ref{tracezero}). 
 By differentiating by $t$ in the $x$-direction,
$$
\left(1-A\lambda_1^n\right )\frac{\partial p_x}{\partial t}= B\lambda_1^n\frac{\partial p_y}{\partial t}+ \frac{\partial F^{N+n}_x}{\partial t}=
B\lambda_1^n\frac{\partial p_y}{\partial t}+O\left(n\lambda_1^n\right),
$$
and in the $y$-direction, using (\ref{tracezero}) and (\ref{py})
\begin{eqnarray*}
\left(1+A\lambda_1^n\right )\frac{\partial p_y}{\partial t}&=&
C\mu^n\frac{\partial p_x}{\partial t}+\frac{\partial F^{N+n}_y}{\partial t}=
C\mu^n\frac{\partial p_x}{\partial t}+\frac{n}{\mu}\frac{\partial\mu}{\partial t}\mu^n \left(F^N_y(p)\right)_y+O\left(\left(\lambda_1\mu\right)^n\right)\\
&=&C\mu^n\frac{\partial p_x}{\partial t}+\frac{n}{\mu}\frac{\partial\mu}{\partial t}\left[1+O\left(\left(\lambda_1\mu\right)^n\right)\right].
\end{eqnarray*}
 This implies
\begin{equation}\label{partial pyst}
\frac{\partial p_y}{\partial t}=\frac{n}{\mu}\frac{\partial\mu}{\partial t}\left[1+O\left(\left(\lambda_1\mu\right)^n\right)\right],
\end{equation}
and
\begin{equation}\label{partial pxst}
\frac{\partial p_x}{\partial t}=O\left(n\lambda_1^n\right).
\end{equation}
Hence,
\begin{equation}\label{dtracedt}
\begin{aligned}
 \frac{\partial }{\partial t}\left(\text{tr} DF^{N+n}_p\right)= 
 &\left[\frac{\partial {A}}{\partial x}\frac{\partial p_x}{\partial t}+ \frac{\partial {A}}{\partial y}\frac{\partial p_y}{\partial t}+\frac{\partial {A}}{\partial t}\right] \lambda_1^n+ n{A}\lambda_1^{n-1} \frac{\partial \lambda_1}{\partial t}\\+
 &\left[\frac{\partial D}{\partial x}\frac{\partial p_x}{\partial t}+ \frac{\partial D}{\partial y}\frac{\partial p_y}{\partial t}+\frac{\partial D}{\partial t}\right] \mu^n+
 nD\mu^{n-1} \frac{\partial \mu}{\partial t}\\
 =& \frac{\partial D}{\partial y}\frac{n}{\mu}\frac{\partial\mu}{\partial t}\left[1+O\left(\left(\lambda_1\mu\right)^n\right)\right],
 \end{aligned}
 \end{equation}
 where we used (\ref{partial pyst}), (\ref{partial pxst}), (\ref{tracezero}) and (\ref{dDdt}). The lemma follows by calculating the gradient of the trace using (\ref{dtraceda}) and  (\ref{dtracedt}).
\end{proof}
Choose a real parameter $\tau$. Each curve $b_{n,n_0}$ crosses the graph of strong sinks, see Proposition \ref{angle}. Namely, the real part of the graph of $sa_n$ is contained in the strip $\mathcal{A}_n$, which is contained in the strip $\mathcal{B}_n$. Hence, there are secondary tangencies where the corresponding sink is a strong sink.
The next corollary follows from Lemma \ref{dsandt}, Proposition \ref{angle} and (\ref{thetacond}).
\begin{cor}\label{sanbntransv}
The curves $sa_n$ and $b_{n,n_0}$ intersect transversally.
\end{cor}

Choose a parameter $(t_0,sa_n(t_0,\tau_0),\tau_0)\in \mathbb{D}\times \mathbb{D}\times \mathbb{D}^T$ where the secondary tangency coexists with a strong sink. We may use again a local holomorphic change of coordinates such that the stable manifold of the saddle point coincides locally with the $x$-axis and the unstable manifold locally with the $y$-axis.  Let $(0,y_0)\in W^u_\text{loc}((0,y_0))$ such that $F^{3N+(1+\theta)n}(0,y_0)$ is the secondary tangency. The projection to the $y$-axis is as before denoted by an index $y$. Consider a small disc $\mathbb{D}_\rho$ of radius $\rho>0$ centered around $y_0$ and contained in $W^u_{\text{loc}}(y_0)$. 

\begin{lem} \label{degreetwomap} For parameters near $(t_0,sa_n(t_0,\tau_0),\tau_0)$ the map
$$
\mathbb{D}_\rho\ni y\mapsto \left(F^{3N+(1+\theta)n}(0,y_0)\right)_y\in \mathbb{C}
$$
has degree two.
\end{lem}

\begin{proof} The map 
$$
\mathbb{D}_\rho\ni y\mapsto \left(F^{N+\theta n}(0,y_0)\right)_y\in \mathbb{C}
$$
is univalent. The map 
$$
\mathbb{D}_\rho\ni y\mapsto \left(F^{2N+\theta n}(0,y_0)\right)_y\in \mathbb{C}
$$
has degree two. This is where the folding happens to create the tangency. The next stage will not cause more folding and the map 
$$
\mathbb{D}_\rho\ni y\mapsto \left(F^{2N+(1+\theta) n}(0,y_0)\right)_y\in \mathbb{C}
$$
has still degree two. Finally,
the last stage will not cause more folding and the map 
$$
\mathbb{D}_\rho\ni y\mapsto \left(F^{3N+(1+\theta) n}(0,y_0)\right)_y\in \mathbb{C}
$$
has degree two.
\end{proof}
\noindent{\it Proof of Proposition \ref{analyticleaves}.}
Consider the functions $sa_n$ as described in \ref{sandescr}. From Lemma \ref{dsandt} we know that 
\begin{equation}\label{sanbound}
\left|{sa_n}\right|=O\left(\frac{n}{|\mu|_{\text{min}}^n}\right).
\end{equation}
For any parameter $(t, sa_n(t,\tau), \tau)$ near $(t_0,sa_n(t_0,\tau_0),\tau_0)$ the image of 
$$
\mathbb{D}_\rho\ni y\mapsto F_{t, sa_n(t,\tau), \tau}^{3N+(1+\theta)n}(0,y)
$$
intersects twice $W_{n-n_0}$ which is a graph of a holomorphic function. At $(t_0,sa_n(t_0,\tau_0),\tau_0)$ the two intersections coincide to form the non-degenerate tangency. Consider the maximal extension of the function $b_{n,n_0}(t,\tau)$ to a certain complex domain $U\in\mathbb D\times \mathbb D^{T}$. Observe that, if $(t,\tau)\in U$, then in the parameter $(t,b_{n,n_0}(t,\tau),\tau)$ there is a non degenerate homoclinic tangency, see Lemma \ref{degreetwomap}. Hence the function $b_{n,n_0}(t,\tau)$ can be extended to a neighborhood and as consequence, $U$ is open and closed. In particular $U=\mathbb D\times \mathbb D^{T}$.
\\
The set of parameters $\left\{(t,a,\tau)\right\}$ where a non-degenerate tangency and a strong sink coexist is given by two equations: $a=sa_n(t,\tau)$ and $a=b_{n,n_0}(t,\tau)$. Hence, this set forms a branched cover over the $\tau$-domain. The transversality in Corollary \ref{sanbntransv} implies that there are no branch points in the real slice. Moreover, from the fact that, in the real slice, there is a unique parameter corresponding to a map where a tangency and a strong sink coexist, see Proposition \ref{angle},  it follows that this branched cover is actually the graph of a holomorphic function 
$$
\mathbb{D}^T\ni \tau\mapsto l_{n,n_0}(\tau)=(t(\tau), a(\tau)).
$$ 
In particular, in the parameter $(l_{n,n_0}(\tau), \tau)$ the map has a non-degenerate secondary tangency and a strong sink. Observe that the points $(t(\tau), a(\tau),\tau)$ are in the graph of $sa_n$ and because of (\ref{sanbound}) they are uniformly bounded. For real parameters $\tau\in (-1,1)^T$, $l_{n,n_0}(\tau)$ selects in an analytic manor a strong sink along the secondary tangency curve $b_{n,n_0}$ at $\tau$.
\\
The leaves of the lamination $L_F$, as constructed in the proof of Theorem $B$ are the uniform limits of the real part of the graphs of holomorphic functions
$$
\mathbb{D}^T\ni \tau\mapsto l_{n^{(k)},n^{(k)}_0}(\tau)=(t(\tau), a(\tau))\in \mathbb{D}\times \mathbb{D},
$$
which are all defined in the same uniform domain $\mathbb{D}^T$.
Hence, by a normal family argument, the leaves of $L_F$ are analytic. This finishes the proof of Proposition \ref{analyticleaves}.
\qed
\begin{rem}
Observe that each leave of the lamination $L_F$ is the real part of the graph of an holomorphic function $\ell:\mathbb{D}^T\to \mathbb{D}\times\mathbb{D}$.
\end{rem}
\begin{figure}
\centering
\includegraphics[width=0.76\textwidth]{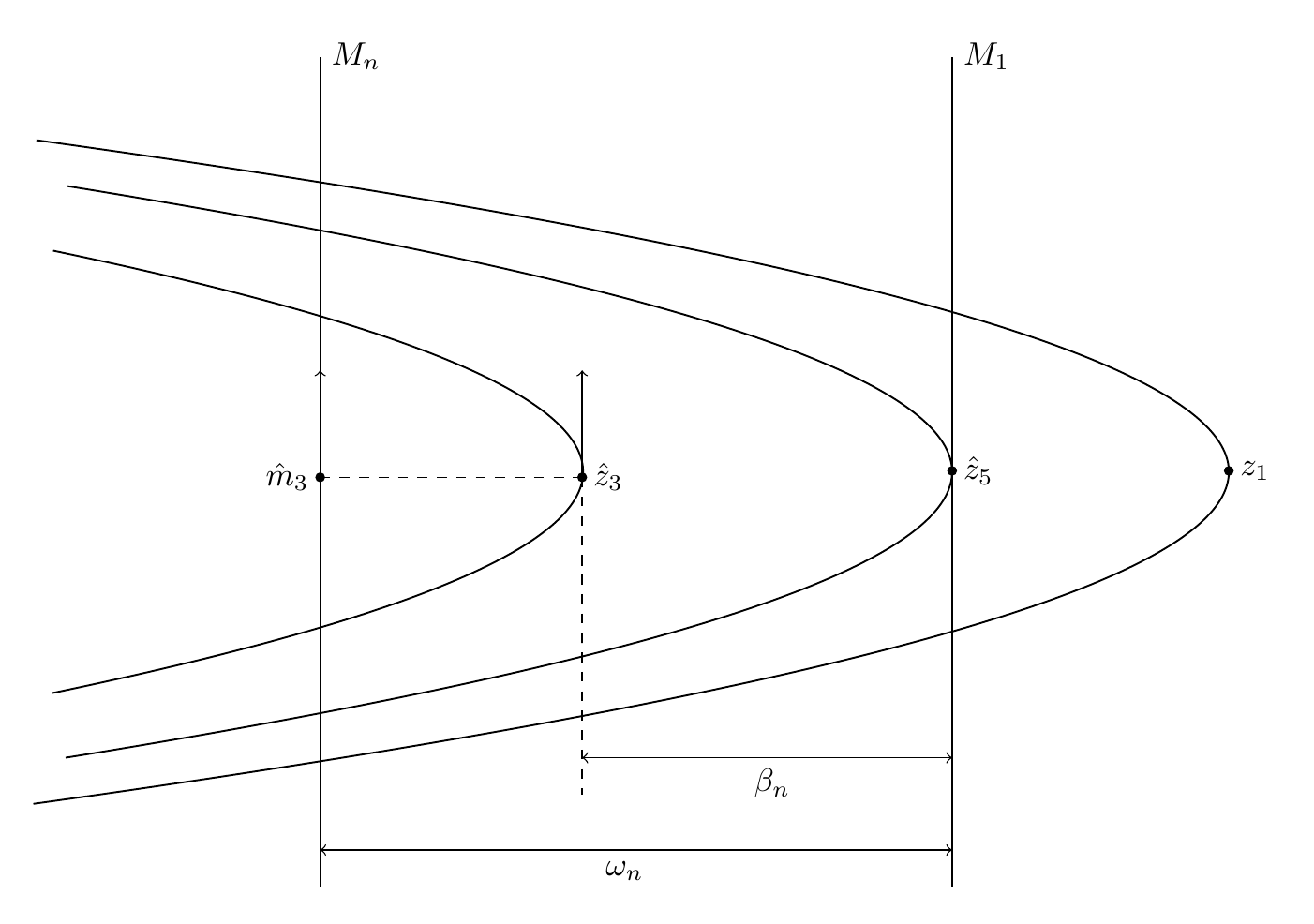}
\caption{Illustration to the proof of Proposition \ref{doubletg}}
\label{Fig9}
\end{figure}
\section{Coexistence of sinks with non-periodic attractors}

In this section we will construct quadratic H\'enon-like maps which have infinitely many sinks and a non-periodic attractor.  The coexistence of sinks with a period doubling attractor is the first example.

\begin{defin}\label{cantorA} Let $M$ be a manifold and $f:M\to M$. An invariant Cantor set $A\subset M$ is called a period doubling Cantor attractor of $f$ if $f|A$ is conjugated to a $2$-adic  adding machine and there is a neighborhood $M\supset U\supset A$ such that the orbit of almost every point in $U$ accumulates at $A$.
\end{defin}

\begin{rem} A period doubling Cantor attractor has zero topological entropy. It carries a unique invariant probability measure. Strongly dissipative H\'enon-like maps at the boundary of chaos have period doubling Cantor attractors, see \cite{CLM}.
\end{rem}

The existence of a double tangency was shown already in  \cite{BP}.  Here we need the following version. 

\begin{prop}\label{doubletg}
There exist two analytic functions $a_{1,2}:[b_{-},b_{+}]\mapsto\R$ such that, for all $b\in [b_{-},b_{+}]$, each H\'enon map $F_{a_1(b),b}$ and $F_{a_2(b),b}$ has a strong homoclinic tangency which depends analytically on $b$. The graphs of the functions $a_{1,2}$ have a unique intersection at $b=b_0$. Moreover the intersection is transversal.
\end{prop}

\begin{proof}
Choose a parameter $(a,b)$ close to $(2,0)$. The H\'enon map $F_{a,b}$ has two saddle points. We consider the saddle point $p$ in the positive quadrant where the unstable eigenvalue is negative. Observe that the map $F^2_{a,b}$, for small $b$ and $a$ sufficiently large has a full horse shoe. Let $z_1$, $z_2$, $z_3$, $z_4$ and $z_5$ be the first five points on the right leg of the unstable manifold of $p$ where the tangent space is vertical. 
\\
Consider the local stable manifold at $p$, $M_0$, of unit length. This is an almost vertical curve whose slope is of order of $b$. Let $M_1$ be the local preimage of $M_0$ of unit length around $z_5$. Again this is an almost vertical curve whose slope is of order of $b$. Let $M_n$ be a segment of the stable manifold of $p$ near $z_3$ of unit length  such that $(F^2)^n(M_n)\subset M_0$. Indeed there are many such components. However with an appropriate choice we can assure that each curve is an almost vertical curve of unit length and that the horizontal distance between $M_n$ and $M_1$ is proportional to ${1}/{|\mu|^n}$, where $\mu$ is the square of the unstable eigenvalue of $p$. 
\\
We are now ready to define the $a_1$ curve of tangencies. Namely, choose $b$ small and $a$ large, then $z_5$ is to the right of the curve $M_1$. By diminishing $a$ we find a point, $\hat z_5$, near $z_5$, where the local unstable manifold of $z_5$ is tangent to $M_1$. This tangency persists along a curve of the form $(a_1(b),b)$ where $a_1:[0,b_1]\ni b\mapsto a_1(b)$ is an analytic function, see Figure \ref{Fig9}. 
\\
Take now $b\in [0,b_1]$ and consider the $z_3$ point associate to the map $F_{a_1(b),b}$. Let $n$ be maximal such that $M_n\cap W^u_\text{loc}(z_3)\neq\emptyset$ and let $\hat z_3$ be in the local unstable manifold of $z_3$ such that $$T_{\hat z_3}W^u_{\text{loc}}(z_3)=T_{\hat m_3}M_n,$$ where $\hat m_{3,y}=\hat z_{3,y}$. Let 
$$\omega_n=\hat z_{5,x}-\hat m_{3,x},$$ and 
$$\beta_n=\hat z_{5,x}-\hat z_{3,x},$$ see Figure \ref{Fig9}. Observe that $\omega_n-\beta_n\geq 0$ and when $\hat z_3$ is an homoclinic tangency in $M_n$ then $\beta_n-\omega_n=0$. Moreover 
\begin{equation}\label{eq0newprop}
\beta_n=C b^4\text{ and }\omega_n=\frac{C'}{|\mu|^n},
\end{equation}  where $C$ and $C'$ depend analytically  on the parameter $b$ and they are uniformly bounded away from zero. As a consequence we get
\begin{equation}\label{eq1newprop}
\frac{1}{|\mu|^n}=K b^4,
\end{equation}
 where $K$ is bounded and uniformly bounded away from zero. Moreover we have that 
$$
\frac{d}{db}(\omega_n-\beta_n)=\frac{d}{db}\left(\frac{C'}{|\mu|^n}-Cb^4\right)=O\left(\frac{n}{\mu^n}\right)-4b^3C+O\left(b^4\right).
$$
By using (\ref{eq1newprop}), we get
\begin{equation}\label{eq2newprop}
\frac{d}{db}(\omega_n-\beta_n)=O\left(\log\left(\frac{1}{b}\right)b^4\right)-4b^3C\leq -3C b^3<0,
\end{equation}
for $b$ small enough.
Because $\omega_n-\beta_n=K_1 b^4$, see (\ref{eq0newprop}) and (\ref{eq1newprop}), there exists $\Delta b>0$ of order $b$, such that $\omega_n=\beta_n$ at $b+\Delta b$. Hence, for every $b$ there exists $b_0=O\left(b\right)$ such that $\hat z_3$ is an homoclinic tangency for the parameter $(a_1(b_0),b_0)$. This tangency persists along an analytic curve $a_2:b\mapsto (a_2(b),b)$ defined in a neighborhood of $b_0$. Because of (\ref{eq2newprop}), the graphs of $a_1$ and $a_2$ intersect transversally.
\end{proof}
A diffeomorphism of the plane $\mathbb{R}^2$ of the form 
\begin{equation}\label{cubichenon}
F_{a,b,\tau}\left(\begin{matrix}
x\\y
\end{matrix}
\right)=\left(\begin{matrix}
a-x^2-by+\tau y^2
\\
x
\end{matrix}
\right)
\end{equation}
is called a {\it quadratic H\'enon-like map}.
\begin{figure}
\centering
\includegraphics[width=0.76\textwidth]{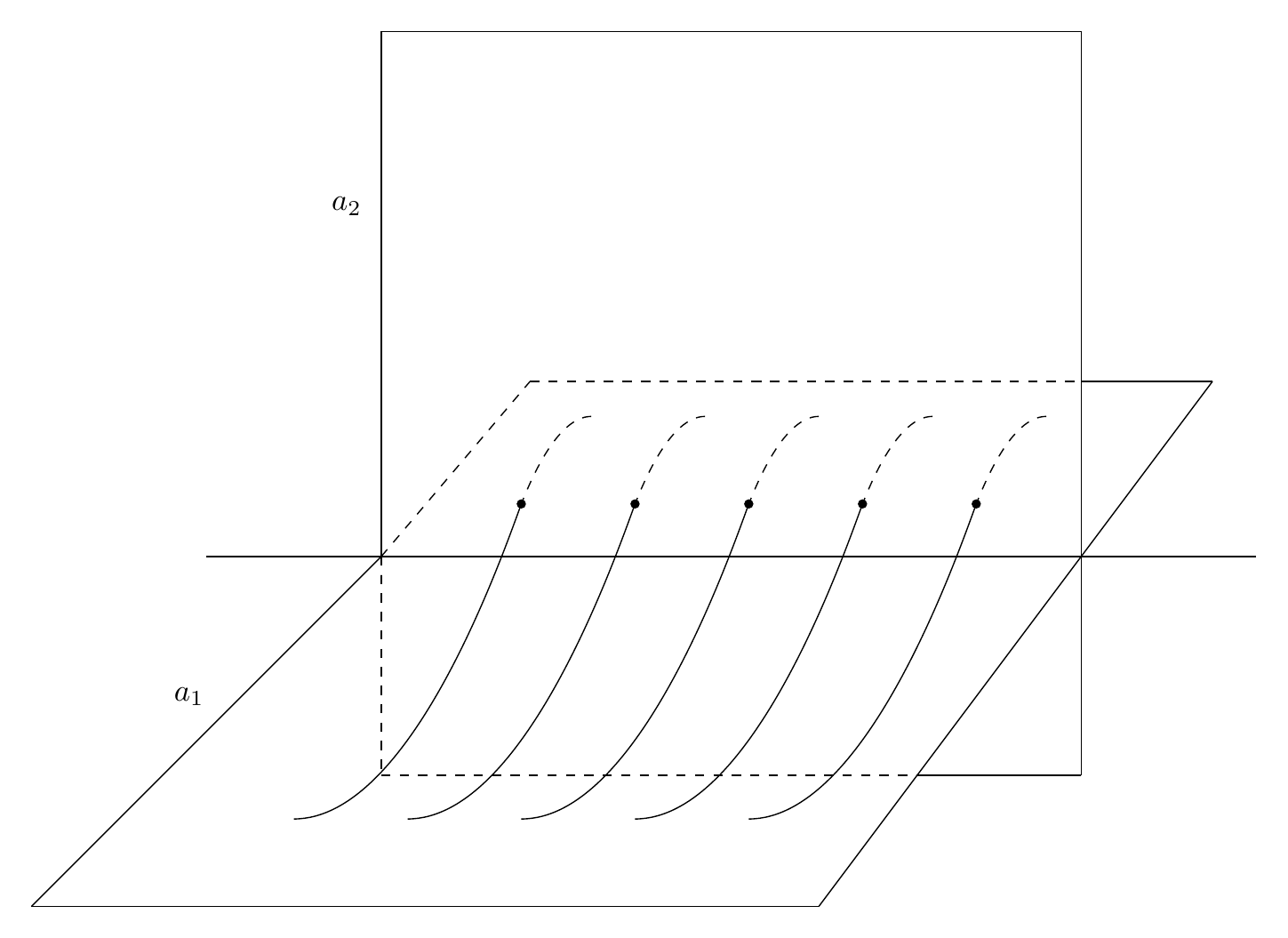}
\caption{The surfaces of tangencies and Newhouse lines}
\label{Fig8}
\end{figure}

\vskip .2 cm
\paragraph{\bf Theorem D.} There are uncountable many quadratic H\'enon-like maps with infinitely many sinks and a period doubling Cantor attractor.

\begin{proof}
The analytic functions $a_{1,2}$ from Proposition \ref{doubletg}, describing the homoclinic tangencies at $z_1$ and $z_2$, have a real-analytic extension to $[b_{-},b_{+}]\times [-\tau_0,\tau_0]$ for some $\tau_0>0$. Denote the graphs of the extensions of $a_{1,2}$ also by $a_{1,2}$. The transversality of the intersection between $a_1$ and $a_2$ at $\tau=0$ implies that the graph of $a_1$ and $a_2$ intersects in an analytic curve.
Along this curve there are two strong homoclinic tangencies. The graphs intersect transversally. We can locally reparametrize the quadratic H\'enon-like family near $(a_1(b_0),b_0,0)$ to obtain a family $F_{t, a,\tau }$ such that, the maps $F_{t,a,0}$ are the maps in the graph of $a_2$ with a strong homoclinic tangency at $z_2$ and the maps $F_{t,0,\tau}$ are the maps in the graph of $a_1$ with a strong homoclinic tangency at $z_1$. The curve $F_{t,0,0}$ consists of the maps in the intersection of the two graphs having two strong homoclinic tangencies, see Figure \ref{Fig8}. In particular $F_{0,0,0}$ is the H\'enon map with a double tangency.
\\
If $\partial\mu/\partial t\neq 0$ at $(0,0,0)$, then the family $F_{t,a,0}$ is an unfolding of the strong homoclinic tangency near the point $z_2$.  According to Theorem A the family $(t,a)\mapsto F_{t,a,0}$ contains an uncountable set $NH$ of parameters of maps with infinitely many sinks. This set accumulates at a segment of parameters $(t,0,0)$. According to Theorem B and Proposition \ref{analyticleaves}, each point $m\in NH$ is contained in an analytic curve $\gamma_m: \tau\mapsto (t(\tau),a(\tau),\tau)$. These curves are pairwise disjoint forming a lamination. Each curve is a one-parameter unfolding of the homoclinic tangency near $z_2$ with $\tau=0$.
\\
If $\partial\mu/\partial t= 0$ at $(0,0,0)$, then without loss of generality we may assume that $\partial\mu/\partial \tau\neq 0$ at $(0,0,0)$. This follows from the transversality between the curves $a_1$ and $a_2$ in the H\'enon family and the fact that this double homoclinic tangencies occur near $a=2$ and $b=0$ where $\partial\mu/\partial a\neq 0$.  By reparametrizing the $\tau,t$ coordinates we may assume that the $\tau$-axis is perpendicular to the $t$-axis. For $\epsilon>0$ very small, consider the family $G_{t,a}=F_{t,a,\epsilon\tau}$. This family is very close to the family  $F_{t,a,0}$. As consequence $G_{t,a}$ contains a curve close to the $t$-axis of maps with a tangency near $z_1$. The derivative at $(0,0)$ of the eigenvalue $\mu$ along this curve is non zero by construction. Hence the family $G_{t,a}$ is an unfolding of a tangency near $z_1$. According to Theorem A the family $(t,a)\mapsto G_{t,a}$ contains an uncountable set $NH$ of parameters of maps with infinitely many sinks. This set accumulates at the curve of tangencies near $z_1$. According to Theorem B and Proposition \ref{analyticleaves}, each point $m\in NH$ is contained in an analytic curve $\gamma_m: \tau\mapsto (t(\tau),a(\tau),\tau)$. These curves are pairwise disjoint forming a lamination. Moreover they have a definite length and they are graphs over the $\tau$ direction. These curves accumulate at $(0,0,0)$. As consequence the curves which are very close to $(0,0,0)$ intersect transversally the family $F_{t,a,0}$. In particular each curve is a one-parameter unfolding of the homoclinic tangency near $z_2$ with $\tau=0$.
\\
Consider one of the curves $\gamma_m$. For $n\ge 1$ large enough there exits, according to \cite{PT}, a rectangle $Q_n$ near $z_2$ and an interval $[\tau^0_n,\tau^1_n]$ such that the family
$$
[\tau^0_n,\tau^1_n]\ni \tau \mapsto F^n_{\gamma_m(\tau)}|Q_n,
$$
after an analytic change of coordinates becomes a one-parameter family of H\'enon-like maps which is exponentially $\Cuno$ close to the degenerated H\'enon family
$$
F_{a,0}\left(\begin{matrix}
x\\y
\end{matrix}
\right)=\left(\begin{matrix}
a-x^2\\x
\end{matrix}
\right).
$$
At $\tau^0_n$ the map is exponentially close to $F_{0,0}$, a map with a sink and at $\tau^1_n$ the map is exponentially close to $F_{3,0}$, a map with a fully developed horse shoe. According to \cite{CLM}, for $n\ge 1$ large enough,  there is a unique $\tau\in [\tau^0_n,\tau^1_n]$ such that $F^n_{\gamma_m(\tau)}|Q_n$ has a period doubling Cantor attractor. Namely, at $\tau$ the curve $
[\tau^0_n,\tau^1_n]\ni \tau \mapsto F^n_{\gamma_m(\tau)}|Q_n
$ crosses transversally a codimension one manifold of maps which have a period doubling Cantor attractor.
\end{proof}
The coexistence of sinks with a strange attractor is discussed next.

\begin{defin} \label{strange} Let $M$ be a manifold and $f:M\to M$. An open set $U\subset M$ is called a trapping region if $\overline{f(U)}\subset U$. An attractor in the sense of Conley is
$$
\Lambda=\bigcap_{j\ge 0} f^j(U).
$$
The attractor $\Lambda$ is called topologically transitive if it contains a  dense orbit. If $\Lambda$ contains a dense orbit which satisfies the Collet-Eckmann condition, i.e. there exist a point $z$, a vector $v\in T_zM$ and a constant $\kappa>0$ such that $$|Df^n(z)v|\geq e^{\kappa n} \text{ for all } n>0,$$ then $\Lambda$ is called a strange attractor.
\end{defin}

\vskip .2 cm
\paragraph{\bf Theorem E.} The set of quadratic H\'enon-like maps with infinitely many sinks and a strange attractor has Hausdorff dimension at least one.

\begin{proof} Consider the curves $\gamma_m$ introduced in the proof of Theorem D. The maps in these curves have infinitely many sinks. Moreover, they are one-parameter analytic unfoldings of a homoclinic tangency, see Proposition \ref{analyticleaves}. According to \cite{BC} and \cite{MV} each curve contains a set of positive arclength measure of maps with a strange attractor.
\end{proof}

There might be a H\'enon map which have infinitely many sinks and a period doubling Cantor attractor. However, the two-dimensional 
H\'enon family is not the natural family in which to look for such dynamics. This coexistence type of infinitely many sinks and a period doubling Cantor attractor is a codimension three phenomenon. This comes from the fact that the codimension two Newhouse leaves intersect transversally the codimension one manifold of maps with a period doubling attractor. The proof of Theorem D actually shows:

\vskip .2 cm
\paragraph{\bf Theorem F.}
The space $\text{Poly}_d({\mathbb{R}^2})$, $d\ge 2$,  of real polynomials of $\mathbb{R}^2$ of degree at most $d$ contains a codimension $3$ lamination of maps with infinitely many sinks and a period doubling Cantor attractor. The lamination is homeomorphic to $\left(\mathbb R\setminus\mathbb Q\right)\times\mathbb R^{\text{D}-3}$ where $\text{D}$ is the dimension of $\text{Poly}_d({\mathbb{R}^2})$ and the leaves of the lamination are real-analytic. The sinks and the period doubling Cantor attractor persist along the leaves. 
\vskip .2 cm

As a consequence of the proof of Theorem D we have the following:
\vskip .2 cm
\paragraph{\bf Theorem G.}
Every $d+2$-dimensional unfolding, $d\geq 1$, of a map with a strong homoclinic tangency contains smooth $d$-dimensional families of maps where each map has infinitely many sinks.
In particular, there are non trivial $d$-dimensional analytic families of polynomial H\'enon-like maps in which every map has infinitely many sinks.

\end{document}